\documentclass[12pt]{amsart}
\usepackage{epsfig}
\usepackage{amsfonts}
\usepackage{color}
\usepackage{graphicx}
\usepackage{amsmath}
\usepackage{amssymb}
\newtheorem{thm}{Theorem}[section]

\newtheorem{lem}[thm]{Lemma}
\newtheorem{prop}[thm]{Proposition}
\newtheorem*{prob*}{Problem}
\newtheorem*{thm*}{Theorem}

\theoremstyle{definition}
\newtheorem{defn}[thm]{Definition}

\newtheorem*{defn*}{Definition}
\newtheorem{rem}[thm]{Remark}

\newtheorem*{rem*}{Remark}
\numberwithin{equation}{section}

\DeclareMathOperator{\Conf}{Conf}

\DeclareMathOperator{\Prob}{Prob}

\DeclareMathOperator{\R}{\mathbb{R}}

\DeclareMathOperator{\Y}{\mathbb{Y}}
\DeclareMathOperator{\C}{\mathbb{C}}

\DeclareMathOperator{\diag}{diag}

\DeclareMathOperator{\ch}{\textbf{ch}}
\DeclareMathOperator{\DIM}{DIM}

\DeclareMathOperator{\Ewens}{Ewens}
\DeclareMathOperator{\Uniform}{Uniform}

\DeclareMathOperator{\Whittaker}{Whittaker}
\DeclareMathOperator{\Reg}{Reg}
\DeclareMathOperator{\Sym}{Sym}
\DeclareMathOperator{\Cyl}{Cyl}

\newcommand{\tLambda}{\widetilde{\Lambda}}

\begin{document}
\title[Generalized regular representations of  big wreath products]
{\bf{Generalized regular representations of  big wreath products}}

\author{Eugene Strahov}
\address{Department of Mathematics, The Hebrew University of Jerusalem, Givat Ram, Jerusalem 91904, Israel}
\email{strahov@math.huji.ac.il}
\keywords{Wreath products, the infinite symmetric group, harmonic analysis on groups, generalized regular representations,  characters, measures on partitions, central measures.\\
This work was supported by the BSF grant 2018248 ``Products of random matrices via the theory of symmetric functions''. }
\commby{}
\begin{abstract}
Let $G$ be a finite group with $k$ conjugacy classes, and $S(\infty)$ be the infinite symmetric group, i.e. the group of finite permutations of $\left\{1,2,3,\ldots\right\}$. Then the wreath product $G_{\infty}=G\sim S(\infty)$
of $G$ with $S(\infty)$  (called the big wreath product) can be defined. The group $G_{\infty}$ is a generalization of the infinite symmetric group, and it is  an example of a
``big'' group, in Vershik's terminology. For such groups  the two-sided regular
representations are irreducible, the conventional scheme of harmonic analysis
is not applicable, and the problem of harmonic analysis is a nontrivial problem with connections
to different areas of mathematics and mathematical physics.

Harmonic analysis on the infinite symmetric group
was developed in the works by Kerov, Olshanski, and Vershik, and Borodin and Olshanski. The goal of this
paper is to extend this theory to the case of
$G_{\infty}$.
In particular,  we construct  an analogue $\mathfrak{S}_{G}$ of the space of virtual permutations. We then formulate  and prove a theorem characterizing all central probability
measures on $\mathfrak{S}_{G}$. Next, we introduce generalized regular representations $\left\{T_{z_1,\ldots,z_k}:\; z_1\in\C,\ldots,z_k\in\C\right\}$ of the big wreath product $G_{\infty}$, which are
analogues of the Kerov-Olshanski-Vershik generalized regular representations
of the infinite symmetric group. We derive an explicit formula for the characters
of $T_{z_1,\ldots,z_k}$. The spectral measures of these representations are characterized in different ways. In particular, these spectral measures are associated with point processes whose correlation functions are explicitly computed. Thus, in representation-theoretic terms, the paper solves a natural problem of harmonic analysis for the big wreath products:
our results describe the decomposition of $T_{z_1,\ldots,z_k}$ into irreducible components.

\end{abstract}
\maketitle
\section{Introduction}
\subsection{Preliminaries and formulation of the problem}\label{SubsectionPreliminaries}
One of the main goals of noncommutative harmonic analysis on groups
is to  describe the decomposition of a natural representation into irreducible components. For example, if $G$ is a finite group or a compact group, and $T$ is a regular
representation of $G$, then  each irreducible
representation is contained in $T$ with multiplicity equal to its degree.
However, if $G$ is replaced by an infinite-dimensional analogue of a classical
group (such as the infinite symmetric group $S(\infty)=\underset{\longrightarrow}{\lim}\;S(n)$, or the infinite unitary group $U(\infty)=\underset{\longrightarrow}{\lim}\;U(n)$), the situation becomes much more complicated, and a deep theory
with connections to different areas of mathematics, from enumerative combinatorics to random growth models, emerges.

Harmonic analysis on the infinite symmetric group is developed in the papers by
Kerov, Olshanski and Vershik \cite{KerovOlshanskiVershik, KerovOlshanskiVershikAnnouncement},
Olshanski \cite{OlshanskiPointProcesses}, Borodin \cite{Borodin1, Borodin2}, Borodin and Olshanski
\cite{BorodinOlshanskiLetters}. The problem of harmonic analysis on
$S(\infty)$ is reformulated as that for the Gelfand pair $\left(S(\infty)\times S(\infty),\diag\left(S(\infty)\right)\right)$ (in the sense of Olshanski \cite{Olshanski}).
We then  deal with the biregular representation of the infinite symmetric group
which is defined as follows. Let $\mu$ be the counting measure on $S(\infty)$. Then the biregular representation of $S(\infty)$ is
a unitary representation  $T$ of the group $S(\infty)\times S(\infty)$ in the Hilbert space
$L^2(S(\infty),\mu)$  defined by
$$
\left(T(g_1,g_2)f\right)(x)=f(g_2^{-1}xg_1),\;\; f\in L^2(S(\infty),\mu),\;\; (g_1,g_2)\in S(\infty)\times S(\infty).
$$
The starting point of analysis in Kerov, Olshanski and Vershik \cite{KerovOlshanskiVershik, KerovOlshanskiVershikAnnouncement} is the observation that the biregular representation of the infinite symmetric group is irreducible. Thus
the conventional scheme of the harmonic analysis should be modified.
This is achieved by construction of the space of virtual permutations $\mathfrak{S}$ which is a compactification of $S(\infty)$. Then the natural action of $S(\infty)\times S(\infty)$ on $S(\infty)$ is extended to $\mathfrak{S}$. On the space $S(\infty)$ a one-parameter family of measures
$\left\{\mu_t:\; t>0\right\}$ is introduced. These measures are defined as projective limits  of
the Ewens measures on the finite symmetric groups, and have a number of remarkable  properties. These properties enable to construct a deformation $T_{z}$ of the biregular representation, which is reducible and has a rich structure. The Kerov-Olshanski-Vershik generalized regular representation $T_{z}$
is labelled by the complex parameter $z$ such that $|z|^2=t$, and acts in the Hilbert space
$L^2\left(\mathfrak{S},\mu_t\right)$.

Clearly, a usual definition of a representation character is not applicable in the case of
the representation $T_{z}$. However, the character $\chi_z$ of $T_z$ can be introduced
using the language of spherical representations, and of associated spherical functions.
Denote by $\mathbf{1}$ the function on $\mathfrak{S}$ identically equal to $1$. It can be viewed as a vector from $L^2\left(\mathfrak{S},\mu_t\right)$ which is invariant with respect to the action of $\diag\left(S\left(\infty\right)\right)$. Thus $\left(T_z,\mathbf{1}\right)$ can be understood as a spherical representation of the Gelfand pair $\left(S(\infty)\times S(\infty),\diag\left(S(\infty)\right)\right)$.  The spherical function $\varphi_z$ of $\left(T_z,\mathbf{1}\right)$
is the matrix element
\begin{equation}
\varphi_z\left(g_1,g_2\right)=\left<T_z\left(g_1,g_2\right)\mathbf{1},\mathbf{1}\right>_{L^2\left(\mathfrak{S},\mu_t\right)},\;\;\left(g_1,g_2\right)\in S(\infty)\times S(\infty).
\nonumber
\end{equation}
A complex-valued function $\chi$ on $S(\infty)$ is called a character of $S(\infty)$ if it is positive definite, central, and normalized to take value $1$
at the unit element of $S(\infty)$.
There is a one-to-one correspondence $\varphi\longleftrightarrow\chi$ between the set
 of spherical functions associated with the spherical representations of the Gelfand pair
$\left(S(\infty)\times S(\infty),\diag\left(S(\infty)\right)\right)$,
and the set
$\mathcal{X}\left(S(\infty)\right)$ of characters of $S(\infty)$. In particular, the spherical function $\varphi_z$ corresponds to a character $\chi_z$,  $\chi_z\in\mathcal{X}\left(S(\infty)\right)$, and
the relation between $\varphi_z$ and $\chi_z$ is
\begin{equation}\label{varchicor}
\chi_z\left(g\right)=\varphi_z\left(g,e\right),\;\;g\in S(\infty).
\end{equation}
The function $\chi_z$ on $S(\infty)$ defined by equation (\ref{varchicor}) is called the character of $T_z$.

 Kerov, Olshanski and Vershik
\cite{KerovOlshanskiVershik, KerovOlshanskiVershikAnnouncement} found the restriction of $\chi_z$ to $S(n)$ in terms of irreducible characters of $S(n)$. Namely, let $\Y_n$ be the set of Young diagrams with $n$ boxes. For $\lambda\in\Y_n$ denote by $\chi^{\lambda}$ the corresponding normalized irreducible character
of the symmetric group $S(n)$. Then for any $n=1,2,\ldots $ the following formula holds true
\begin{equation}
\chi_z|_{S(n)}=\sum\limits_{\lambda\in\Y_n}M_z^{(n)}(\lambda)\chi^{\lambda}.
\nonumber
\end{equation}
The coefficient $M_z^{(n)}$ is a probability measure (called the $z$-measure) on the set $\Y_n$ of Young diagrams with $n$ boxes, and there is an explicit formula for $M_{z}^{(n)}$. The $z$-measures
are interesting objects by themselves, and are studied in many papers, see, for example,
Borodin and Olshanski \cite{BorodinOlshanskiRSK, BorodinOlshanskiKernel}, Okounkov \cite{Okounkov},
Borodin, Olshanski, and Strahov \cite{BorodinOlshanskiStrahov}.

As any character of $S(\infty)$, the character $\chi_z$ can be represented
in terms of the extreme characters, namely
\begin{equation}
\chi_z\left(g\right)=\int\limits_{\Omega}\chi^{(\omega)}(g)P_{z}(d\omega).
\nonumber
\end{equation}
Here $\Omega$ is the Thoma set,
\begin{equation}\label{ThomaSet}
\Omega=\left\{\alpha_1\geq\alpha_2\geq\ldots\geq 0;\;\beta_1\geq\beta_2\geq\ldots\geq 0:\;\;
\sum\limits_{i=1}^{\infty}\left(\alpha_1+\beta_i\right)\leq 1\right\},
\nonumber
\end{equation}
$\chi^{(\omega)}$ are the extreme characters of  $S(\infty)$
parameterized by points $\omega$ of $\Omega$, and $P_{z}$ is a probability measure on $\Omega$  called the spectral measure of the Kerov-Olshanski-Vershik generalized representation $T_z$.

The extreme characters $\chi^{(\omega)}$ are given explicitly by the Thoma theorem \cite{Thoma}. One of the problems of the harmonic analysis on the infinite symmetric group is to describe the probability measure $P_{z}$. The solution of this problem is obtained in the papers by Olshanski \cite{OlshanskiPointProcesses}, Borodin \cite{Borodin1, Borodin2}, Borodin and Olshanski \cite{BorodinOlshanskiLetters}, where the measure $P_{z}$ is interpreted as a point process
$\mathcal{P}_{z}$ on the punctured interval $I^{\ast}=[-1,1]\setminus\{0\}$.
Borodin and Olshanski  show that a certain modification (``lifting") of
$\mathcal{P}_{z}$ is a determinantal point process whose correlation functions can be explicitly computed.

The goal of the present paper is to extend the results mentioned above to the case of the wreath product $G\sim S(\infty)$ of a finite group $G$ with the infinite symmetric group $S(\infty)$. If $G_{\infty}=G\sim S(\infty)$, then we are dealing with the Gelfand pair $\left(G_{\infty}\times G_{\infty},\diag\left(G_{\infty}\right)\right)$, with its spherical representations  and the spherical functions. As in the case of the infinite symmetric group, the biregular representation of $G_{\infty}\times G_{\infty}$ is irreducible, and the standard scheme of the harmonic analysis should be modified.

Below we give a summary of main results obtained in this paper.
\subsection{Summary of results}
\subsubsection{The space of $G$-virtual permutations $\mathfrak{S}_{G}$}
Let $G$ be a finite group with $k$ conjugacy classes, and let $G\sim S(n)$ be the wreath product of $G$ with the symmetric group $S(n)$. On $G\sim S(n)$ a probability measure
$P^{\Ewens}_{t_1,\ldots,t_k;n}$ can be introduced, which depends on $k$ strictly positive parameters $t_1$, $\ldots$, $t_k$. The measure $P^{\Ewens}_{t_1,\ldots,t_k;n}$ is a generalization of the Ewens probability measure on the symmetric group.

For any $n\geq 1$ we define a projection $p_{n,n+1}: G\sim S(n+1)\longrightarrow G\sim S(n)$, which is equivariant with respect to  the two-sided action of $G\sim S(n)$. Then we define the space $\mathfrak{S}_{G}$ (called the space of $G$-virtual permutations in the paper) as
the projective limit of the finite sets $G\sim S(n)$ taken with respect to $p_{n,n+1}$.

To ensure a reasonable definition of the generalized regular representations of the big wreath products, the projection $p_{n,n+1}$ is required to satisfy several  conditions.
The construction of such a projection is a non-trivial task, and it is one of the achievements of the present paper.
\subsubsection{Central measures}
A remarkable property of $p_{n,n+1}$ is that the Ewens probability measures
$P^{\Ewens}_{t_1,\ldots,t_k;n}$ are pairwise consistent with respect to $p_{n,n+1}$.
This property makes it possible to define, for any $t_1>0$, $\ldots$, $t_k>0$, a probability
measure $P^{\Ewens}_{t_1,\ldots,t_k}$ on the space $\mathfrak{S}_{G}$
as the projective limit, $P^{\Ewens}_{t_1,\ldots,t_k}=\underset{\longleftarrow}{\lim}
P^{\Ewens}_{t_1,\ldots,t_k;n}$.
This probability measure, $P^{\Ewens}_{t_1,\ldots,t_k}$, is central, i.e. it is invariant under the conjugations by $G\sim S(\infty)$.  A non-trivial problem is to describe all
central measures on $\mathfrak{S}_{G}$, and our Theorem  \ref{THEOREMCENTRALMEASURES} gives the solution of this problem. Namely, Theorem  \ref{THEOREMCENTRALMEASURES} establishes a one-to-one correspondence between central measures on $\mathfrak{S}_{G}$, and arbitrary probability measures on some subspace $\overline{\nabla}^{(k)}$ of $\underset{k}{\underbrace{[0,1]^{\infty}\times\ldots\times[0,1]^{\infty}}}$. In particular, $P^{\Ewens}_{t_1,\ldots,t_k}$ turns into the \textit{multiple Poisson-Dirichlet distribution} under this correspondence, see Section \ref{CorPPD}.

Note that there is a one-to-one correspondence between central probability measures on the wreath product $G\sim S(n)$, and probability measures on the set $\Y_n^{(k)}$ of multiple partitions of $n$ into $k$ components (the elements of $\Y_n^{(k)}$ parameterize the conjugacy classes of $G\sim S(n)$). Our Theorem \ref{THEOREMCENTRALMEASURES} can be viewed as a nontrivial infinite-dimensional analogue of this correspondence.
\subsubsection{The generalized regular representation of $G_{\infty}=G\sim S(\infty)$}
We show that the probability measure $P^{\Ewens}_{t_1,\ldots,t_k}$ on the space $\mathfrak{S}_{G}$ is quasiinvariant with respect to the action of
$G_{\infty}\times G_{\infty}$ on $\mathfrak{S}_{G}$. This enables us to construct
in Section \ref{SectionGeneralizedRepresentations} an analogue $T_{z_1,\ldots,z_k}$ of the Kerov-Olshanski-Vershik generalized regular representation $T_z$.
The parameters $z_1\in\C\setminus\{0\}$, $\ldots$,  $z_k\in\C\setminus\{0\}$
of $T_{z_1,\ldots,z_k}$ are related with the parameters $t_1$, $\ldots$, $t_k$
of $P_{t_1,\ldots,t_k}^{\Ewens}$ as
$|z_1|^2=t_1$, $\ldots$, $|z_k|^2=t_k$. We show (see Theorem \ref{THEOREM6.3}) that $T_{z_1,\ldots,z_k}$
is equivalent to the inductive limit of the two-sided regular representations of $G_n\times G_{n}$, $G_n=G\sim S(n)$.
\subsubsection{The formula for the character $\chi_{z_1,\ldots,z_k}$ of $T_{z_1,\ldots,z_k}$}.
Let $\left(T_{z_1,\ldots,z_k},\zeta_{0}\right)$ be the spherical representation of the Gelfand pair  $\left(G_{\infty}\times G_{\infty},\diag\left(G_{\infty}\right)\right)$.
Denote by $\Phi_{z_1,\ldots,z_k}$ the spherical function of $\left(G_{\infty}\times G_{\infty},\diag\left(G_{\infty}\right)\right)$. The character $\chi_{z_1,\ldots,z_k}$ of $T_{z_1,\ldots,z_k}$ is defined in terms of $\Phi_{z_1,\ldots,z_k}$ as
$$
\chi_{z_1,\ldots,z_k}(x)=\Phi_{z_1,\ldots,z_k}(x,e),\; x\in G_{\infty},
$$
where $e$ is the unit element of $G_{\infty}$. The function $\chi_{z_1,\ldots,z_k}$ is a character of $G_{\infty}$ (in the sense of Definition \ref{characters7} below).  Theorem \ref{TheoremMultpleZmeasures} of the present paper gives a formula  for $\chi_{z_1,\ldots,z_k}$:
each restriction of $\chi_{z_1,\ldots,z_k}$ to $G_n=G\sim S(n)$ is represented as a linear combination of the normalized irreducible characters of $G_n$, and the coefficients of this linear combination are explicitly computed. These coefficients, $M_{z_1,\ldots,z_k}^{(n)}$, are probability measures on the set $\Y_n^{(k)}$ of multiple partitions of $n$ into $k$ components, and can be understood as generalizations of the $z$-measures $M^{(n)}_{z}$ mentioned in Section \ref{SubsectionPreliminaries}. An  explicit formula for $M_{z_1,\ldots,z_k}^{(n)}$ is derived in this paper, see equation (\ref{MultipleZmeasures}).
\subsubsection{The spectral measures}
The characters $\chi_{z_1,\ldots,z_k}$ admit the following integral representation
\begin{equation}\label{1.2.4.1}
\chi_{z_1,\ldots,z_k}(x)=\int\limits_{\Omega\left(G\right)}f_{\omega}(x)P_{z_1,\ldots,z_k}(d\omega),\;\; x\in G\sim S(\infty).
\end{equation}
Here $\Omega\left(G\right)$ is the generalized Thoma set, $f_{\omega}$ is the extreme character of $G\sim S(\infty)$, and  $P_{z_1,\ldots,z_k}$ is a probability measure on the set $\Omega\left(G\right)$ (called the spectral measure of
$\chi_{z_1,\ldots,z_k}$). Equation (\ref{1.2.4.1}) is a consequence of Theorem 3.5 in Hora and Hirai \cite{HoraHirai} which gives an integral representation for any character of $G\sim S(\infty)$. Hora and Hirai
\cite{HoraHirai} provides explicit formulae for both $\Omega\left(G\right)$ and $f_{\omega}$, see Theorem 2.5 and Theorem 3.4 in Ref. \cite{HoraHirai}.

The problem
addressed in the present paper is to describe the spectral measures
$P_{z_1,\ldots,z_k}$.  In representation-theoretic terms, this is equivalent to
description of the decomposition of $T_{z_1,\ldots,z_k}$ into irreducible components, which is
 a natural problem of harmonic analysis for the big wreath product $G_{\infty}$.
 A solution is given by our Theorem \ref{THEOREMDESCRIPTIONOFPz1z2zk}
where the  spectral measures
$P_{z_1,\ldots,z_k}$ are described in terms of the spectral measures
$P_{z_1}$, $\ldots$, $P_{z_k}$ of the Kerov-Olshanski-Vershik generalized representations $T_{z_1}$, $\ldots$, $T_{z_k}$.
\subsubsection{Correlation functions}
In Section \ref{SectionPointProcess11} we convert the measure $P_{z_1,\ldots,z_k}$ into a point process $\mathcal{P}_{z_1,\ldots,z_k}$.
Our Theorem \ref{MAINTHEOREMCORRELATIONS} gives the correlation functions for a lifted version
$\widetilde{\mathcal{P}}_{z_1,\ldots,z_k}$ of $\mathcal{P}_{z_1,\ldots,z_k}$ in terms of the known correlation functions of the Whittaker determinantal process.
\subsection{Remarks on related works}
\subsubsection{}
There are many works devoted to representation theory  of infinite analogues of classical groups, and to related questions of harmonic analysis. In particular,
the books by Kerov \cite{KerovDissertation}, Borodin and Olshanski \cite{BorodinOlshanskiBook}, and the survey paper by Olshanski \cite{OlshanskiNato} are basic references on the subject, and provide an introduction to this field of research. Besides,
the paper by Borodin and Olshanski \cite{BorodinOlshanskiUnitary} solves the problem of harmonic analysis on the infinite unitary group.
Borodin and Olshanski \cite{BorodinOlshanskiErgodic}, Olshanski \cite{OlshanskiUnitary} deal with
$\left(U(\infty),\mathfrak{U}\right)$ instead of  $\left(S(\infty),\mathfrak{S}\right)$ (where
$\mathfrak{U}$ is a certain analogue of the space of virtual permutations $\mathfrak{S}$).
The authors construct a distinguished family of invariant measures on $\mathfrak{U}$, and study the decomposition of these measures on ergodic components in terms of determinantal point processes. In Gorin \cite{Gorin}, Gorin and Olshanski \cite{GorinOlshanski}, Cuenca and Gorin \cite{CuencaGorin}  a quantization of the harmonic analysis on the infinite-dimensional unitary, symplectic, and orthogonal groups is considered, and $q$-deformed versions of characters are classified. The  papers by  Gorin, Kerov, and  Vershik \cite{GorinKerovVershik}, Cuenca and Olshanski \cite{CuencaOlshanski1}
are devoted to characters and  representations of the group of infinite matrices over a finite field.

\subsubsection{}
The study of representation theory of wreath products with the infinite symmetric group begins in the works by Boyer \cite{Boyer}, Hirai,  Hirai and Hora \cite{HiraiHiraiHoraI},  Hora, Hirai and Hirai \cite{HoraHiraiHiraiII}.
In particular, in papers \cite{HiraiHiraiHoraI,HoraHiraiHiraiII}
the authors investigate asymptotic behaviour of characters of $G\sim S(n)$ as $n\rightarrow\infty$, and analyze its connection with the characters of $G\sim S(\infty)$. Paper by Hora and Hirai \cite{HoraHirai} studies harmonic functions on the branching graph $\Gamma\left(G\right)$ of the inductive system of $G\sim S(n)$'s, and derive Martin integral expressions for such functions. The Martin integral representation for harmonic functions on a Jack deformation $\Gamma_{\theta}\left(G\right)$ of $\Gamma\left(G\right)$ is derived in Strahov \cite{StrahovMPS}.

\subsubsection{} In the present paper we are dealing with the Gelfand pair
$\left(G_{\infty}\times G_{\infty},\diag\left(G_{\infty}\right)\right)$,
where $G_{\infty}$ is the big wreath product, $G_{\infty}=G\sim S(\infty)$.
Similar results can be obtained for other Gelfand pairs
constructed from the infinite symmetric group and its analogues.
For example, let $S(2n)$ be  the group of permutations of the set $\left\{-n,\ldots,-1,1,\ldots,n\right\}$, and let $H(n)$ be its subgroup defined as the centralizer of the product of transpositions $(-n,n)$, $(-n+1,n-1)$,
 $\ldots$, $(-1,1)$. It is known that $\left(S(2n),H(n)\right)$ is a Gelfand pair,
and that its inductive limit, $\left(S(2\infty), H(\infty)\right)$ , is a Gelfand pair in the sense of Olshanski \cite{Olshanski}.

Paper by Strahov \cite{StrahovSH}
 describes the construction of a family of spherical representations
 $T_{z,\frac{1}{2}}$, and shows that the $z$-measures with the Jack parameter $\theta=\frac{1}{2}$ is a coefficient in decomposition of the spherical functions of $T_{z,\frac{1}{2}}$ into irreducible components. The $z$-measures with the Jack parameters $\theta=\frac{1}{2}$, $2$ are studied in Borodin and Strahov \cite{BorodinStrahov}, Strahov \cite{StrahovMatrixKernels, StrahovSH, StrahovPfaffian}.

 \subsubsection{}Our Theorem \ref{THEOREMCENTRALMEASURES} describes all central measures on the space of $\mathfrak{S}_{G}$ of $G$-virtual permutations. Theorem \ref{THEOREMCENTRALMEASURES} can be understood as a generalization of Theorem 25.2.1 in Olshanski \cite{OlshanskiRandomPermutations}.
Theorem 25.2.1 in Ref. \cite{OlshanskiRandomPermutations} is a reformulation of the celebrated Kingman theorem \cite{Kingman2} on certain sequences of probability measures on partitions called partition structures. Our Theorem \ref{THEOREMCENTRALMEASURES} is closely related to Theorem 2.2 in Strahov \cite{StrahovMPS} on multiple partition structures.
\section{The space $\mathfrak{S}_{G}$ of G-virtual permutations}
\subsection{The wreath product $G_n=G\sim S(n)$}\label{SectionFiniteWreathProduct}
The material of this section is standard, see Macdonald \cite{Macdonald}, Appendix B.
Let $G$ be a finite group. Denote by $G_{\ast}$ the set of conjugacy classes in $G$.
We  assume that $G_{\ast}$ consists of $k$ conjugacy classes labelled $G_{\ast}=\left\{c_1,\ldots,c_k\right\}\footnote{The labeling of the conjugacy classes of $G$ plays a role in the theory.}$.
Let $S(n)$ be the symmetric group of degree $n$, i.e. the group of permutations of the finite set
$\left\{1,\ldots,n\right\}$.
The wreath product $G\sim S(n)$ is the group whose underlying set is
$$
G^n\times S(n)=\left\{\left(\left(g_1,\ldots,g_n\right),s\right):\;g_i\in G, s\in S(n)\right\}.
$$
The multiplication in $G\sim S(n)$ is defined by
$$
\left((g_1,\ldots,g_n),s\right)\left((h_1,\ldots,h_n),t\right)=
\left((g_1h_{s^{-1}(1)},\ldots,g_nh_{s^{-1}(n)}),st\right).
$$
When $n=1$, $G\sim S(1)$ is $G$. The number of elements in $G\sim S(n)$ is equal to $|G|^nn!$.
The elements of $G\sim S(n)$ can be thought of as permutation matrices with entries in $G$. Namely, the element
$
\left(\left(g_1,\ldots,g_n\right),s\right)
$
can be represented as
$
\left(g_i\delta_{i,s(j)}\right)_{i,j=1}^n.
$
In addition, the elements of $G\sim S(n)$ can be identified with bipartite graphs.
Namely, we associate with an element $\left(\left(g_1,\ldots,g_n\right),s\right)$  a graph with the vertex set $\left\{1,\ldots,n;g_1,\ldots,g_n\right\}$.
Its edges are couples of the form $\left(i,g_j\right)$, where $s(i)=j$, see Fig. \ref{Fig.1},
and we refer to $g_j$ as  the weight of the edge $\left(i,g_j\right)$.
 \begin{figure}[t]
{\scalebox{0.5}{\includegraphics{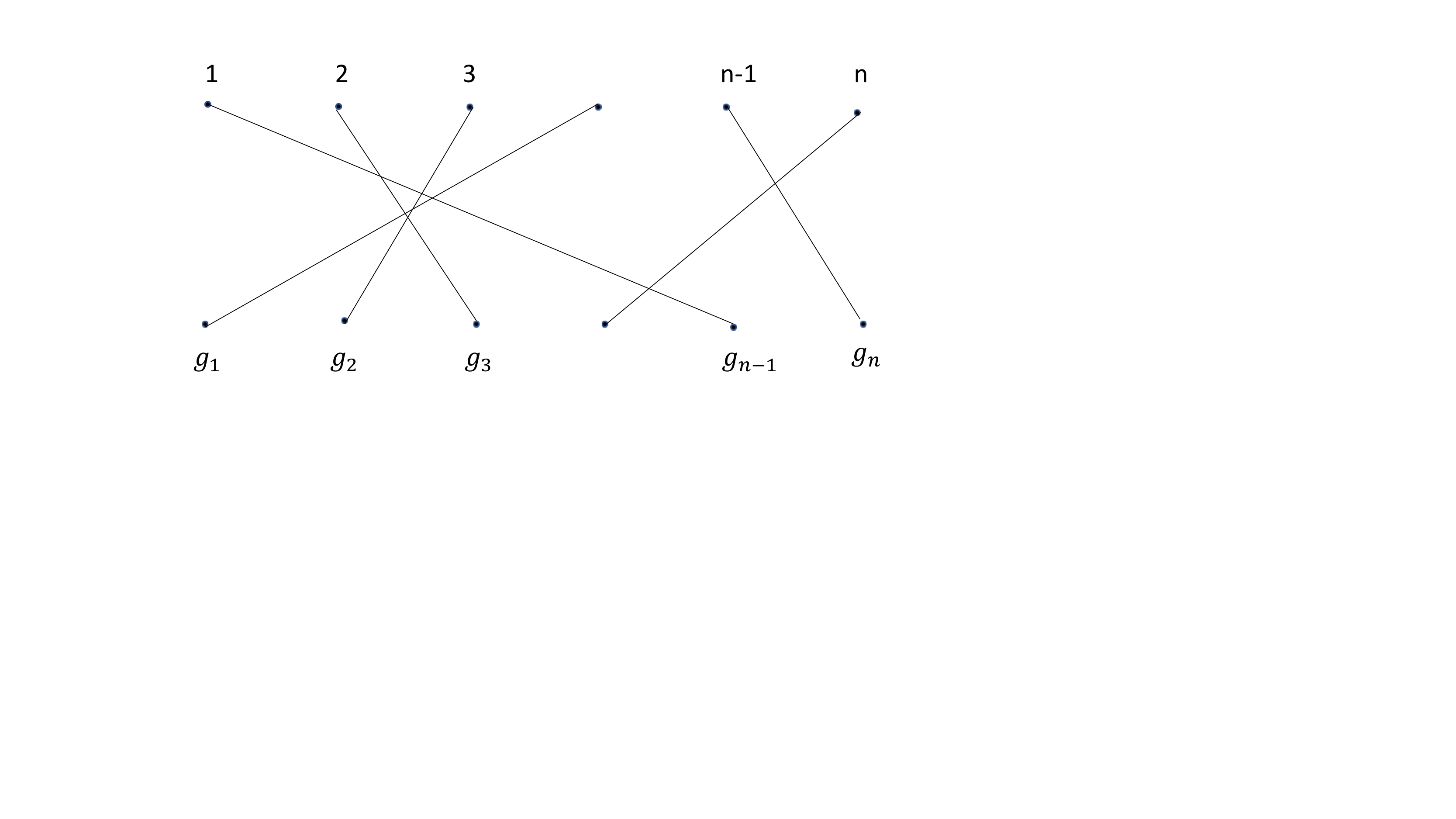}}}
\caption{An element $\left(\left(g_1,\ldots,g_n\right),s\right)$ as a bipartite graph. The symbols $g_1$, $\ldots$, $g_n$ can be understood as the weights of the corresponding edges.}
\label{Fig.1}
\end{figure}

The bipartite graphs can be used to illustrate  multiplication of two elements
$\left(\left(g_1,\ldots,g_n\right),s\right)$ and $\left(\left(h_1,\ldots,h_n\right),t\right)$
of $G\sim S(n)$, see Fig. \ref{Fig.2}.  The list $\left(g_1h_{s^{-1}(1)},\ldots,g_nh_{s^{-1}(n)}\right)$ is obtained by  reading of the weights of the corresponding edges in the direction of the arrows. The element $st$ of $S(n)$ is obtained as in the usual graphical representation of multiplication of two elements of the symmetric group $S(n)$.
 \begin{figure}[t]
{\scalebox{0.5}{\includegraphics{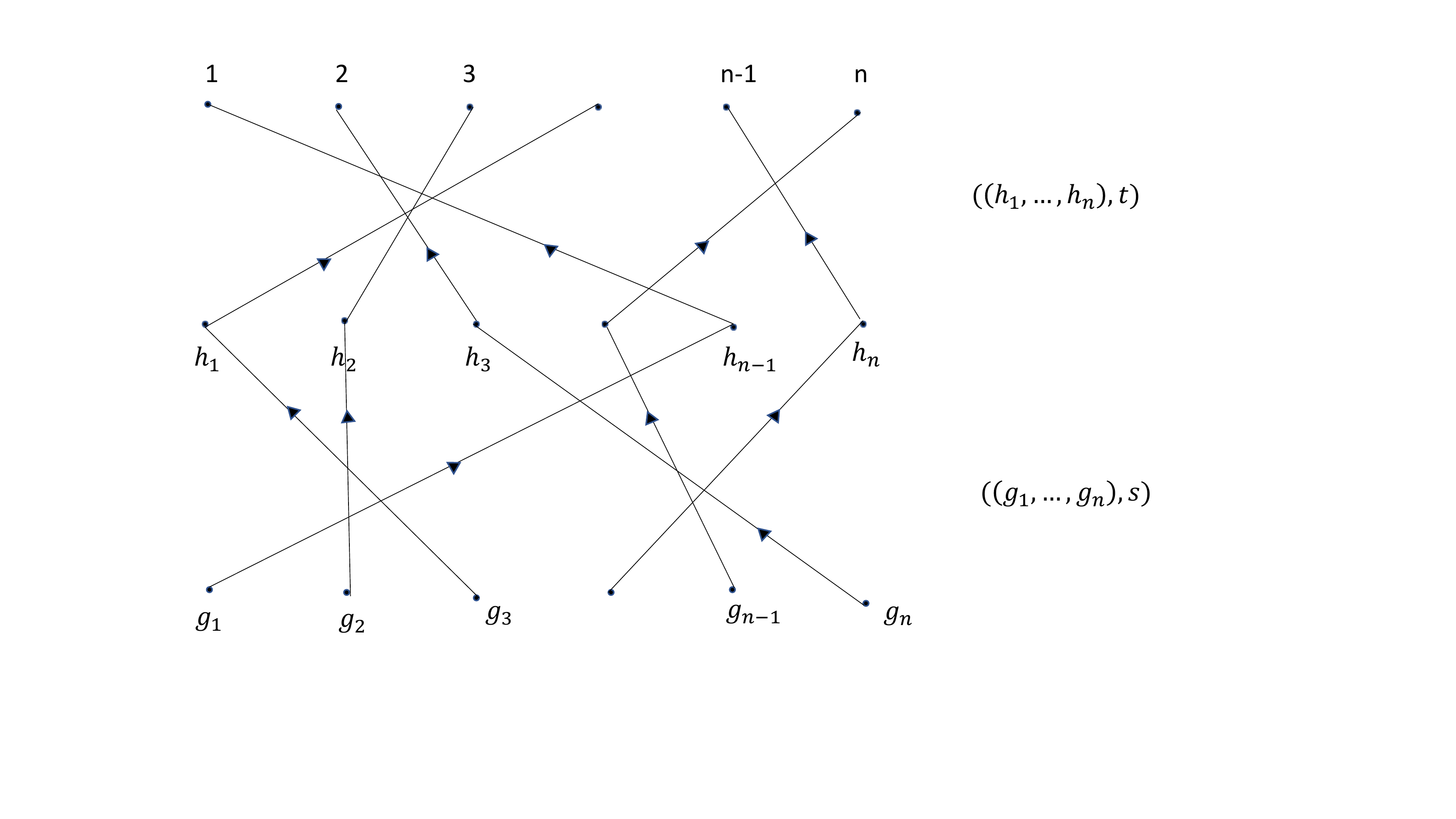}}}
\caption{The multiplication of two group elements in terms of bipartite graphs.}
\label{Fig.2}
\end{figure}

Let
$$
x=\left(\left(g_1,\ldots,g_n\right),s\right)\in G\sim S(n).
$$
The permutation $s$ can be written as a product of disjoint cycles. If
$(i_1,i_2,\ldots,i_r)$ is one of these cycles, then the element
$g_{i_{r}}g_{i_{r-1}}\ldots g_{i_1}$ is called \textit{the cycle-product of $x$
corresponding to $(i_1,i_2,\ldots,i_r)$}.
A cycle $(i_1,\ldots,i_r)$  of $s$ is called \textit{of type $c_l$}
if the corresponding cycle-product of $x$  belongs to $c_l$.
If $c_l$ is a conjugacy class in $G$, then
we denote by $[x]_{c_l}$
the number of cycles of $s$ whose cycle-product of $x$ belongs to $c_l$.

It is well known that both the conjugacy classes and the irreducible representations of $G\sim S(n)$ are parameterized by \textit{multiple partitions} $\Lambda_n^{(k)}=\left(\lambda^{(1)},\ldots,\lambda^{(k)}\right)$,
where $k$ is the number of conjugacy classes in $G$, and where $|\lambda^{(1)}|+\ldots+|\lambda^{(k)}|=n$. In particular, if
$x=\left(\left(g_1,\ldots,g_n\right),s\right)$ belongs to the conjugacy class
$K_{\Lambda_n^{(k)}}$  of $G\sim S(n)$  parameterized by $\Lambda_n^{(k)}=\left(\lambda^{(1)},\ldots,\lambda^{(k)}\right)$, then
$$
\lambda^{(l)}=\left(1^{m_1^{(l)}(s)}2^{m_2^{(l)}(s)}\ldots n^{m_n^{(l)}(s)}\right),
$$
where $m_j^{(l)}(s)$ is equal to the number of $j$-cycles of type $c_l$ in $s$.

The number of elements in the conjugacy class $K_{\Lambda^{(k)}_n}$ of $G\sim S(n)$ parameterized by
$\Lambda_n^{(k)}=\left(\lambda^{(1)},\ldots,\lambda^{(k)}\right)$ is given by
\begin{equation}\label{TheSizeOfTheConjugacyClass}
\left|K_{\Lambda_n^{(k)}}\right|=\frac{n!|G|^n}{\prod\limits_{l=1}^k\prod\limits_{j=1}^nj^{r_j\left(\lambda^{(l)}\right)}
\left(r_j\left(\lambda^{(l)}\right)\right)!}\;\frac{1}{\prod\limits_{l=1}^k\zeta_{c_l}^{r_1\left(\lambda^{(l)}\right)+\ldots+r_n\left(\lambda^{(l)}\right)}},
\end{equation}
where $r_j\left(\lambda^{(l)}\right)$ denotes the number of rows of length $j$ in the Young diagram $\lambda^{(l)}$
(in particular, the sum $r_1\left(\lambda^{(l)}\right)+\ldots+r_n\left(\lambda^{(l)}\right)$ is equal to the total number of rows in $\lambda^{(l)}$), and
\begin{equation}\label{zetal}
\zeta_{c_l}=\frac{|G|}{|c_l|}.
\end{equation}
\subsection{The canonical projection}\label{SectionProjection}
Here we define the canonical projection,
$$
p_{n,n+1}: G\sim S(n+1)\longrightarrow G\sim S(n).
$$
Let $\tilde{x}=\left(\left(g_1,\ldots,g_{n+1}\right),\tilde{s}\right)$ be an element of $G\sim S(n+1)$.
Represent $\tilde{s}$ in terms of cycles. If $n+1$ is a fixed point of $\tilde{s}$, then we set
$s=\tilde{s}$, and $p_{n,n+1}\left(\tilde{x}\right)=\left(\left(g_1,\ldots,g_{n}\right), s\right)$.
If $n+1$ belongs to a cycle,
\begin{equation}\label{Cycle}
i_1\rightarrow\ldots\rightarrow i_m\rightarrow n+1\rightarrow i_{m+1}\rightarrow\ldots\rightarrow i_r,
\end{equation}
then we remove $n+1$ out of the cycle, and replace
$$
\tilde{g}=\left(g_1,\ldots,g_n,g_{n+1}\right)
$$
by
$$
g=\left(g_1,\ldots, g_{i_{m+1}}g_{n+1},\ldots,g_n\right).
$$
Thus $g$ is obtained from $\tilde{g}$ by  removing the $n+1$th element $g_{n+1}$ from $\tilde{g}$, and by
replacing the element $g_{i_{m+1}}$ of $\tilde{g}$ by $g_{i_{m+1}}g_{n+1}$. Note that the cycle-product of $\tilde{x}=\left(\left(g_1,\ldots,g_{n+1}\right),\tilde{s}\right)$ corresponding to the cycle (\ref{Cycle})
is $g_{i_r}\ldots g_{i_{m+1}}g_{n+1}g_{i_m}\ldots g_{i_1}$, and it is  the same as the cycle-product of the obtained element  $x$, $x=p_{n,n+1}\left(\tilde{x}\right)$, corresponding to the cycle
$$
i_1\rightarrow\ldots\rightarrow i_m\rightarrow i_{m+1}\rightarrow\ldots\rightarrow i_r.
$$
We conclude that if $x=\left(\left(g_1,\ldots,g_n\right),s\right)\in G_n$, and
$\tilde{x}=\left(\left(g_1,\ldots,g_{n+1}\right),\tilde{s}\right)\in G_{n+1}$ are such that $p_{n,n+1}(\tilde{x})=x$, then each cycle of $s$ is of the same type as that
of the corresponding cycle of $\tilde{s}$. In other words, the projection $p_{n,n+1}$ preserves the types of the cycles.

The projection $p_{n,n+1}: G\sim S(n+1)\longrightarrow G\sim S(n)$ can be understood as an operation on the bipartite graph representing
an element of $G\sim S(n+1)$. Namely, in order to describe the action of $p_{n,n+1}$ on $\left(\left(g_1,\ldots,g_{n+1}\right),\tilde{s}\right)$ take the graph of $\left(\left(g_1,\ldots,g_{n+1}\right),\tilde{s}\right)$ and
add an extra edge connecting the vertex $g_{n+1}$ and $n+1$.
If $\tilde{s}\in S(n+1)$ includes the cycle which can be written as (\ref{Cycle}), then the graph of  $p_{n,n+1}\left(\left(g_1,\ldots,g_{n+1}\right),\tilde{s}\right)$ is that whose edge coming from $g_{i_{m+1}}$ goes to $i_m$, and pass
through $g_{n+1}$. Thus we can say that the weight of this edge is equal to $g_{i_{m+1}}g_{n+1}$, see Fig. 3 for a specific example. If $n+1$ is a fixed point of $\tilde{s}$, we remove the edge connecting $g_{n+1}$ with $n+1$ from the graph.
\begin{figure}[t]
{\scalebox{0.5}{\includegraphics{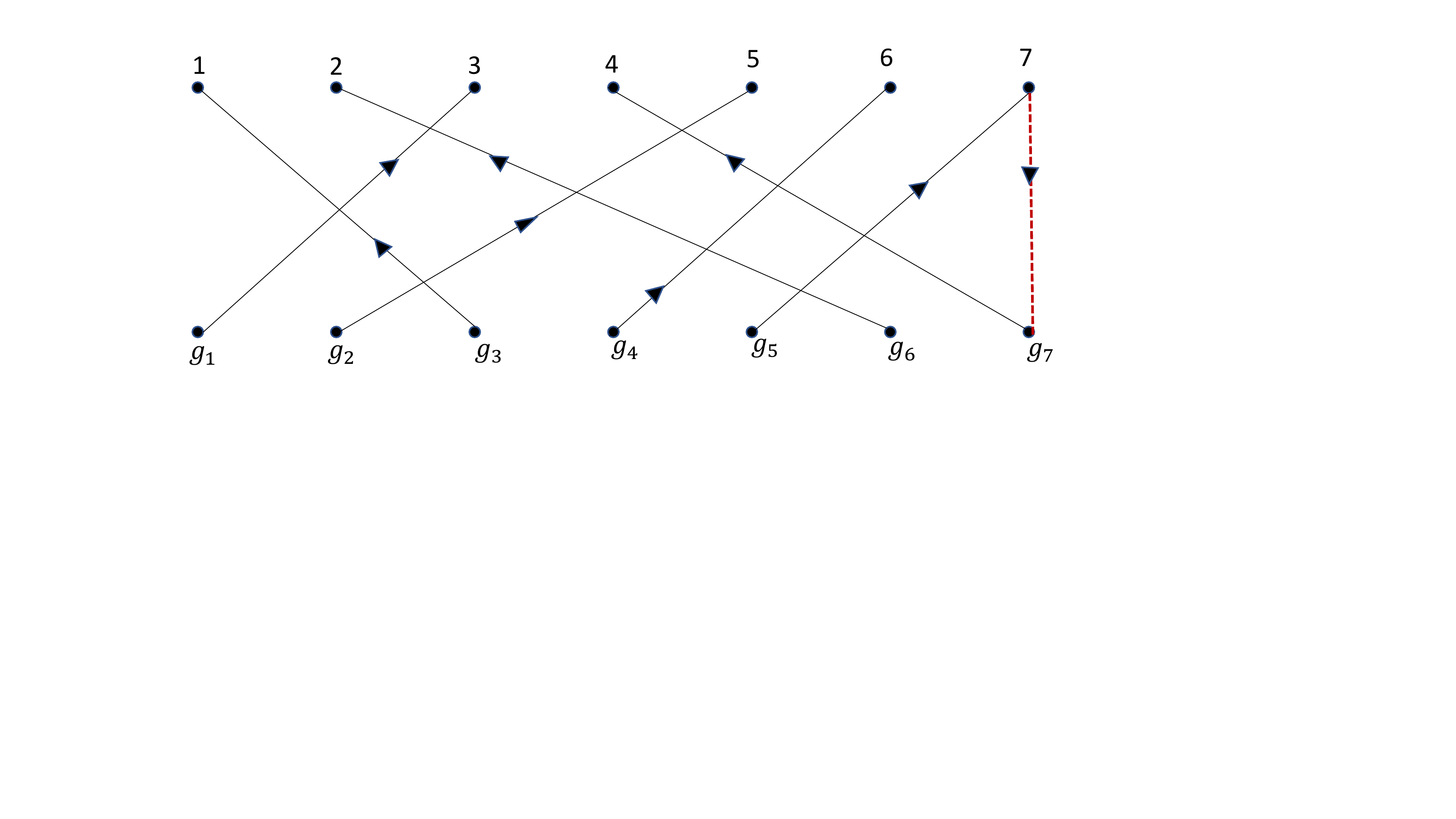}}}
\caption{The definition of the canonical projection $p_{n,n+1}$. In this example $n=6$, and  the original element of $G\sim S(7)$ is
$\left(\left(g_1,g_2,g_3,g_4,g_5,g_6,g_7\right),(13)(26475)\right)$. The cycle including $n+1=7$ is
$2\rightarrow 6\rightarrow 4\rightarrow 7\rightarrow 5$, and $g_{n+1}=g_7$, $g_{i_m}=g_4$, $g_{i_{m+1}}=g_5$.
We add the extra edge (the red dashed line) connecting the vertices $7$ and $g_7$. As a result we obtain a graph with an edge connecting $g_5$ with $4$, and passing through $g_7$. The weight of this edge is $g_5g_7$. Thus we have
$$
p_{6,7}\left(\left(\left(g_1,g_2,g_3,g_4,g_5,g_6,g_7\right),(13)(26475)\right)\right)
=\left(\left(g_1,g_2,g_3,g_4,g_5g_7,g_6\right),(13)(2645)\right).
$$
}
\label{Fig.3}
\end{figure}
\begin{prop}\label{PROPOSITIONEQUIVARIANCEPROJECTION}The projection $p_{n,n+1}: G\sim S(n+1)\longrightarrow G\sim S(n)$ is equivariant with respect to the two-sided action of $G\sim S(n)$, i.e.
$$
p_{n,n+1}\left((\kappa,\pi)(g,s)(h,t)\right)=(\kappa,\pi)p_{n,n+1}\left((g,s)\right)(h,t),
$$
for each $(g,s)\in G\sim S(n+1)$, and each $(\kappa,\pi)\in G\sim S(n)$, $(h,t)\in G\sim S(n)$.
\end{prop}
\begin{proof}The equivariance of $p_{n,n+1}$ follows from the description of the projection in terms of the corresponding bipartite graph.
\end{proof}
\subsection{Definition of $\mathfrak{S}_{G}$}
Recall the definition of the space of virtual permutations $\mathfrak{S}$ introduced in Kerov, Olshanski, and Vershik \cite{KerovOlshanskiVershik}, \S 2. If
$$
S(1)\longleftarrow\ldots\longleftarrow S(n)\longleftarrow S(n+1)\longleftarrow\ldots
$$
is a sequence of canonical projections for symmetric groups, then the space of virtual permutations $\mathfrak{S}$ is the projective limit, $\mathfrak{S}=\underset{\longleftarrow}{\lim}\;S(n)$. Note that $\mathfrak{S}$ is a profinite group, and correspondingly,
it is a compact toplogical space.

Here we introduce an analogue of the space of virtual permutations $\mathfrak{S}$ starting from
a sequence of canonical projections for wreath products. Set
$$
G_n=G\sim S(n)
$$
and consider the sequence of canonical projections
$$
G_1\longleftarrow\ldots\longleftarrow G_n\longleftarrow G_{n+1}\longleftarrow\ldots .
$$
Let
$$
\mathfrak{S}_{G}=\underset{\longleftarrow}{\lim}\;G_n
$$
denote the projective limit of the sets $G_n$. By definition, the elements of $\mathfrak{S}_{G}$ are sequences
$$
x=\left(x_1,x_2,\ldots\right)
$$
such that $x_n\in G_n$, and $p_{n,n+1}\left(x_{n+1}\right)=x_n$ for all $n=1,2,\ldots$, where
$$
p_{n,n+1}: G_{n+1}\longrightarrow G_n
$$
is the canonical projection introduced in section \ref{SectionProjection}. The space $\mathfrak{S}_{G}$ is called the \textit{space of $G$-virtual permutations}.
\section{Central measures on $\mathfrak{S}_{G}$}
\subsection{The wreath product $G_{\infty}=G\sim S(\infty)$}\label{SubSectionWreathInfty}
Recall that $S(\infty)$ is the group of finite permutations of the set $\left\{1,2,\ldots\right\}$, and $G$ is a finite group.
Denote by $D_{\infty}\left(G\right)$ the restricted direct product of $G$, i.e.
$$
D_{\infty}\left(G\right)=\left\{g=\left(g_1,g_2,\ldots\right)\in G^{\infty}:\;g_j=e_{G}\;\mbox{except  finitely many j's}
\right\}.
$$
Here $e_{G}$ denotes the unit element of $G$. The infinite symmetric group $S(\infty)$ acts on
$D_{\infty}\left(G\right)$ according to the formula
\begin{equation}\label{3Action}
s(g)=\left(g_{s^{-1}(1)},g_{s^{-1}(2)},\ldots \right),\; s\in S(\infty),\; g\in D_{\infty}\left(G\right).
\end{equation}
\begin{defn}
The wreath product $G_{\infty}=G\sim S(\infty)$ of a finite group $G$ with the infinite symmetric group $S(\infty)$ is the semidirect product of $D_{\infty}\left(G\right)$ with $S(\infty)$ defined by action (\ref{3Action}). The underlying set of $G_{\infty}$ is $D_{\infty}\left(G\right)\times S(\infty)$, with multiplication defined by
$$
\left(g,s\right)\left(h,t\right)=\left(gs(h),st\right),
$$
where $g, h\in D_{\infty}\left(G\right)$, and $s, t\in S(\infty)$.
\end{defn}
Under the canonical inclusion $i_n:\; G_{n}\longrightarrow G_{\infty}$ we can regard $G_{\infty}$ as $\bigcup_{n=1}^{\infty}G_n$. Also, the group $G_{\infty}$ can be identified with the subset of $\mathfrak{S}_{G}$ consisting of  the stable sequences $\left(x_n\right)$ such that
$$
x_n=\left(g,s\right), \;\; \left(g,s\right)\in D_{\infty}\left(G\right)\times S(\infty),
$$
for sufficiently large $n$.

Let $W=\left(w_1,w_2\right)\in G_{\infty}\times G_{\infty}$. The right action of $G_{\infty}\times G_{\infty}$ on $\mathfrak{S}_{G}$ is defined as
$$
xW=y,\; x=\left(x_1,x_2,\ldots\right),\;y=\left(y_1,y_2,\ldots\right),
$$
where $y_n=w_2^{-1}x_nw_1$ for all large enough $n$. Specifically, the equality just written above holds whenever $n$ is so large that both $w_1$, $w_2$ are already in $G_n$. Note that
the action of $G_{\infty}\times G_{\infty}$ on $\mathfrak{S}_{G}$ preserves the topology.
\subsection{Central measures}
Let $P^{(n)}$ be a probability measure on $G_n=G\sim S(n)$. The measure $P^{(n)}$ is called central if it is invariant under the conjugation of $G_n$. Likewise, we say that a probability measure $P$ on $\mathfrak{S}_{G}$ is central if it is invariant under the conjugation of $G_{\infty}=G\sim S(\infty)$, i.e. if it is invariant under the action of
the diagonal subgroup $\diag\left(G_{\infty}\right)$ of $G_{\infty}\times G_{\infty}$
on $\mathfrak{S}_{G}$, see Section \ref{SubSectionWreathInfty}.

By the classical Kolmogorov theorem, any family $\left(P^{(n)}\right)_{n=1}^{\infty}$ of probability measures on the groups $G_n=G\sim S(n)$ consistent with respect to the canonical projection $p_{n,n+1}:\; G_{n+1}\longrightarrow G_n$ gives rise to a probability measure $\mathcal{P}=\underset{\longleftarrow}{\lim}P^{(n)}$
on the space $\mathfrak{S}_{G}$. If $P^{(n)}$ is central for each $n$, then it is not hard to see that
$\mathcal{P}=\underset{\longleftarrow}{\lim}P^{(n)}$ is a central probability measure on $\mathfrak{S}_{G}$.
Conversely, each central probability measure $\mathcal{P}$ on $\mathfrak{S}_{G}$ can be represented as a projective limit $\mathcal{P}=\underset{\longleftarrow}{\lim}P^{(n)}$, where $P^{(n)}$ is a central
probability measure on $G_n=G\sim S(n)$.
\begin{thm}\label{THEOREMCENTRALMEASURES} There exists a one-to-one correspondence $\mathcal{P}\longleftrightarrow\Pi$ between arbitrary central probability measures $\mathcal{P}$ on the space $\mathfrak{S}_G$, and arbitrary probability measures $\Pi$
on the space
\begin{equation}\label{SetNabla}
\begin{split}
\overline{\nabla}^{(k)}=&\biggl\{(x,\delta)\biggr|
x=\left(x^{(1)},\ldots,x^{(k)}\right),\;
\delta=\left(\delta^{(1)},\ldots,\delta^{(k)}\right);\\
&x^{(l)}=\left(x^{(l)}_1\geq x^{(l)}_2\geq\ldots\geq 0\right),\;\delta^{(l)}\geq 0,\;
1\leq l\leq k,\\
&\mbox{where}\;\;\sum\limits_{i=1}^{\infty}x_i^{(l)}\leq\delta^{(l)},\; 1\leq l\leq k,\;\mbox{and}\;\sum\limits_{l=1}^k\delta^{(l)}=1\biggl\}.
\end{split}
\end{equation}
Here $k=\left|G_{\ast}\right|$, where $G_{\ast}$ denotes the set of conjugacy classes in $G$.
\end{thm}
Namely, each  central probability measure $\mathcal{P}$ on $\mathfrak{S}_{G}$
can be represented
as a projective limit measure,
$
\mathcal{P}=\underset{\longleftarrow}{\lim}P^{(n)},
$
where each  $P^{(n)}$ is a central probability measure on $G_n$. The measures
$P^{(n)}$ are given by  integral representations over $\overline{\nabla}^{(k)}$,
\begin{equation}\label{Dop2}
P^{(n)}\left(\Lambda_n^{(k)}\right)=\frac{1}{\left|K_{\Lambda_n^{(k)}}\right|}
\int\limits_{\overline{\nabla}^{(k)}}\mathbb{K}\left(\Lambda_n^{(k)},\omega\right)\Pi(d\omega),
\end{equation}
where $\mathbb{K}\left(\Lambda_n^{(k)},\omega\right)$ can be expressed in terms
of extended monomial symmetric functions, see Theorem 2.2 of Ref. \cite{StrahovMPS},
$\Pi$ is a probability measure on $\overline{\nabla}^{(k)}$, and $\left|K_{\Lambda_n^{(k)}}\right|$ is the number of elements in the conjugacy class
$K_{\Lambda_n^{(k)}}$ of $G_n$ given by equation (\ref{TheSizeOfTheConjugacyClass}).
Conversely, each probability measure  $\Pi$ on $\overline{\nabla}^{(k)}$ gives rise
to a central probability measure
$
\mathcal{P}=\underset{\longleftarrow}{\lim}P^{(n)},
$
via equation (\ref{Dop2}).

Let us describe in several words the idea of the proof of Theorem \ref{THEOREMCENTRALMEASURES}.  Since the conjugacy classes of $G_n$ are parameterized by multiple partitions, each central measure $P^{(n)}$ gives rise to a probability measure $\mathcal{M}_n^{(k)}$ on $\Y_n^{(k)}$, the set of multiple partitions of $n$ into $k$ components. The consistency of the family $\left(P^{(n)}\right)_{n=1}^{\infty}$ with canonical projection $p_{n,n+1}:\; G_{n+1}\longrightarrow G_n$ will imply that the family $\mathcal{M}_n^{(k)}$ is a \textit{multiple partition structure} in the sense of Strahov \cite{StrahovMPS}. Then we will use  Theorem 2.2 from Strahov \cite{StrahovMPS} which establishes a bijective correspondence between multiple partition structures  $\left(\mathcal{M}_n^{(k)}\right)_{n=1}^{\infty}$ and probability measures on the space
$\overline{\nabla}^{(k)}$.

The  proof of Theorem \ref{THEOREMCENTRALMEASURES} is given in Section \ref{SUBSECTIONPR}. In the next Section we recall the definition of multiple partition structures, and present  their relevant properties.
\subsection{Multiple partition structures}\label{SubSectionMPS}
It is convenient to identify multiple partitions with configuration of balls partitioned into boxes of different types.  Namely, suppose that a sample of $n$ identical balls is partitioned into boxes of $k$ different types.
Denote by $A_i^{(l)}$ the number of boxes of type $l$
containing precisely $i$ balls, where  $l\in\{1,\ldots,k\}$ and $i\in\{1,\ldots,n\}$.
Each list $\left(A_1^{(l)},\ldots,A_n^{(l)}\right)$ can be identified with a Young diagram
$\lambda^{(l)}$ according to the rule
$$
A_i^{(l)}=\sharp\;\mbox{of rows of size}\;i\;\;\mbox{in}\;\; \lambda^{(l)}.
$$
We write
\begin{equation}\label{SvyazDiagramaA}
\lambda^{(l)}=\left(1^{A_1^{(l)}}2^{A_2^{(l)}}\ldots n^{A_n^{(l)}}\right),\;\; 1\leq l\leq k,
\end{equation}
and form the family $\Lambda_n^{(k)}=\left(\lambda^{(1)},\ldots,\lambda^{(k)}\right)$.
It is not hard to check that $|\lambda^{(1)}|+\ldots+|\lambda^{(k)}|=n$, i.e. $\Lambda_n^{(k)}$
is a multiple partition of $n$ into $k$ components. Conversely, let $\Lambda_n^{(k)}=\left(\lambda^{(1)},\ldots,\lambda^{(k)}\right)$
be a multiple partition of $n$ into $k$ components.  Given $\lambda^{(l)}$ define $A_1^{(l)}$, $A_2^{(l)}$, $\ldots$, $A_n^{(l)}$
by formula (\ref{SvyazDiagramaA}) which means that exactly $A_i^{(l)}$ of the rows of $\lambda^{(l)}$
are equal to $i$. Then refer to $A_i^{(l)}$ as to the number of those boxes of type $l$ that
contain precisely $i$ balls.
Thus each $\Lambda_n^{(k)}$ corresponds to a configuration of $n$ balls partitioned into boxes
of $k$ different types and vice versa.

A \textit{random multiple partition} of $n$ with $k$ components is a   random variable
$\Lambda_n^{(k)}$ with values in the set $\Y_n^{(k)}$ defined by
\begin{equation}\label{SetYn}
\Y_n^{(k)}=\left\{\left(\lambda^{(1)},\ldots,\lambda^{(k)}\right):\;
|\lambda^{(1)}|+\ldots+|\lambda^{(k)}|=n\right\}.
\end{equation}
The set $\Y_n^{(k)}$ is called the set of multiple partitions of $n$
into $k$ components.
\begin{defn}\label{DEFINITIONMULTIPLEPARTITIONSTRUCTURE}
A multiple partition structure is a sequence $\mathcal{M}_1^{(k)}$, $\mathcal{M}_2^{(k)}$, $\ldots$ of distributions for $\Lambda_1^{(k)}$, $\Lambda_2^{(k)}$, $\ldots$ which is consistent in the following sense: if $n$ balls are partitioned into boxes of $k$ different types such that their configuration is $\Lambda_n^{(k)}$, and a ball is deleted uniformly at random, independently of $\Lambda_n^{(k)}$, then the multiple partition $\Lambda_{n-1}^{(k)}$ describing the configuration of the remaining balls is distributed according to $\mathcal{M}^{(k)}_{n-1}$.
\end{defn}

If $k=1$ then a multiple partition structure is a partition structure in the sense of Kingman \cite{Kingman1}.

It is shown in Strahov \cite{StrahovMPS}, Section 8.1, that the sequence
$\left(\mathcal{M}_n^{(k)}\right)_{n=1}^{\infty}$ is a multiple partition structure if and only if the consistency condition
\begin{equation}\label{ConsistencyConditionMPS}
\mathcal{M}_n^{(k)}\left(\Lambda_n^{(k)}\right)=
\sum\limits_{\tLambda_{n+1}^{(k)}\in\Y_{n+1}^{(k)}}
\Prob\left(\Lambda_n^{(k)}|\tLambda_{n+1}^{(k)}\right)\mathcal{M}_{n+1}^{(k)}\left(\tLambda_{n+1}^{(k)}\right),\;\forall\Lambda_n^{(k)}\in\Y_n^{(k)},\;\forall n=1,2,\ldots
\end{equation}
is satisfied. The Markov transition kernel in equation (\ref{ConsistencyConditionMPS}) can be written explicitly. Indeed, assume that  $\tLambda_{n+1}^{(k)}=\left(\mu^{(1)},\ldots,\mu^{(k)}\right)$, and $\Lambda_n^{(k)}=\left(\lambda^{(1)},\ldots,\lambda^{(k)}\right)$. Then the number
$\Prob\left(\Lambda_n^{(k)}|\tLambda_{n+1}^{(k)}\right)$ is not equal to zero only if
\begin{equation}\label{CTCM}
\mu^{(l)}\searrow\lambda^{(l)},\;\mbox{and}\;
\mu^{(i)}=\lambda^{(i)}\;
\mbox{for all}\; i\in\left\{1,\ldots,k\right\}, i\neq l
\end{equation}
is satisfied for some $l$, $l\in\left\{1,\ldots,k\right\}$.
The notation $\mu^{(l)}\searrow\lambda^{(l)}$ means that
$\mu^{(l)}$ is obtained from $\lambda^{(l)}$ by adding a box to some row of $\lambda^{(l)}$ of size $L^{(l)}-1$, $L^{(l)}\geq 1$, and we have
\begin{equation}
\Prob\left(\Lambda_n^{(k)}|\tLambda_{n+1}^{(k)}\right)=\frac{1}{n+1}r_{L^{(l)}}\left(\mu^{(l)}\right)L^{(l)}
\end{equation}
see Strahov \cite{StrahovMPS}, Section 8.1,
Here $r_{L^{(l)}}\left(\mu^{(l)}\right)$ is the number of rows of size $L^{(l)}$ in the Young diagram $\mu^{(l)}$.

\subsection{Proof of Theorem \ref{THEOREMCENTRALMEASURES}}\label{SUBSECTIONPR}
Let $\mathcal{P}$ be a central measure on $\mathfrak{S}_{G}$. Represent $\mathcal{P}$ as a projective measure, $\mathcal{P}=\underset{\longleftarrow}{\lim}P^{(n)}$. Then each $P^{(n)}$ is a central probability measure on $G_n=G\sim S(n)$. The consistency condition of $\left(P^{(n)}\right)_{n=1}^{\infty}$ with respect to the canonical projection $p_{n,n+1}:\; G_{n+1}\longrightarrow G_n$ reads
\begin{equation}\label{7.1CMT}
\underset{p_{n,n+1}(y)=x}{\sum\limits_{y:\;y\in G_{n+1}}}P^{(n+1)}(y)=P^{(n)}(x).
\end{equation}
We would like to show that each such sequence $\left(P^{(n)}\right)_{n=1}^{\infty}$ is in one-to-one correspondence with a multiple partition structure $\left(\mathcal{M}_n^{(k)}\right)_{n=1}^{\infty}$.
Denote by $P^{(n)}\left(\Lambda_n^{(k)}\right)$ the value of $P^{(n)}$ on the conjugacy class $K_{\Lambda_n^{(k)}}$ of $G_n=G\sim S(n)$ parameterized by a multiple partition $\Lambda_n^{(k)}$, and by
$P^{(n+1)}\left(\tLambda_{n+1}^{(k)}\right)$ the value of $P^{(n+1)}$ on the conjugacy class
$K_{\tLambda_{n+1}^{(k)}}$ of $G_{n+1}=G\sim S(n+1)$ parameterized by  a multiple partition
$\tLambda_{n+1}^{(k)}$. Then the consistency condition (\ref{7.1CMT}) implies
\begin{equation}\label{7.2CMT}
\sum\limits_{\tLambda_{n+1}^{(k)}\in\;\Y_{n+1}^{(k)}}\upsilon\left(\tLambda_{n+1}^{(k)},\Lambda_n^{(k)}\right)
P^{(n+1)}\left(\tLambda_{n+1}^{(k)}\right)=P^{(n)}\left(\Lambda_n^{(k)}\right),
\end{equation}
where $\upsilon\left(\tLambda_{n+1}^{(k)},\Lambda_n^{(k)}\right)$ is the number of  elements $y$,
$y\in K_{\tLambda_{n+1}^{(k)}}$, whose image under $p_{n,n+1}$ is a given element $x$, $x\in K_{\Lambda_n^{(k)}}$. The number $\upsilon\left(\tLambda_{n+1}^{(k)},\Lambda_n^{(k)}\right)$ can be computed explicitly. Assume that $\tLambda_{n+1}^{(k)}=\left(\mu^{(1)},\ldots,\mu^{(k)}\right)$, $\Lambda_n^{(k)}=\left(\lambda^{(1)},\ldots,\lambda^{(k)}\right)$.
Set $x=(g,s)$, $y=(\tilde{g},t)$, where
$g=\left(g_1,\ldots,g_n\right)\in G^n$, $\tilde{g}=\left(\tilde{g}_1,\ldots,\tilde{g}_{n+1}\right)\in G^{n+1}$,
$s\in S(n)$, and $t\in S(n+1)$. If $p_{n,n+1}(y)=x$, then $n+1$ is extracted from one of the cycles of $t$, and $s$ is obtained.  If the cycle of $n+1$ in $t$ is $(n+1)$, then $s$ is obtained from $t$ by removing this cycle.

Suppose that the cycle of $n+1$ in $t$ has length two or more. Then we remove $n+1$ from this cycle, and the obtained cycle of $s$ is of the same type as the original cycle of $t$, see Section \ref{SectionProjection}. Since the multiple partition $\Lambda_n^{(k)}$ describes the cycles of $s$, and the multiple partition $\tLambda_{n+1}^{(k)}$ describes the cycles of $t$, the number
$\upsilon\left(\tLambda_{n+1}^{(k)},\Lambda_n^{(k)}\right)$ is not equal to zero only if the condition (\ref{CTCM})
is satisfied for some $l$, $l\in\left\{1,\ldots,k\right\}$.

Add $n+1$ to   $s$  as a separate cycle $(n+1)$, and change
$g=\left(g_1,\ldots,g_n\right)$ into $\tilde{g}=\left(g_1,\ldots,g_n,g_{n+1}\right)$ (where
$g_{n+1}\in c_l$) to get $y=\left(\tilde{g},t\right)\in K_{\tLambda_{n+1}^{(k)}}$.
If $\Lambda_n^{(k)}=\left(\lambda^{(1)},\ldots,\lambda^{(k)}\right)$, then $\tLambda_{n+1}^{(k)}=\left(\mu^{(1)},\ldots,\mu^{(k)}\right)$ is such that $\mu^{(l)}$ is obtained
by adding one box to the bottom of $\lambda^{(l)}$ (and by keeping $\lambda^{(i)}=\mu^{(i)}$
for all $i\neq l$). In this case
$$
\upsilon\left(\tLambda_{n+1}^{(k)},\Lambda_n^{(k)}\right)=|c_l|.
$$

Assume that $n+1$ is inserted into a cycle of $s$ whose cycle-type is $l$, and whose length is $L^{(l)}-1$
(where $L^{(l)}\geq 2$). Namely, assume that the cycle
$$
i_1\rightarrow\ldots\rightarrow i_m\rightarrow i_{m+1}\rightarrow\ldots\rightarrow
i_{L^{(l)}-1}
$$
of $s$  turns into the cycle
$$
i_1\rightarrow\ldots\rightarrow i_m\rightarrow n+1\rightarrow i_{m+1}\rightarrow\ldots\rightarrow
i_{L^{(l)}-1}
$$
of $t$. Then $g=\left(g_1,\ldots,g_{i_{m+1}},\ldots,g_n\right)$ is replaced by
$\tilde{g}=\left(g_1,\ldots,g_{i_{m+1}}g^{-1}_{n+1},\ldots,g_n,g_{n+1}\right)$, where
$g_{n+1}$ is an arbitrary element of $G$. As a result, $\left(\tilde{g},t\right)\in K_{\tLambda_{n+1}^{(k)}}$, and $\mu^{(l)}$ is obtained from $\lambda^{(l)}$ by adding
a box to some row of $\lambda^{(l)}$  whose size is $L^{(l)}-1$. In this case
$$
\upsilon\left(\tLambda_{n+1}^{(k)},\Lambda_n^{(k)}\right)=|G|\left(L^{(l)}-1\right)r_{L^{(l)}-1}\left(\lambda^{(l)}\right).
$$
We conclude that
\begin{equation}
\upsilon\left(\tLambda_{n+1}^{(k)},\Lambda_n^{(k)}\right)=
\left\{
  \begin{array}{ll}
    |G|\left(L^{(l)}-1\right)r_{L^{(l)}-1}\left(\lambda^{(l)}\right), & L^{(l)}\geq 2, \\
    |c_l|, & L^{(l)}=1.
  \end{array}
\right.
\end{equation}
provided that  $\Lambda_n^{(k)}=\left(\lambda^{(1)},\ldots,\lambda^{(k)}\right)$ and $\tLambda_{n+1}^{(k)}=\left(\mu^{(1)},\ldots,\mu^{(k)}\right)$ are such that condition
(\ref{CTCM}) is satisfied.

Set
$$
\mathcal{M}_n^{(k)}\left(\Lambda_n^{(k)}\right)=\left|K_{\Lambda_n^{(k)}}\right|
P^{(n)}\left(\Lambda_n^{(k)}\right).
$$
This equation defines a probability measure on the set $\Y_n^{(k)}$ of the multiple partitions of $n$ with $k$ components.
In the case $L^{(l)}\geq 2$ we find
\begin{equation}\label{3K3.7}
\frac{\left|K_{\tLambda_{n+1}^{(k)}}\right|}{\left|K_{\Lambda_n^{(k)}}\right|}
=(n+1)\left|G\right|\frac{L^{(l)}-1}{L^{(l)}}
\frac{r_{L^{(l)}-1}\left(\lambda^{(l)}\right)}{r_{L^{(l)}}\left(\mu^{(l)}\right)},
\end{equation}
where we have used equation (\ref{TheSizeOfTheConjugacyClass}). In the case $L^{(l)}=1$ we obtain
\begin{equation}\label{3K3.8}
\frac{\left|K_{\tLambda_{n+1}^{(k)}}\right|}{\left|K_{\Lambda_n^{(k)}}\right|}
=(n+1)\left|G\right|
\frac{1}{r_{1}\left(\mu^{(l)}\right)}\frac{1}{\zeta_{c_l}}.
\end{equation}
It follows from (\ref{7.2CMT}), (\ref{7.2CMT}), (\ref{3K3.7}), (\ref{3K3.8}) that the sequence
$\left(\mathcal{M}_n^{(k)}\right)_{n=1}^{\infty}$ of such probability measures satisfies the consistency condition (\ref{ConsistencyConditionMPS}).

In other words, the consistency of the family $\left(P^{(n)}\right)_{n=1}^{\infty}$
under canonical projection $p_{n,n+1}:\; G_{n+1}=G\sim S(n+1)\longrightarrow G_n=G\sim S(n)$
translates as the condition on $\left(\mathcal{M}_n^{(k)}\right)_{n=1}^{\infty}$ to be a  multiple partition structure, see Section \ref{SubSectionMPS}, and  Strahov \cite{StrahovMPS}. In particular, Theorem 2.2 in Strahov \cite{StrahovMPS} establishes a bijective correspondence between multiple partition structures  $\left(\mathcal{M}_n^{(k)}\right)_{n=1}^{\infty}$ and probability measures $\Pi$ on the space $\overline{\nabla}^{(k)}$ defined by equation (\ref{SetNabla}).
As a consequence, we obtain a bijective correspondence between arbitrary central probability measures $\mathcal{P}$ on the space $\mathfrak{S}_{G}$, and arbitrary probability measures $\Pi$ on the space
$\overline{\nabla}^{(k)}$. Theorem \ref{THEOREMCENTRALMEASURES} is proved.
\qed
\section{The fundamental cocycles of the dynamical system $\left(\mathfrak{S}_{G},G_{\infty}\times G_{\infty}\right)$}
Recall that $c_1$, $\ldots$, $c_k$ denote the conjugacy classes of the finite group $G$. If
$$
y=\left(\left(g_1,\ldots,g_n\right),s\right)\in G\sim S(n),
$$
then $[y]_{c_l}$ denotes the number of cycles of type $c_l$ in $s$.
\begin{thm}\label{THEOREMCocycles} There exist $k$ integer functions $C_{l}\left(x,W\right)$ on $\mathfrak{S}_{G}\times\left(G_{\infty}\times G_{\infty}\right)$ uniquely defined by the following property:
if $n$ is large enough so that $W\in G_n\times G_n$, then
\begin{equation}\label{Cocycles}
C_l\left(x,W\right)=\left[p_n\left(xW\right)\right]_{c_l}-\left[p_n(x)\right]_{c_l}.
\end{equation}
Here $l=1,\ldots,k$;  $p_n :\;\mathfrak{S}_G\longrightarrow G\sim S(n)$ is the natural projection, and the action of $G_{\infty}\times G_{\infty}$ on $\mathfrak{S}_{G}$ is defined as in Section \ref{SubSectionWreathInfty}.
\end{thm}
\begin{proof}We need to prove that the quantities $C_l(x,W)$ defined by equation (\ref{Cocycles}) do not depend on $n$ provided that $n$ is so large that
the element $W$ of $G_{\infty}\times G_{\infty}$ already belongs to $G_n\times G_n$.

Let $W$ be an element of $G_n\times G_n$. By Proposition \ref{PROPOSITIONEQUIVARIANCEPROJECTION} the projection $p_n$ is equivariant with respect to two-sided action of $G_n$. Thus in order to
prove Theorem \ref{THEOREMCocycles} it is enough to show that the condition
\begin{equation}\label{CocCondition}
\left[x_nW\right]_{c_l}-\left[x_n\right]_{c_l}=\left[x_{n+1}W\right]_{c_l}-\left[x_{n+1}\right]_{c_l}
\end{equation}
is satisfied for all $x_n\in G_n$, and for all $x_{n+1}\in G_{n+1}$ such that $x_n=p_{n,n+1}\left(x_{n+1}\right)$.
\begin{prop}\label{PROPOSITIONCoW}Assume that (\ref{CocCondition}) is satisfied for $W_1,\ldots, W_m\in G_n\times G_n$. Then
(\ref{CocCondition}) is also satisfied for the product $W_1\ldots W_m$. In other words, if
\begin{equation}\label{CocCondition1}
\left[x_nW_p\right]_{c_l}-\left[x_n\right]_{c_l}=\left[x_{n+1}W_p\right]_{c_l}-\left[x_{n+1}\right]_{c_l}
\end{equation}
holds true for each $p=1,\ldots,m$,  and for all $x_n\in G_n$, $x_{n+1}\in G_{n+1}$ such that $x_n=p_{n,n+1}\left(x_{n+1}\right)$, then
\begin{equation}\label{CocCondition2}
\left[x_nW_1\ldots W_m\right]_{c_l}-\left[x_n\right]_{c_l}=\left[x_{n+1}W_1\ldots W_m\right]_{c_l}-\left[x_{n+1}\right]_{c_l}
\end{equation}
holds true as well.
\end{prop}
\begin{proof} The proof is by induction. If $m=1$, then (\ref{CocCondition2}) turns into
(\ref{CocCondition1}) which holds true by our assumption in the statement of Proposition \ref{PROPOSITIONCoW}.
Assume that $m\geq 2$, and that
\begin{equation}\label{CocCondition3}
\left[x_nW_1\ldots W_{m-1}\right]_{c_l}-\left[x_n\right]_{c_l}=\left[x_{n+1}W_1\ldots W_{m-1}\right]_{c_l}-\left[x_{n+1}\right]_{c_l}
\end{equation}
is satisfied for  all $x_n\in G_n$, and for all $x_{n+1}\in G_{n+1}$ such that $x_n=p_{n,n+1}\left(x_{n+1}\right)$.
Denote
$$
\tilde{x}_n=x_nW_1\ldots W_{m-1},
$$
and observe that
\begin{equation}\label{P3211}
p_{n,n+1}\left(x_{n+1}W_1\ldots W_{m-1}\right)=p_{n,n+1}\left(x_{n+1}\right)W_1\ldots W_{m-1}=\tilde{x}_n,
\end{equation}
where we have used the equivariance of the projection $p_{n,n+1}$, see Proposition \ref{PROPOSITIONEQUIVARIANCEPROJECTION}. Now we write
\begin{equation}\label{P3212}
\left[x_nW_1\ldots W_{m}\right]_{c_l}-\left[x_n\right]_{c_l}=\left[\tilde{x}_nW_m\right]_{c_l}-\left[\tilde{x}_n\right]_{c_l}+\left[\tilde{x}_n\right]_{c_l}-\left[x_n\right]_{c_l}.
\end{equation}
By our assumption in the statement of  Proposition \ref{PROPOSITIONCoW}, and by (\ref{P3211}) the first difference in the right-hand side of equation (\ref{P3212}) can be rewritten as
\begin{equation}\label{P3213}
\left[\tilde{x}_nW_m\right]_{c_l}-\left[\tilde{x}_n\right]_{c_l}=\left[x_{n+1}W_1\ldots W_m\right]_{c_l}-\left[x_{n+1}W_1\ldots W_{m-1}\right]_{c_l}.
\end{equation}
The second difference in the right-hand side of equation (\ref{P3212}) is
\begin{equation}\label{P3214}
\left[\tilde{x}_n\right]_{c_l}-\left[x_n\right]_{c_l}=\left[x_nW_1\ldots W_{m-1}\right]_{c_l}-\left[x_n\right]_{c_l}.
\end{equation}
By assumption (\ref{CocCondition3}), it can be replaced by $\left[x_{n+1}W_1\ldots W_{m-1}\right]_{c_l}-\left[x_{n+1}\right]_{c_l}$. Then we conclude that
$$
\left[x_nW_1\ldots W_{m}\right]_{c_l}-\left[x_n\right]_{c_l}=\left[x_{n+1}W_1\ldots W_{m}\right]_{c_l}-\left[x_{n+1}\right]_{c_l}
$$
holds true as well. Proposition \ref{PROPOSITIONCoW} is proved.
\end{proof}
Proposition \ref{PROPOSITIONCoW} implies that it suffices to prove (\ref{CocCondition})
for $W$ of the form $\left(w,e_{G_n}\right)$, and for $W$ of the form $\left(e_{G_n},w\right)$
where $w\in G_n=G\sim S(n)$, and $e_{G_n}$ is the unit element of $G_n=G\sim S(n)$. Below we consider the second case. In this case we need to prove that
\begin{equation}
\left[wx_n\right]_{c_l}-\left[x_n\right]_{c_l}=\left[wx_{n+1}\right]_{c_l}-\left[x_{n+1}\right]_{c_l}
\end{equation}
for any $w=\left(\left(g_1,\ldots,g_n\right),s\right)\in G\sim S(n)$, any $x_n\in G\sim S(n)$, and any $x_{n+1}\in G\sim S(n+1)$ such that $p_{n,n+1}\left(x_{n+1}\right)=x_n$.
\begin{prop}\label{PROPOSITIONCoW1} With $x_n$ and $x_{n+1}$ as above,
\begin{equation}\label{ge}
\left[\left(\left(g_1,\ldots,g_n\right),e_{S(n)}\right)x_n\right]_{c_l}-\left[x_n\right]_{c_l}
=\left[\left(\left(g_1,\ldots,g_n\right),e_{S(n)}\right)x_{n+1}\right]_{c_l}-\left[x_{n+1}\right]_{c_l},
\end{equation}
where $e_{S(n)}$ denotes the unit element of $S(n)$, and $g_1$, $\ldots$, $g_n$ are arbitrary elements of $G$.
\end{prop}
\begin{proof} Assume that $x_n=(h,t)$, $x_{n+1}=\left(\tilde{h},t'\right)$, and $t'$ is obtained from $t$ by inserting $n+1$ into the existing cycle:
$$
i_1\rightarrow \ldots\rightarrow i_p\rightarrow i_{p+1}\rightarrow\ldots\rightarrow i_1\Longrightarrow
i_1\rightarrow\ldots\rightarrow i_{p}\rightarrow n+1\rightarrow i_{p+1}\rightarrow\ldots\rightarrow i_1.
$$
If $h=\left(h_1,\ldots,h_n\right)$,  and $n+1$ is inserted between the numbers $i_p$ and $i_{p+1}$,
then $\tilde{h}=\left(h_1,\ldots,h_{i_{p+1}}h_{n+1}^{-1},\ldots,h_n,h_{n+1}\right)$,
i.e. $\tilde{h}$ is obtained from $h$ by adding an additional element of $G$ to the list $\left(h_1,\ldots,h_n\right)$
(this additional element is denoted by $h_{n+1}$), and by multiplication of $h_{i_{p+1}}$ by $h_{n+1}^{-1}$ from the right. Note that  the cycle-product of $x_{n+1}$ corresponding to the cycle
\begin{equation}\label{cycleiinserted}
i_1\rightarrow i_2\rightarrow\ldots\rightarrow i_{p}\rightarrow n+1\rightarrow i_{p+1}\rightarrow\ldots\rightarrow i_1
\end{equation}
is $h_{i_1}\ldots\left(h_{i_{p+1}}h_{n+1}^{-1}\right)h_{n+1}h_{i_p}\ldots h_{i_2}$,
which is the same as the cycle-product of $x_n$ corresponding to the cycle
\begin{equation}\label{cyclei}
i_1\rightarrow \ldots\rightarrow i_p\rightarrow i_{p+1}\rightarrow\ldots\rightarrow i_1
\end{equation}
We conclude that
$$
\left[x_n\right]_{c_l}=\left[x_{n+1}\right]_{c_l},\;\; l=1,\ldots, k
$$
in this case.

Now we have
\begin{equation}\label{4.13.a}
\left(\left(g_1,\ldots,g_n\right),e_{S(n)}\right)x_n=\left(\left(g_1h_1,\ldots,g_nh_n\right),t\right),
\end{equation}
and
$$
\left(\left(g_1,\ldots,g_n\right),e_{S(n)}\right)x_{n+1}
=\left(\left(g_1h_1,\ldots,g_{i_{p+1}}h_{i_{p+1}}h_{n+1}^{-1},\ldots,
g_nh_n,h_{n+1}\right),t'\right).
$$
We check that the cycle-product of $\left(\left(g_1,\ldots,g_n\right),e_{S(n)}\right)x_n$ corresponding to the
cycle (\ref{cyclei}) is the same as that of $\left(\left(g_1,\ldots,g_n\right),e_{S(n)}\right)x_{n+1}$
corresponding to the cycle (\ref{cycleiinserted}). Therefore,
$$
\left[\left(\left(g_1,\ldots,g_n\right),e_{S(n)}\right)x_n\right]_{c_l}
=\left[\left(\left(g_1,\ldots,g_n\right),e_{S(n)}\right)x_{n+1}\right]_{c_l}, \;\; l=1,\ldots, k.
$$
Thus we conclude that if $t'$ is obtained from $t$ by inserting $n+1$ into an existing cycle of $t$, then
equation  (\ref{ge}) is satisfied.

Now assume that $t'$  is obtained from $t$ by adding a new cycle of the form $(n+1)$. Then
$\tilde{h}=\left(h_1,\ldots,h_n,h_{n+1}\right)$, and we have
$$
\left(\left(g_1,\ldots,g_n\right),e_{S(n)}\right)x_{n+1}
=\left(\left(g_1h_1,\ldots,g_nh_n,h_{n+1}\right),t'\right).
$$
As for the product $\left(\left(g_1,\ldots,g_n\right)e_{S(n)}\right)x_n$, it is given by equation (\ref{4.13.a}).
We find
$$
\left[x_{n+1}\right]_{c_l}-\left[x_n\right]_{c_l}=
 \left\{
 \begin{array}{ll}
 1, & h_{n+1}\in c_l, \\
 0, & \mbox{otherwise},
\end{array}
\right.
$$
and
$$
\left[\left(\left(g_1,\ldots,g_n\right),e_{S(n)}\right)x_{n+1}\right]_{c_l}
-\left[\left(\left(g_1,\ldots,g_n\right),e_{S(n)}\right)x_n\right]_{c_l}=
 \left\{
 \begin{array}{ll}
 1, & h_{n+1}\in c_l, \\
 0, & \mbox{otherwise}.
\end{array}
\right.
$$
Equation (\ref{ge}) holds true in this case as well.
\end{proof}
\begin{prop}\label{PROPOSITIONCoW2} Assume that $s\in S(n)$ is a transposition $(ij)$ (where $1\leq i<j\leq n$), and let $e_{G^{n}}$ denote the unit element of $G^n=G\times\ldots\times G$. We have
\begin{equation}\label{3.3.13}
\left[\left(e_{G^n},s\right)x_n\right]_{c_l}-\left[x_n\right]_{c_l}
=\left[\left(e_{G^n},s\right)x_{n+1}\right]_{c_l}-\left[x_{n+1}\right]_{c_l}
\end{equation}
for each $l=1,\ldots,k$, each $x_n\in G\sim S(n)$, and each $x_{n+1}\in G\sim S(n+1)$ such that $p_{n,n+1}\left(x_{n+1}\right)=x_n$.
\end{prop}
\begin{proof}Set
$$
x_n=\left(g,t\right),\; g=\left(g_1,\ldots,g_n\right),\; t\in S(n),
$$
and write $t$ as a product of cycles. First, let us assume that $i$ and $j$ are situated in two different cycles of $t$.
We can write these cycles explicitly as
\begin{equation}\label{3.3.14}
i\rightarrow i_1\rightarrow\ldots\rightarrow i_p\rightarrow i,
\end{equation}
and
\begin{equation}\label{3.3.15}
j\rightarrow j_1\rightarrow\ldots\rightarrow j_f\rightarrow j.
\end{equation}
The cycle-product of $x_n$ corresponding to the first cycle is $g_{i_p}\ldots g_{i_1}g_i$, and the cycle-product of $x_n$ corresponding to the second cycle is $g_{j_f}\ldots g_{j_1}g_{j}$.

Set $x_{n+1}=\left(g',t'\right)$, and assume that $t'$ is obtained from $t$ by creating a new cycle $(n+1)$, or by inserting $n+1$ to an existing cycle of $t$ which do not contain both $i$ and $j$. In this case it is not hard to check that (\ref{3.3.13}) is satisfied.
Indeed, the cycle in which $n+1$ is going to be inserted does not affect the difference
$\left[\left(e_{G^n},s\right)x_n\right]_{c_l}-\left[x_n\right]_{c_l}$, and the same cycle with inserted $n+1$ does not affect the difference
$\left[\left(e_{G^n},s\right)x_{n+1}\right]_{c_l}-\left[x_n\right]_{c_l}$.

If $t'$ is obtained from $t$ by inserting $n+1$ into the cycle with $i$, then  the cycle (\ref{3.3.14})
turns into the cycle
\begin{equation}\label{3.3.17}
i\rightarrow i_1\rightarrow\ldots\rightarrow i_m\rightarrow n+1\rightarrow i_{m+1}\rightarrow
\ldots\rightarrow i_p\rightarrow i,
\end{equation}
and we obtain $\left[x_{n+1}\right]_{c_l}=\left[x_n\right]_{c_l}$.  Indeed,
as soon as $x_n=p_{n,n+1}\left(x_{n+1}\right)$,
the cycle-product of $x_n$ corresponding to the cycle (\ref{3.3.14}) is the same as that of $x_{n+1}$ corresponding to the cycle
 (\ref{3.3.17}).

In $\left(e_{G^n},s\right)x_n$ the cycles (\ref{3.3.14}) and (\ref{3.3.15}) of $t$ merge into the cycle
\begin{equation}\label{3.3.18}
i\rightarrow i_1\rightarrow\ldots\rightarrow i_p\rightarrow j\rightarrow j_{1}\rightarrow
\ldots\rightarrow j_f\rightarrow i.
\end{equation}
The cycle-product of $\left(e_{G^n},s\right)x_n$ corresponding to the cycle (\ref{3.3.18})
is
\begin{equation}\label{zctype}
g_ig_{j_f}\ldots g_{j_1}g_{j}g_{i_p}\ldots g_{i_1}.
\end{equation}
In $\left(e_{G^n},s\right)x_{n+1}$ the cycles (\ref{3.3.17}) and (\ref{3.3.15}) of $t'$ merge into the cycle
\begin{equation}\label{3.3.19}
i\rightarrow i_1\rightarrow\ldots\rightarrow i_m\rightarrow n+1\rightarrow i_{m+1}\rightarrow
\ldots\rightarrow i_p\rightarrow j\rightarrow j_1\rightarrow \ldots\rightarrow j_f\rightarrow i,
\end{equation}
and the cycle-product of $\left(e_{G^n},s\right)x_{n+1}$ corresponding to this cycle is given by
(\ref{zctype}) as well. We conclude that
\begin{equation}\label{Zvzbv}
\left[\left(e_{G^n},s\right)x_{n}\right]_{c_l}=\left[\left(e_{G^n},s\right)x_{n+1}\right]_{c_l},
\end{equation}
and that equation (\ref{3.3.13}) holds true provided $i$ and $j$ are situated in two different cycles of $t$.

Second, consider the case where $i$ and $j$ are situated in the same cycle of $t$. Let us write this cycle as
\begin{equation}\label{3.3.20}
i\rightarrow i_1\rightarrow\ldots\rightarrow i_p\rightarrow j\rightarrow j_{1}\rightarrow
\ldots\rightarrow j_f\rightarrow i.
\end{equation}
In $t'$ the corresponding cycle includes $n+1$, and it takes the form
\begin{equation}\label{3.3.21}
i\rightarrow i_1\rightarrow\ldots\rightarrow i_{m}\rightarrow n+1\rightarrow i_{m+1}\rightarrow\ldots
\rightarrow i_p\rightarrow j\rightarrow j_{1}\rightarrow
\ldots\rightarrow j_f\rightarrow i.
\end{equation}
Again, we have
\begin{equation}
\left[x_{n}\right]_{c_l}=\left[x_{n+1}\right]_{c_l},
\end{equation}
for all $l=1,\ldots,k$ as the cycle-product of $x_n$ corresponding to (\ref{3.3.20}) is the same as that of $x_{n+1}$ corresponding to (\ref{3.3.21}). Now, in $\left(e_{G^n},s\right)x_{n}$ the multiplication of $x_n$ by $\left(e_{G^n},s\right)$ leads to the splitting of (\ref{3.3.20}) into two cycles
\begin{equation}\label{3.3.22}
i\rightarrow i_1\rightarrow\ldots\rightarrow i_p\rightarrow i\;\;\; \mbox{and}\;\;\; j\rightarrow j_{1}\rightarrow
\ldots\rightarrow j_f\rightarrow j.
\end{equation}
Moreover, in $\left(e_{G^n},s\right)x_{n+1}$ the multiplication of $x_{n+1}$ by $\left(e_{G^n},s\right)$ leads to the splitting of (\ref{3.3.21}) into two cycles
\begin{equation}\label{3.3.23}
i\rightarrow i_1\rightarrow\ldots\rightarrow i_m\rightarrow n+1\rightarrow i_m\rightarrow\ldots\rightarrow i\;\;\; \mbox{and}\;\;\; j\rightarrow j_{1}\rightarrow
\ldots\rightarrow j_f\rightarrow j.
\end{equation}
Again we check that the cycle-products of   $\left(e_{G^n},s\right)x_n$ corresponding to (\ref{3.3.22}) are the same as those of $\left(e_{G^n},s\right)x_{n+1}$ corresponding to (\ref{3.3.23}), and conclude that
(\ref{Zvzbv}) is satisfied. Therefore, equation (3.3.13) holds true in the situation where $i$ and $j$ are situated in the same cycle of $t$ as well.
\end{proof}
Propositions (\ref{PROPOSITIONCoW}), (\ref{PROPOSITIONCoW1}), and (\ref{PROPOSITIONCoW2}) imply that equation (\ref{CocCondition}) is satisfied for $W=\left(e_{G_n},w\right)$, where $w$ is an arbitrary element of $G_n$. The case of  $W=\left(w,e_{G_n}\right)$ can be considered in the same way.
Theorem \ref{THEOREMCocycles} is proved.
\end{proof}
The $k$ integer-valued  functions $C_{l}\left(x,W\right)$ are called the \textit{fundamental cocycles of the dynamical system} $\left(\mathfrak{S}_G,G_{\infty}\times G_{\infty}\right)$.
It is important for what follows that these functions can be defined
correctly  for all $x\in\mathfrak{S}_G$ and $W\in G_{\infty}\times G_{\infty}$.
\section{The Ewens measures}
\subsection{The Ewens distribution on $G_n=G\sim S(n)$}
Recall that the Ewens probability measure $P_{t,n}^{\Ewens}$ on the symmetric group $S(n)$ is defined by
\begin{equation}
P_{t,n}^{\Ewens}(s)=\frac{t^{[s]}}{t(t+1)\ldots(t+n-1)},\;\; s\in S(n),
\end{equation}
where $t>0$, and $[s]$ denotes the number of cycles in $s$. The measure $P^{\Ewens}_{t,n}$ is invariant under the action of $S(n)$ on itself by conjugation, so $P_{t,n}^{\Ewens}$ is a central measure. As a central measure
$P_{t,n}^{\Ewens}$ gives rise to a probability measure $\mathcal{M}_{t,n}^{\Ewens}$ on the set $\Y_n$ of Young diagrams with $n$ boxes. The measure $\mathcal{M}_{t,n}^{\Ewens}$ is
\begin{equation}\label{MEwensSymmetric}
\mathcal{M}_{t, n}^{\Ewens}\left(\lambda\right)
=\frac{n!}{(t)_{n}}
\frac{t^{l(\lambda)}}{\prod\limits_{j=1}^{n}j^{r_j(\lambda)}r_j(\lambda)!},
\end{equation}
where $l(\lambda)$ is the length of $\lambda$.

The Ewens probability measure $P_{t,n}^{\Ewens}$ admits a nontrivial generalization
$P_{t_1,\ldots,t_k,n}^{\Ewens}$ which is a probability distribution on $G_n=G\sim S(n)$.
\begin{defn} Fix $t_1>0$, $\ldots$, $t_k>0$, and set
\begin{equation}\label{EwensMeasureWreathProduct}
P_{t_1,\ldots,t_k;n}^{\Ewens}(x)=\frac{t_1^{[x]_{c_1}}t_2^{[x]_{c_2}}\ldots t_k^{[x]_{c_k}}}{
|G|^n\left(\frac{t_1}{\zeta_{c_1}}+\ldots+\frac{t_k}{\zeta_{c_k}}\right)_n},\;
x=\left((g_1,\ldots,g_n),s\right)\in G\sim S(n),
\end{equation}
where $c_1$, $\ldots$, $c_k$ are the conjugacy classes of $G$, $[x]_{c_l}$ is the number of cycles of type $c_l$ in $s$, $\zeta_{c_l}=\frac{|G|}{|c_l|}$, and $(a)_n=a(a+1)\ldots (a+n-1)$ is the Pochhammer symbol. Each
$P_{t_1,\ldots,t_k;n}^{\Ewens}$ is a probability measure on $G\sim S(n)$. It is known that
$P_{t_1,\ldots,t_k;n}^{\Ewens}(x)$ is a probability distribution on $G\sim S(n)$,
see Strahov \cite{StrahovMPS},  Proposition 4.1. This probability distribution  is called the \textit{ Ewens distribution on the wreath product of a finite group $G$ with the symmetric group $S(n)$}.
\end{defn}
Recall that $\Y_n^{(k)}$ denotes the set of  multiple partitions of $n$ into $k$ components. This set is defined by equation (\ref{SetYn}).
The pushforward of $P_{t_1,\ldots,t_k; n}^{\Ewens}$ on $\Y_n^{(k)}$ is denoted by $\mathcal{M}_{t_1,\ldots,t_k}^{\Ewens}\left(\Lambda_n^{(k)}\right)$.
This is a probability measure on $\Y_n^{(k)}$, which can be written explicitly as
\begin{equation}
\begin{split}
\mathcal{M}_{t_1,\ldots,t_k}^{\Ewens}\left(\Lambda_n^{(k)}\right)
&=\frac{n!}{|\lambda^{(1)}|!\ldots |\lambda^{(k)}|!}
\frac{\left(T_1\right)_{|\lambda^{(1)}|}
\ldots\left(T_k\right)_{|\lambda^{(k)}|}}{\left(T_1+\ldots+T_k\right)_n}\\
&\times
\mathcal{M}_{T_1,|\lambda^{(1)}|}^{\Ewens}\left(\lambda^{(1)}\right)\ldots
\mathcal{M}_{T_k,|\lambda^{(k)}|}^{\Ewens}\left(\lambda^{(k)}\right),
\end{split}
\end{equation}
where $\Lambda_n^{(k)}=\left(\lambda^{(1)},\ldots,\lambda^{(k)}\right)\in\Y_n^{(k)}$, the parameters
$T_1$, $\ldots$, $T_k$ are defined by
\begin{equation}\label{5.4.a}
T_l=\frac{t_l}{\zeta_{c_l}},\; 1\leq l\leq k,
\end{equation}
and $\mathcal{M}_{T, n}^{\Ewens}$ stands for the Ewens distribution on the set of Young diagrams with $n$ boxes defined by equation (\ref{MEwensSymmetric}).

The crucial property of the probability measures $P_{t_1,\ldots,t_k; n}^{\Ewens}$
is that they are pairwise consistent with respect to the canonical projection
$p_{n,n+1}: G_{n+1}\longrightarrow G_n$.
\begin{prop}\label{PropositionProjection}We have
\begin{equation}
\underset{p_{n,n+1}(\tilde{x})=x}{\sum\limits_{\tilde{x}:\; \tilde{x}\in G\sim S(n+1)}}P_{t_1,\ldots,t_k;n+1}^{\Ewens}(\tilde{x})
=P_{t_1,\ldots,t_k;n}^{\Ewens}(x).
\end{equation}
\end{prop}
\begin{proof} We compute
\begin{equation}
\underset{p_{n,n+1}(\tilde{x})=x}{\sum\limits_{\tilde{x}:\; \tilde{x}\in G\sim S(n+1)}}t_1^{[\tilde{x}]_{c_1}}
\ldots t_k^{[\tilde{x}]_{c_k}}
=\left(|G|n+|c_1|t_1+\ldots+|c_k|t_k\right)t_1^{[x]_{c_1}}\ldots t_k^{[x]_{c_k}}.
\end{equation}
The first term in the brackets, $|G|n$, comes from the fact that there are $n$ ways to insert $n+1$ to the existing cycle of $s\in S(n)$. If $n+1$ forms an extra cycle, then there are $|c_l|$ ways to increase the number of cycles of type $c_l$ in $s$. Proposition \ref{PropositionProjection} follows.
\end{proof}
\subsection{The probability space $\left(\mathfrak{S}_{G},P_{t_1,\ldots,t_k}^{\Ewens}\right)$}
\label{SectionProbSpaceConstruction}
 It follows from Proposition \ref{PropositionProjection} that for any $t_1>0$, $\ldots$, $t_k>0$ the canonical projection $p_{n,n+1}: G_{n+1}\rightarrow G_n$ introduced in Section \ref{SectionProjection} preserves the measures $P_{t_1,\ldots,t_k;n}^{\Ewens}$. Hence the measure
$$
P_{t_1,\ldots,t_k}^{\Ewens}=\underset{\longleftarrow}{\lim}P_{t_1,\ldots,t_k;n}^{\Ewens}
$$
is correctly defined, $P_{t_1,\ldots,t_k}^{\Ewens}$ is a probability measure on $\mathfrak{S}_{G}$, and
$\left(\mathfrak{S}_{G},P_{t_1,\ldots,t_k}^{\Ewens}\right)$ is a probability space.
\begin{prop}\label{Proposition5.1.1} For each set of strictly positive parameters $t_1$, $\ldots$, $t_k$
the measure $P_{t_1,\ldots,t_k}^{\Ewens}$ is quasiinvariant with respect to the action of $G_{\infty}\times G_{\infty}$ on $\mathfrak{S}_{G}$. More precisely, the Radon-Nikodym derivative is given by
\begin{equation}\label{5.1.1.1}
\frac{P_{t_1,\ldots,t_k}^{\Ewens}\left(dxW\right)}{P_{t_1,\ldots,t_k}^{\Ewens}(dx)}
=t_1^{C_1\left(x,W\right)}\ldots t_k^{C_k\left(x,W\right)},
\; x\in\mathfrak{S}_{G},\; W\in G_{\infty}\times  G_{\infty},
\end{equation}
where $C_i\left(x,W\right)$ are the fundamental cocycles of the dynamical system
$\left(\mathfrak{S}_G, G_{\infty}\times G_{\infty}\right)$, see Theorem \ref{THEOREMCocycles}.
\end{prop}
\begin{proof}
We need to check that  the equation
\begin{equation}\label{5.1.1.2}
P_{t_1,\ldots,t_k}^{\Ewens}\left(VW\right)
=\int\limits_{V}t_1^{C_1\left(x,W\right)}\cdots t_k^{C_k\left(x,W\right)}
P_{t_1,\ldots,t_k}^{\Ewens}\left(dx\right)
\end{equation}
is satisfied for every Borel subset $V\subseteq\mathfrak{S}_{G}$.

Assume that $W\in G_m\times G_m$, $n\geq m$, and choose $y\in G_n$.
Define $V_n(y)$ as a subset of $\mathfrak{S}_{G}$ consisting of the sequences
$\left(x_1, x_2,\ldots\right)\in\mathfrak{S}_{G}$ with the property $x_n=y$.
Note that each Borel set is generated by cylinder sets, and each cylinder set is a disjoint union of the sets of the form $V_n(y)$. Therefore, it is enough to check (\ref{5.1.1.2}) in the case $V=V_n(y)$.

Next observe that the function
$$
t_1^{C_1\left(.,W\right)}\cdots t_k^{C_k\left(.,W\right)}:\;\mathfrak{S}_{G}\longrightarrow\R
$$
is constant on $V_n(y)$, and its value on this set is given by
$$
t_1^{\left[y W\right]_{c_1}-[y]_{c_1}}\ldots t_k^{\left[yW\right]_{c_k}-\left[y\right]_{c_k}},
$$
see Theorem \ref{THEOREMCocycles}. Also,
$$
P_{t_1,\ldots,t_k}^{\Ewens}\left(V_n(y)\right)=P_{t_1,\ldots,t_k;n}^{\Ewens}\left(y\right)
$$
by the very construction of $P_{t_1,\ldots,t_k}^{\Ewens}$, see Section \ref{SectionProbSpaceConstruction}. We conclude that if $V=V_n(y)$
then the integral in the right-hand of equation (\ref{5.1.1.2}) is equal to
\begin{equation}
t_1^{\left[y W\right]_{c_1}-[y]_{c_1}}\ldots t_k^{\left[yW\right]_{c_k}-\left[y\right]_{c_k}}
P_{t_1,\ldots,t_k;n}^{\Ewens}\left(y\right)=P_{t_1,\ldots,t_k;n}^{\Ewens}\left(yW\right).
\end{equation}
Since $V_n(y) W=V_n\left(yW\right)$, and
$P_{t_1,\ldots,t_k}^{\Ewens}\left(V_n\left(yW\right)\right)=
P_{t_1,\ldots,t_k;n}^{\Ewens}\left(yW\right)$, we see that equation (\ref{5.1.1.2}) holds true indeed.
\end{proof}
\subsection{The correspondence between $P_{t_1,\ldots,t_k}^{\Ewens}$ and the multiple Poisson-Dirichlet distribution $PD\left(T_1,\ldots,T_k\right)$.}
\label{CorPPD}
Observe that $P_{t_1,\ldots,t_k}^{\Ewens}$ is a central probability measure on $\mathfrak{S}_{G}$. This follows from  the representation $P_{t_1,\ldots,t_k}^{\Ewens}=\underset{\longleftarrow}{\lim} P_{t_1,\ldots,t_k; n}^{\Ewens}$, and from the invariance of the probability measure
$P_{t_1,\ldots,t_k; n}^{\Ewens}$ on $G_n$ under the action on itself by conjugations. Also, this agrees with the fact that
$$
C_l\left(x,W\right)=0,\;\; W\in\diag\left(G_{\infty}\right),
$$
for each $l=1,\ldots,k$ in formula (\ref{5.1.1.1}) (as it follows from the very definition of the fundamental cocycles $C_l\left(x,W\right)$ in the statement of Theorem \ref{THEOREMCocycles}).

By Theorem \ref{THEOREMCENTRALMEASURES}  there is a unique probability measure
on the set $\overline{\nabla}^{(k)}$ defined by equation (\ref{SetNabla})
which corresponds to $P_{t_1,\ldots,t_k}^{\Ewens}$. This is the \textit{multiple Poisson-Dirichlet distribution} $PD\left(T_1,\ldots,T_k\right)$ introduced in Strahov \cite{StrahovMPS}, Section 4.4. The parameters  $T_1$, $\ldots$, $T_k$ are defined in terms of $t_1$, $\ldots$, $t_k$ by equation (\ref{5.4.a}). The multiple Poisson-Dirichlet distribution is a generalization of the Poisson-Dirichlet distribution $PD(t)$ (see Kingman \cite{Kingman1}), and it is defined as follows.

Recall that  the Poisson-Dirichlet distribution $PD(t)$ can be understood as the Poisson-Dirichlet limit of the Dirichlet distribution $D(\tau_1,\ldots,\tau_M)$ with density
\begin{equation}\label{PoissonDirichletDensity}
\frac{\Gamma\left(\tau_1+\ldots+\tau_M\right)}{\Gamma\left(\tau_1\right)\ldots\Gamma\left(\tau_M\right)}
x_1^{\tau_1-1}x_2^{\tau_2-1}\ldots x_M^{\tau_M-1}
\end{equation}
relative to the $(M-1)$- dimensional Lebesgue measure on the simplex
$$
\triangle_M=\left\{(x_1,\ldots,x_M):\; x_i\geq 0,\; x_1+\ldots+x_M=1\right\},
$$
where $\tau_1$, $\ldots$, $\tau_M$ are strictly positive parameters.
Assume that $(x_1,\ldots,x_M)$ has the Dirichlet distribution with equal parameters,
$$
\tau_1=\ldots=\tau_M=\frac{t}{M-1}.
$$
If
$
x_{(1)}\geq x_{(2)}\geq\ldots\geq x_{(M)}
$
denote the $x_j$ arranged in descending order, then $x_{(1)}$, $x_{(2)}$, $\ldots$ converge in joint distribution
as $M\longrightarrow\infty$, the limit is $PD(t)$. The Poisson-Dirichlet distribution
$PD(t)$ is concentrated on the set
\begin{equation}\label{nablazero}
\overline{\nabla}_0^{(1)}=\left\{x=\left(x_1,x_2,\ldots,\right):\;
x_1\geq x_2\geq\ldots\geq 0,\;\sum\limits_{i=1}^{\infty}x_i=1\right\}.
\end{equation}

Let $t_1>0$, $\ldots$, $t_k>0$.  For each $l$, $1\leq l\leq k$, let $x^{(l)}=\left(x^{(l)}_1,x^{(l)}_2,\ldots \right)$
be independent sequences of random variables such that
$$
x^{(l)}\sim PD(t_l),\; l=1,\ldots, k.
$$
Furthermore, let  $\delta^{(1)}$, $\ldots$, $\delta^{(k)}$ be random variables independent of $x^{(1)}$, $\ldots$, $x^{(k)}$, and such that joint distribution of
 $\delta^{(1)}$, $\ldots$, $\delta^{(k)}$ is the Dirichlet distribution $D\left(t_1,\ldots,t_k\right)$.
The joint distribution of the sequences
$\delta^{(1)}x^{(1)}$, $\ldots$, $\delta^{(k)}x^{(k)}$ is called the \textit{multiple Poisson-Dirichlet distribution}
$PD(t_1,\ldots,t_k)$.\\
The distribution $PD(t_1,\ldots,t_k)$ is concentrated on
\begin{equation}\label{GeneralizedThomaSet}
\begin{split}
\overline{\nabla}_0^{(k)}=&\biggl\{(x,\delta)\biggr|
x=\left(x^{(1)},\ldots,x^{(k)}\right),
\delta=\left(\delta^{(1)},\ldots,\delta^{(k)}\right);\\
&x^{(l)}=\left(x^{(l)}_1,x^{(l)}_2,\ldots\right),
x_1^{(l)}\geq x_2^{(l)}\geq\ldots\geq 0,\; 1\leq l\leq k,\\
&
\delta^{(1)}\geq 0,\ldots,\delta^{(k)}\geq 0,\\
&\mbox{where}\;\;\sum\limits_{i=1}^{\infty}x_i^{(l)}=\delta^{(l)},\; 1\leq l\leq k,\;\mbox{and}\;\sum\limits_{l=1}^k\delta^{(l)}=1\biggl\}.
\end{split}
\nonumber
\end{equation}
 If $k=1$, the multiple Poisson-Dirichlet distribution turns into the usual Poisson-Dirichlet distribution $PD(t_1)$.

The fact that $PD\left(T_1,\ldots, T_k\right)$ corresponds to $P_{t_1,\ldots,t_k}^{\Ewens}$ is a consequence of Theorem 4.3 in Strahov \cite{StrahovMPS}. Indeed, as it is explained in the proof of Theorem \ref{THEOREMCENTRALMEASURES}, the probability measure $P_{t_1,\ldots,t_k}^{\Ewens}$
gives rise to the multiple partition structure
$\left(\mathcal{M}_{t_1,\ldots,t_k; n}^{\Ewens}\right)_{n=1}^{\infty}$.
Theorem 4.3 in Strahov \cite{StrahovMPS} provides a representation of
$\left(\mathcal{M}_{t_1,\ldots,t_k; n}^{\Ewens}\right)_{n=1}^{\infty}$
in terms of $PD\left(T_1,\ldots, T_k\right)$.
\section{Generalized regular representation}\label{SectionGeneralizedRepresentations}
In this Section we construct a representation $T_{z_1,\ldots,z_k}$ of the group
$G_{\infty}\times G_{\infty}$, and show that $T_{z_1,\ldots,z_k}$ can be understood as the inductive limit of the two-sided regular representations
of $G_n\times G_n$.  We begin with a description of inductive limits of representations for chains of finite groups.
\subsection{Inductive limits of representations}\label{SectionInductive Limits}
Let $G(1)\subseteq G(2)\subseteq\ldots $ be a collection of finite groups. Set
$G=\bigcup\limits_{n=1}^{\infty}G(n)$.  The group $G$ is called the inductive limit of $G(1)$, $G(2)$, $\ldots$.
Assume that for each $n=1,2,\ldots $ a unitary representation $\left(T_n,H(T_n)\right)$ of $G(n)$ is defined (here $H\left(T_n\right)$
denotes the Hilbert space in which $T_n$ acts).  In addition, assume that for each $n$  a linear map
$f_n\; :H\left(T_n\right)\rightarrow H\left(T_{n+1}\right)$ which is an isometric embedding of $H\left(T_n\right)$ into $H\left(T_{n+1}\right)$ is defined, and that this embedding intertwines the $G(n)$-representations $T_n$ and $T_{n+1}|_{G(n)}$,
i.e. the condition
\begin{equation}\label{3.6.1}
T_{n+1}|_{G(n)}(g) f_n=f_nT_n(g)
\end{equation}
is satisfied for each $g\in G(n)$. Denote by $H$ the Hilbert completion of the space
$\bigcup\limits_{n=1}^{\infty}H\left(T_n\right)$.
In the space $H$   a unitary representation $T$ of the group $G$ arises which is
uniquely defined by the formula
\begin{equation}\label{3.6.2}
T(g)\zeta=T_n(g)\zeta,\;\;\mbox{if}\; g\in G(n)\; \mbox{and}\; \zeta\in H\left(T_n\right).
\end{equation}
The representation $\left(T,H\right)$ is called the \textit{inductive limit of the representations} $\left(T_1, H\left(T_1\right)\right)$,
$\left(T_2, H\left(T_2\right)\right)$, $\ldots$, see
Olshanski \cite{Olshanski}, Section 1.16.
  \begin{figure}[t]
{\scalebox{0.6}{\includegraphics{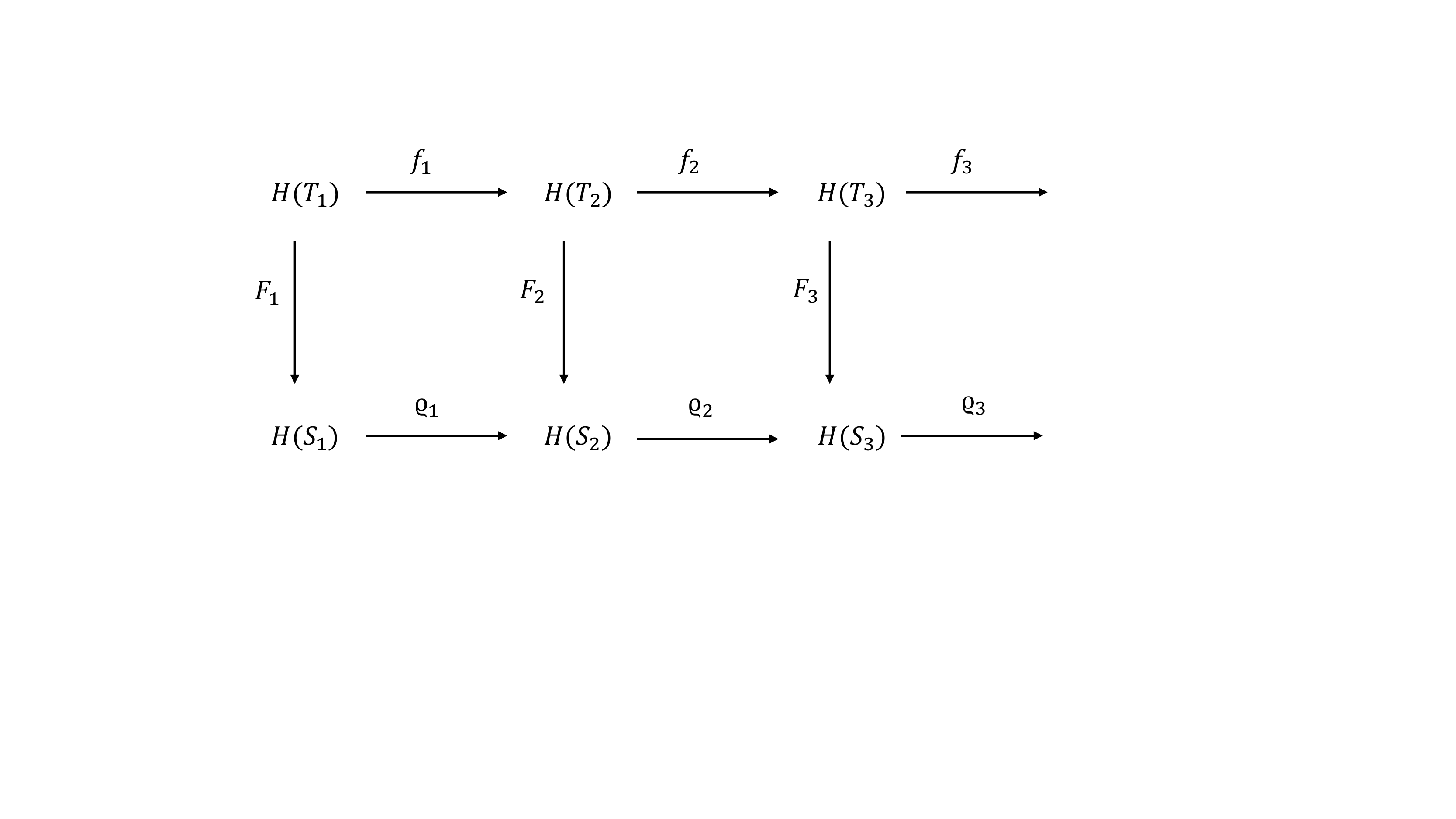}}}
\caption{The maps between the representation spaces}
\label{Fig.5}
\end{figure}

\begin{prop}\label{PropositionEqvInd}
Let $\left\{T_n\right\}_{n=1}^{\infty}$ and $\left\{S_n\right\}_{n=1}^{\infty}$ be collections of representations of finite groups $G(1)$, $G(2)$, $\ldots$ respectively, where $G(1)\subseteq G(2)\subseteq\ldots$.
Consider the  diagram shown on Fig.\ref{Fig.5}, and
assume that for each $n=1,2,\ldots$ the following conditions are satisfied:
\begin{itemize}
\item The linear map $F_n$ is from $H\left(T_n\right)$ onto $H\left(S_n\right)$, and it intertwines the $G(n)$-representations $T_n$ and $S_n$.
\item The linear map $f_n$ is an isometric embedding of $H\left(T_n\right)$ into $H\left(T_{n+1}\right)$, and it intertwines  the $G(n)$-representations $T_n$ and $T_{n+1}|_{G(n)}$.
\item The map $\varrho_n$ is an isometric embedding of $H\left(S_n\right)$ into $H\left(S_{n+1}\right)$ such that
the condition $F_{n+1}f_n=\varrho_nF_n$ is satisfied. In other words, the $n$th block of the diagram
on Fig.\ref{Fig.5} is commutative.
\end{itemize}
Then the inductive limits of $\left\{T_n\right\}_{n=1}^{\infty}$ and $\left\{S_n\right\}_{n=1}^{\infty}$ are well defined, and these inductive limits are equivalent representations.
\end{prop}
\begin{proof} Let us check that the isometric embedding $\varrho_n:\; H\left(S_n\right)\rightarrow H\left(S_{n+1}\right)$ intertwines the $G(n)$-representations $S_n$ and $S_{n+1}|_{G(n)}$, i.e. let us check that the condition
\begin{equation}\label{z743}
S_{n+1}(g)\varrho_n\tilde{\zeta}=\varrho_nS_n(g)\tilde{\zeta}
\end{equation}
is satisfied for each $\tilde{\zeta}\in H\left(S_n\right)$, and each $g\in G(n)$.
Consider the left-hand side of equation (\ref{z743}).
Since $F_n :\; H\left(T_n\right)\rightarrow H\left(S_n\right)$
is onto, $\tilde{\zeta}=F_n\zeta$ for some $\zeta\in H\left(T_n\right)$. Also,
$$
\varrho_nF_n\zeta=F_{n+1}f_n\zeta
$$
by the commutativity of the $n$th block in the diagram shown on Fig.\ref{Fig.5}. In addition, since $F_{n+1}$ intertwines the representations $T_{n+1}$ and $S_{n+1}$,
we have
$$
S_{n+1}(g)F_{n+1}=F_{n+1}T_{n+1}(g),\;\;\forall g\in G(n).
$$
Thus the left-hand side of equation (\ref{z743}) is equal to $F_{n+1}T_{n+1}(g)f_n\zeta$.

Now consider the right-hand side of equation(\ref{z743}). We have
$$
S_n(g)F_n\zeta=F_nT_n(g)\zeta,\;\forall g\in G(n),
$$
since the map $F_n$ intertwines the $G(n)$-representations $T_n$ and $S_n$. By the commutativity of the
$n$th block in the diagram shown on Fig.\ref{Fig.5},
$$
\varrho_nF_n=F_{n+1}f_n.
$$
Finally,
$$
f_nT_n(g)=T_{n+1}(g)f_n,\;\; \forall g\in G(n)
$$
since the isometric embedding $f_n: H\left(T_n\right)\rightarrow H\left(T_{n+1}\right)$ intertwines the $G(n)$-representations $T_n$ and $T_{n+1}|_{G(n)}$.
Thus we conclude that the right-hand side of equation (\ref{z743}) can be written as $F_{n+1}T_{n+1}(g)f_n\zeta$ as well, and condition (\ref{z743}) is satisfied.

Let $H$ be the Hilbert completion of the space $\bigcup_{n=1}^{\infty}H\left(T_n\right)$, and let $\tilde{H}$ be the Hilbert completion of the space
$\bigcup_{n=1}^{\infty}H\left(S_n\right)$.  Since $f_n$ intertwines the $G(n)$-representations  $T_n$ and $T_{n+1}|_{G(n)}$, and  $\varrho_n$ intertwines the $G(n)$-representations $S_n$ and $S_{n+1}|_{G(n)}$,
the inductive limit $(T,H)$ of $\left(T_n,H\left(T_n\right)\right)$, and
the inductive limit $(S,\tilde{H})$ of $\left(S_n,H\left(S_n\right)\right)$ can be defined. Introduce the map
$F:\; H\rightarrow\tilde{H}$ by the condition that $F|_{H\left(T_n\right)}=F_n$. It remains to check that $F$ intertwines the representations $T$ and $S$.
This follows from the fact that for each $n$ the linear map $F_n$ intertwines $T_n$ and $S_n$.
\end{proof}
\subsection{The generalized regular representation of $G_{\infty}\times G_{\infty}$}
\label{SectionTz1z2}
Fix $t_1>0$, $\ldots$, $t_k>0$ and $z_1$, $\ldots$, $z_k\in\C$ such that $|z_1|^2=t_1$, $\ldots$, $|z_k|^2=t_k$, and set
\begin{equation}\label{6.1.1}
\left(T_{z_1,\ldots,z_k}\left(W\right)f\right)(x)=f(xW)\;z_1^{C_1\left(x,W\right)}\ldots z_k^{C_k\left(x,W\right)},
\end{equation}
where $f\in L^{2}\left(\mathfrak{S}_{G}, P_{t_1,\ldots,t_k}^{\Ewens}\right)$, $W\in G_{\infty}\times G_{\infty}$,
the action of $G_{\infty}\times G_{\infty}$ on $\mathfrak{S}_G$ is defined as it is described in Section \ref{SubSectionWreathInfty}, and the functions $C_l\left(x,W\right)$ (where $l=1,\ldots,k$) are defined as in the statement of Theorem
\ref{THEOREMCocycles}, equation (\ref{Cocycles}). It is not hard to check that
$$
C_l\left(x,W_1W_2\right)=C_l\left(x,W_1\right)+C_l\left(xW_1,W_2\right)
$$
for each $l=1,\ldots,k$, each $x\in\mathfrak{S}_G$, and each $W_1, W_2\in G_{\infty}\times G_{\infty}$.
This implies that (\ref{6.1.1}) defines a representation of $G_{\infty}\times G_{\infty}$ in the space $L^2\left(\mathfrak{S}_G, P_{t_1,\ldots,t_k}^{\Ewens}\right)$.  Proposition \ref{Proposition5.1.1} can be applied to show that
$T_{z_1,\ldots,z_k}$ are unitary operators in $L^2\left(\mathfrak{S}_G, P_{t_1,\ldots,t_k}^{\Ewens}\right)$.
The representation $\left(T_{z_1,\ldots,z_k}, L^2\left(\mathfrak{S}_G, P_{t_1,\ldots,t_k}^{\Ewens}\right)\right)$ is called the \textit{generalized regular representation} of $G_{\infty}\times G_{\infty}$.

\subsection{The generalized regular representation as an inductive limit of the two-sided regular representations of $G_n\times G_{n}$}\label{Section6}
Recall that $G_n$ denotes the finite wreath product $G\sim S(n)$.
Let $\mu_n^{\Uniform}$ be the uniform probability measure on $G_n$. The two-sided regular representation of $G_n\times G_n$ is defined by
$$
\left(\Reg^n(g)f\right)(x)=f\left(g_2^{-1}xg_1\right),
$$
where $g=\left(g_1,g_2\right)\in G_n\times G_n$, $x\in G_n$, $f\in L^2\left(G_n,\mu_n^{\Uniform}\right)$.

In the terminology of Section \ref{SectionInductive Limits}, the Hilbert space $L^2\left(G_n,\mu_n^{\Uniform}\right)$ corresponds to $H\left(S_n\right)$, and the two-sided regular representation $\Reg^n$ of $G_n\times G_n$ corresponds to $S_n$.

Given $n=1,2,\ldots$, consider the canonical projection
$p_n: \mathfrak{S}_{G}\longrightarrow G_n$. A function $F=f\circ p_n$, where $f$ is
any function on $G_n$, is called a cylinder function of level $n$ on the space $\mathfrak{S}_{G}$. Denote the set of such functions by $\Cyl^n\left(\mathfrak{S}_{G}\right)$. Note that $\Cyl^n\left(\mathfrak{S}_{G}\right)$ is a subspace of $L^2\left(\mathfrak{S}_G, P^{\Ewens}_{t_1,\ldots,t_k}\right)$.
It is important that the image of $P_{t_1,\ldots,t_k}^{\Ewens}$ with respect to $p_n$ coincides with $P_{t_1,\ldots,t_k;n}^{\Ewens}$, as it follows from Proposition \ref{PropositionProjection}, and from the fact that
$P_{t_1,\ldots,t_k}^{\Ewens}$ is the projective limit measure,
$P_{t_1,\ldots,t_k}^{\Ewens}
=\underset{\longleftarrow}{\lim}P_{t_1,\ldots,t_k;n}^{\Ewens}$. This enables to identify $\Cyl^n\left(\mathfrak{S}_{G}\right)$ and $L^2\left(G_n, P^{\Ewens}_{t_1,\ldots,t_k;n}\right)$.

This representation $T_{z_1,\ldots,z_k}\biggl|_{L^2\left(G_n,P_{t_1,\ldots,t_k;n}^{\Ewens}\right)}$
is defined by
$$
\left(T_{z_1,\ldots,z_k}(g)\biggl|_{L^2\left(G_n,P_{t_1,\ldots,t_k;n}^{\Ewens}\right)}f\right)(x)=f(xg)\;z_1^{[xg]_{c_1}-[x]_{c_1}}\ldots z_k^{[xg]_{c_k}-[x]_{c_k}},
$$
where $f\in L^2\left(G_n,P_{t_1,\ldots,y_k;n}^{\Ewens}\right)$, $g=\left(g_1,g_2\right)\in G_n\times G_n$, $x\in G_n$. The representation
$T_{z_1,\ldots,z_k}\biggl|_{L^2\left(G_n,P_{t_1,\ldots,y_k;n}^{\Ewens}\right)}$ will play a role of $T_n$ in Section \ref{SectionInductive Limits}, and the Hilbert space
$L^2\left(G_n,P_{t_1,\ldots,y_k;n}^{\Ewens}\right)$ will correspond to $H\left(T_n\right)$.

Introduce  $F_{z_1,\ldots,z_k}^n$,
$$
F_{z_1,\ldots,z_k}^n:\;L^2\left(G_n,P_{t_1,\ldots,t_k;n}^{\Ewens}\right)\longrightarrow L^2\left(G_n,\mu_n^{\Uniform}\right),
$$
as the operator of multiplication by the function
\begin{equation}\label{3.7.1}
F_{z_1,\ldots,z_k}^n(x)=\frac{\left(n!\right)^{\frac{1}{2}}}{\left[\left(\frac{t_1}{\zeta_{c_1}}+\ldots+\frac{t_k}{\zeta_{c_k}}\right)_n\right]^{\frac{1}{2}}}
z_1^{[x]_{c_1}}z_2^{[x]_{c_2}}\ldots z_k^{[x]_{c_k}},\; x\in G_n,
\end{equation}
where $\zeta_{c_l}$ is defined as in Section \ref{SectionFiniteWreathProduct}.
\begin{prop}\label{Proposition3.7.1} The operator $F^n_{z_1,\ldots,z_k}$ is an isometry. Moreover, this operator intertwines the representations
$T_{z_1,\ldots,z_k}\biggl|_{L^2\left(G_n,P_{t_1,\ldots,t_k;n}^{\Ewens}\right)}$ and $Reg^n$ of $G_n\times G_n$.
\end{prop}
\begin{proof} The proof is straightforward. Namely, the properties of $F_{z_1,\ldots,z_k}^{n}$ stated in Proposition \ref{Proposition3.7.1}
follow from the very definitions of the corresponding operator and the representations.
\end{proof}
 Define the operator $L_{z_1,\ldots,z_k}^n$,
$$
L_{z_1,\ldots,z_k}^n:\;  L^2\left(G_n,\mu_n^{\Uniform}\right)\longrightarrow L^2\left(G_{n+1},\mu_{n+1}^{\Uniform}\right),
$$
as
\begin{equation}\label{3.3.5}
\left(L_{z_1,\ldots,z_k}^n\psi\right)(x)=\sqrt{\frac{n+1}{\frac{t_1}{\zeta_{c_1}}+\ldots+\frac{t_k}{\zeta_{c_k}}}}\;\;
z_1^{[x]_{c_1}-\left[p_{n,n+1}\left(x\right)\right]_{c_1}}\ldots z_k^{[x]_{c_k}-\left[p_{n,n+1}\left(x\right)\right]_{c_k}}
\psi\left(p_{n,n+1}(x)\right),
\end{equation}
where  $\psi\in L^2\left(G_n,\mu_n^{\Uniform}\right)$, $x\in G_{n+1}$.
It is not hard to check that
for any non-zero complex numbers $z_1$, $\ldots$, $z_k$ the operator $L_{z_1,\ldots,z_k}^n$ is an isometric embedding of $L^2\left(G_n,\mu_n^{\Uniform}\right)$ into $L^2\left(G_{n+1},\mu_{n+1}^{\Uniform}\right)$,
which intertwines the $G_n\times G_n$-representations $\Reg^n$ and $\Reg^{n+1}\biggl|_{G_n\times G_n}$.
\begin{thm} \label{THEOREM6.3}
Let $T_{z_1,\ldots,z_k}'$ denote the inductive limit of the representations $\Reg^n$ with respect to the embedding
$$
 L^2\left(G_1,\mu_1^{\Uniform}\right)\overset{L_{z_1,\ldots,z_k}^1}{\longrightarrow} L^2\left(G_2,\mu_2^{\Uniform}\right)\overset{L_{z_1,\ldots,z_k}^2}{\longrightarrow}\ldots
$$
Then the representations $T_{z_1,\ldots,z_k}'$ and $T_{z_1,\ldots,z_k}$ are equivalent. In other words, the generalized regular representation $T_{z_1,\ldots,z_k}$ is equivalent to the inductive limit of the two-sided regular representations of $G_n\times G_{n}$.
\end{thm}
\begin{proof} Consider the diagram shown on Fig. \ref{Fig.6},
\begin{figure}[t]
{\scalebox{0.5}{\includegraphics{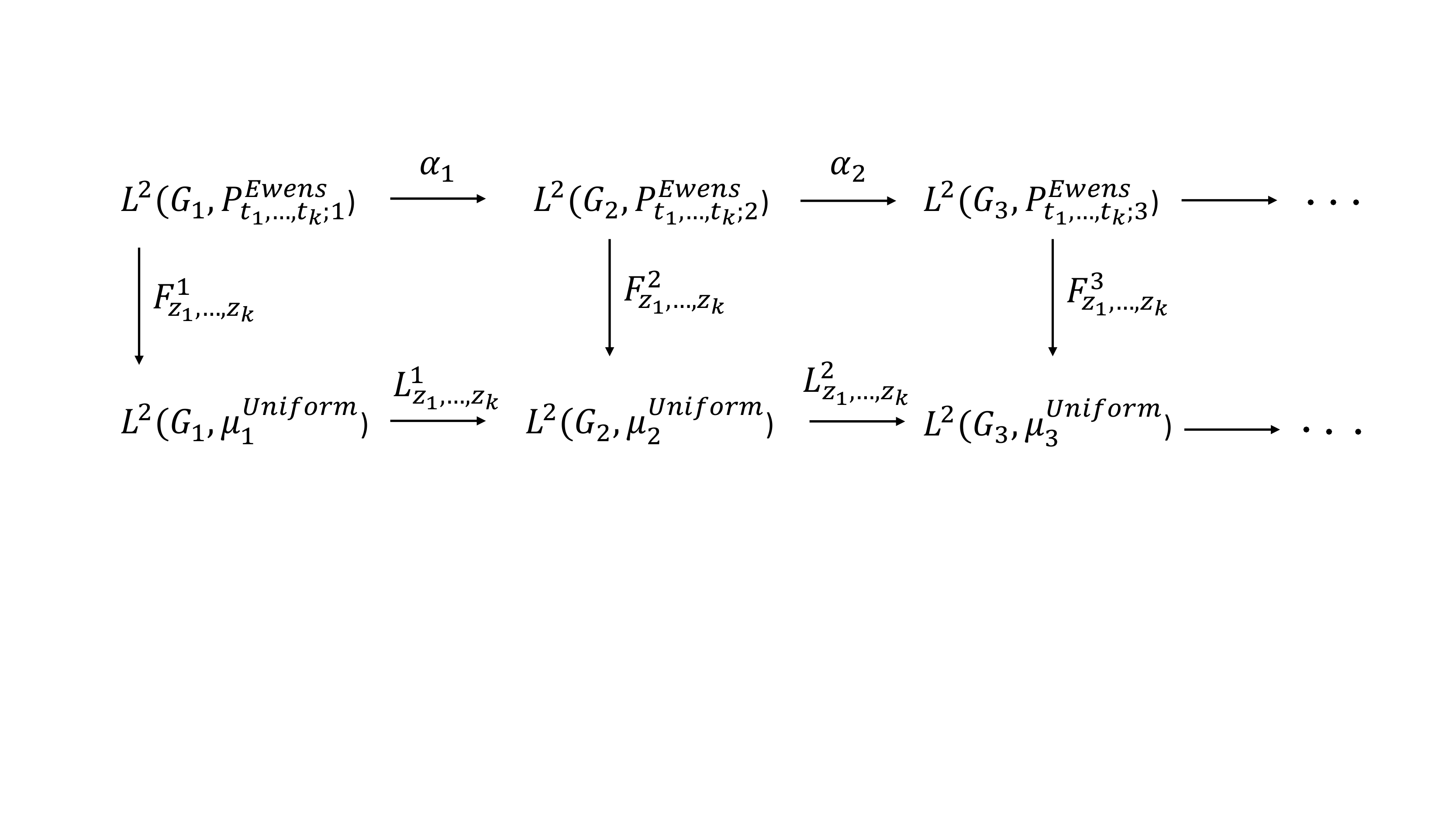}}}
\caption{The equivalence of inductive limits}
\label{Fig.6}
\end{figure}
where the map
$$
\alpha_{n}:\; L^2\left(G_n,P_{t_1,\ldots,t_k;n}^{\Ewens}\right)\longrightarrow
L^2\left(G_{n+1},P_{t_1,\ldots,t_k;n+1}^{\Ewens}\right)
$$
is an isometric embedding defined by
$$
\left(\alpha_n\varphi\right)(x)=\varphi\left(p_{n,n+1}(x)\right),\;\; x\in G_{n+1},\;\;\varphi\in L^2\left(G_n, P_{t_1,\ldots,t_k;n}^{\Ewens}\right).
$$
Using the definition of $T_{z_1,\ldots,z_k}$ it is straightforward to verify that $\alpha_n$ intertwines the $G_n\times G_n$-representations  $T_{z_1,\ldots,z_k}\biggl|_{L^2\left(G_n,P_{t_1,\ldots,t_k;n}^{\Ewens}\right)}$ and
$T_{z_1,\ldots,z_k}\biggl|_{L^2\left(G_{n+1},P_{t_1,\ldots,t_k;n+1}^{\Ewens}\right)}$.

Recall
that we identify the Hilbert spaces $L^2\left(G_n, P_{t_1,\ldots,t_k;n}^{\Ewens}\right)$ and
$L^2\left(\Cyl^n\left(\mathfrak{S}_G\right), P_{t_1,\ldots,t_k}^{\Ewens}\right)$, the last one is the subspace of  $L^2\left(\mathfrak{S}_G, P_{t_1,\ldots,t_k}^{\Ewens}\right)$. Moreover,
$\bigcup_{n=1}^{\infty}L^2\left(\Cyl^n\left(\mathfrak{S}_G\right), P_{t_1,\ldots,t_k}^{\Ewens}\right)$ is a dense subset of $L^2\left(\mathfrak{S}_G, P_{t_1,\ldots,t_k}^{\Ewens}\right)$.
It follows that the representation
$\left(T_{z_1,\ldots,z_k}, L^2\left(\mathfrak{S}_{G}, P_{t_1,\ldots,t_k}^{\Ewens}\right)\right)$ defined in Section \ref{SectionTz1z2} can be understood as the inductive limit of the representations
$$
T_{z_1,\ldots,z_k}\biggl|_{ L^2\left(G_1, P_{t_1,\ldots,t_k;1}^{\Ewens}\right)},\;\;T_{z_1,\ldots,z_k}\biggl|_{ L^2\left(G_2, P_{t_1,\ldots,t_k;2}^{\Ewens}\right)},\ldots .
$$
Now we use Proposition \ref{PropositionEqvInd} to conclude that this inductive limit is equivalent to that of $\left(\Reg^1, L^2\left(G_1,\mu_1^{\Uniform}\right)\right)$,
$\left(\Reg^2, L^2\left(G_2,\mu_2^{\Uniform}\right)\right)$, $\ldots$. Indeed, for each $n$ the linear map $F_{z_1,\ldots,z_k}^n$ intertwines the representations
$T_{z_1,\ldots,z_k}\biggl|_{ L^2\left(G_n, P_{t_1,\ldots,t_k;n}^{\Ewens}\right)}$ and $\Reg^n$, see Proposition \ref{Proposition3.7.1}. Moreover, the map
$$
L^n_{z_1,\ldots,z_k}:\;\; L^2\left(G_n,\mu_n^{\Uniform}\right)\longrightarrow L^2\left(G_{n+1},\mu_{n+1}^{\Uniform}\right)
$$
is defined by (\ref{3.3.5}) in such a way that the condition
$$
F_{z_1,\ldots,z_k}^{n+1}\circ\alpha_n
=L_{z_1,\ldots,z_k}^n\circ F_{z_1,\ldots,z_k}^n
$$
is satisfied, and the $n$th block of the diagram shown on Fig. \ref{Fig.6}
is commutative. In addition,  $L^n_{z_1,\ldots,z_k}$ defines an isometric embedding of $L^2\left(G_n,\mu_n^{\Uniform}\right)$ into $L^2\left(G_{n+1},\mu_{n+1}^{\Uniform}\right)$.  Thus  Proposition \ref{PropositionEqvInd}
implies that the inductive limit $T_{z_1,\ldots,z_k}'$ of the representations
$\Reg^n$ with respect to the embedding
$$
L^2\left(G_1,\mu_1^{\Uniform}\right)\overset{L^1_{z_1,\ldots,z_k}}{\longrightarrow} L^2\left(G_2,\mu_2^{\Uniform}\right)
\overset{L^2_{z_1,\ldots,z_k}}{\longrightarrow}\ldots
$$
is well-defined, and it is equivalent to $T_{z_1,\ldots,z_k}$.
\end{proof}
\section{Characters and spherical functions}
We wish to introduce and study the character of $T_{z_1,\ldots,z_k}$, which is a  representation of $G_{\infty}\times G_{\infty}$.
For  representations of groups like $G_{\infty}\times G_{\infty}$
 the conventional definition of characters is not applicable. However, as in the case of $S(\infty)\times S(\infty)$ this difficulty can be overcome.

 It is well known that for any finite group $K$, the pair
 $\left(K\times K,\diag (K)\right)$ is a Gelfand pair. In particular,
 this is true for
 $\left(G_n\times G_n,\diag\left(G_n\right)\right)$, where $G_n$ is the wreath product of the symmetric group $S(n)$ with a finite group $G$. The group $G_{\infty}$ is the union of an ascending chain of finite subgroups $G_n$. Proposition
8.15 and Corollary 8.16 in Borodin and Olshanski \cite{BorodinOlshanskiBook}
imply that $\left(G_{\infty}\times G_{\infty},\diag\left(G_{\infty}\right)\right)$
is a Gelfand pair (in the sense of Olshanski \cite{Olshanski}), which enables us
to use the language of spherical representations, and of spherical functions.
For a background on this material we refer the reader to Borodin and Olshanski
\cite{BorodinOlshanskiBook}, Section 8, and to references therein.  Below
we recall several definitions and facts needed in this work.

\subsection{Spherical representations of Gelfand pairs}
 Let $\mathcal{G}$ be a group, and $\mathcal{K}$ be a subgroup of $\mathcal{G}$.
Assume that $\left(\mathcal{G},\mathcal{K}\right)$ is a Gelfand pair in the sense of Olshanski \cite{Olshanski}.
\begin{defn}
A pair $\left(T,\zeta\right)$  where $T$ is a unitary representation of $\mathcal{G}$ acting in a Hilbert space $H(T)$, and $\zeta$ is a unit vector in $H(T)$ is called a \textit{spherical representation} of  $\left(\mathcal{G},\mathcal{K}\right)$
if the following conditions are satisfied:\\
(a) $\zeta$ is $\mathcal{K}$-invariant.\\
(b) $\zeta$ is cyclic, i.e. the span of vectors of the form $T(g)\zeta$, where $g\in \mathcal{G}$, is dense in $H(T)$.
In this case $\zeta$ is called the \textit{spherical vector}.
\end{defn}
Spherical representations $\left(T_1,\zeta_1\right)$ and $\left(T_2,\zeta_2\right)$
of $\left(\mathcal{G},K\right)$ are called \textit{equivalent} if
there exists an  isometric isomorphism between the  Hilbert spaces
$H\left(T_1\right)$ and $H\left(T_2\right)$ which commutes with the action of
$\mathcal{G}$  and preserves the spherical vectors, see Olshanski \cite{OlshanskiNato}, Section 2.1.
 A spherical representation $\left(T,\zeta\right)$ is called \textit{irreducible} if $(T, H(T))$ is an irreducible representation of $\mathcal{G}$.

Let  $(T,\zeta)$ be a spherical representation of a Gelfand pair $\left(\mathcal{G}, \mathcal{K}\right)$, and let $\left<.,.\right>$ denote the inner product in $H(T)$.
The function $\varphi(g)=\left<T(g)\zeta,\zeta\right>$, where $g\in \mathcal{G}$, is called the \textit{spherical function} of $\left(T,\zeta\right)$.
The  spherical function of $\left(T,\zeta\right)$ is called irreducible,
if  $\left(T,\zeta\right)$ is irreducible.
It is not hard to check that two spherical representations are equivalent if and only if their spherical functions are coincide.
If $T$ is irreducible, then $\zeta$ is unique (within a scalar factor $\alpha\in\C$, $|\alpha|=1$).
\subsection{The problem of harmonic analysis}\label{SectionPHAGENERAL}
The problem of harmonic analysis on a general Gelfand pair $\left(\mathcal{G},\mathcal{K}\right)$ can be formulated as follows. A function $\varphi$ on $\mathcal{G}$ is called positive definite if for any $m=1,2$, $\ldots$, and $g_1,\ldots, g_m\in \mathcal{G}$ the matrix
$\left[\varphi\left(g_j^{-1}g_i\right)\right]_{i,j=1}^m$ is Hermitian and positive definite.
Denote by $\Phi_1\left(\mathcal{G}//\mathcal{K}\right)$ the set of all normalized positive definite functions on $\mathcal{G}$ which are $\mathcal{K}$-biinvariant. It is known that $\Phi_1\left(\mathcal{G}//\mathcal{K}\right)$ coincides with the set of all spherical functions of
$\left(\mathcal{G},\mathcal{K}\right)$. Moreover, the space $\Phi_1\left(\mathcal{G}//\mathcal{K}\right)$ is convex, and the irreducible spherical functions are precisely extreme points of the convex set $\Phi_1\left(\mathcal{G}//\mathcal{K}\right)$.

Let  $\varphi\in\Phi_1\left(\mathcal{G}//\mathcal{K}\right)$. The problem
of harmonic analysis on $\left(\mathcal{G},\mathcal{K}\right)$ is to represent  $\varphi$ in terms of the extreme points of the convex set $\Phi_1\left(\mathcal{G}//\mathcal{K}\right)$.
The case where $\varphi$ is a spherical function of some natural  spherical representation
of $\left(\mathcal{G},\mathcal{K}\right)$ is of special interest.

In the particular case of the Gelfand pair $\left(\mathcal{G}\times \mathcal{G},\diag\left(\mathcal{G}\right)\right)$ the problem of harmonic analysis
is reduced to that of representation of a character of $\mathcal{G}$ in terms of extreme characters. To see this recall the definition of a character for an arbitrary group $\mathcal{G}$.
\begin{defn}\label{characters7}
A function $\chi:\; \mathcal{G}\rightarrow\C$ is called a character of $\mathcal{G}$ if it is positive definite, central\footnote{The centrality of $\chi$ means $\chi(g_1^{-1}gg_1)=\chi(g)$ for any $g, g_1 \in \mathcal{G}$.},  and normalized at the unity element.
\end{defn}

Denote by $\mathcal{X}\left(\mathcal{G}\right)$ the set of all characters of $\mathcal{G}$. Note that
$\mathcal{X}\left(\mathcal{G}\right)$ is a convex set.

Now, assume that $\left(\mathcal{G}\times\mathcal{G},\diag(\mathcal{G})\right)$ is a Gelfand pair.
It can be shown (see, for example, Proposition 8.19 in Borodin and Olshanski \cite{BorodinOlshanskiBook}) that there is an isomorphism
$\chi\longleftrightarrow\varphi$ between $\mathcal{X}\left(\mathcal{G}\right)$ and
$\Phi_1\left(\left(\mathcal{G}\times \mathcal{G}\right)//\diag(\mathcal{G})\right)$. Let
$\left(T,\zeta\right)$ be a spherical representation of
$\left(\mathcal{G}\times \mathcal{G},\diag\left(\mathcal{G}\right)\right)$,
and let $\varphi(g_1,g_2)=\left<T\left(g_1,g_2\right)\zeta,\zeta\right>$
be the  spherical function of  $\left(T,\zeta\right)$. Then the correspondence $\chi\longleftrightarrow\varphi$ leads to the formula
\begin{equation}\label{MagenDavid}
\chi(g)=\left<T\left(g,e\right)\zeta,\zeta\right>,\; g\in \mathcal{G},
\end{equation}
where $e$ is the unit element of $\mathcal{G}$.

Under the isomorphism  between $\Phi_1\left(\left(\mathcal{G}\times \mathcal{G}\right)//\diag(\mathcal{G})\right)$ and  $\mathcal{X}\left(\mathcal{G}\right)$ the extreme points of the convex set $\Phi_1\left(\left(\mathcal{G}\times \mathcal{G}\right)//\diag(\mathcal{G})\right)$ correspond to
the extreme points of the convex
set $\mathcal{X}\left(\mathcal{G}\right)$ (called the extreme characters).
Therefore, in the particular case of $\left(\mathcal{G}\times\mathcal{G} ,\diag(\mathcal{G})\right)$,
\textit{the problem of harmonic analysis  is  to express a character $\chi$ of $\mathcal{G}$  in terms of the extreme characters of $\mathcal{G}$}.

We note that the extreme characters of $\mathcal{G}$ are associated with the irreducible spherical functions of the Gelfand pair $\left(\mathcal{G}\times\mathcal{G},\diag(\mathcal{G})\right)$. It follows that
the irreducible spherical representations of the Gelfand pair $\left(\mathcal{G}\times\mathcal{G},\diag(\mathcal{G})\right)$ (up to equivalence) are parameterized by the extreme characters of $\mathcal{G}$.
\subsection{The character $\chi_{z_1,\ldots,z_k}$ of $T_{z_1,\ldots,z_k}$}
Recall that the representation $T_{z_1,\ldots,z_k}$ acts in the Hilbert space $L^2\left(\mathfrak{S}_{G},P_{t_1,\ldots,t_k}^{\Ewens}\right)$.
Denote by $\zeta_0$ the function $f_0\equiv 1$ from this space. We check that $\zeta_0$ is $\diag\left(G_{\infty}\right)$-invariant, and has norm $1$.
Thus $\left(T_{z_1,\ldots,z_k},\zeta_0\right)$ is a spherical representation of the Gelfand pair
$\left(G_{\infty}\times G_{\infty},\diag\left(G_{\infty}\right)\right)$.
Let $\Phi_{z_1,\ldots,z_k}$ be the spherical function of $T_{z_1,\ldots,z_k}$
corresponding to this vector,
\begin{equation}
\Phi_{z_1,\ldots,z_k}(g)=\left\langle T_{z_1,\ldots,z_k}(g)\zeta_0,\zeta_0\right\rangle_{L^2\left(\mathfrak{S}_{G},P_{t_1,\ldots,t_k}^{\Ewens}\right)},\;\; g\in G_{\infty}\times G_{\infty}.
\end{equation}
The character $\chi_{z_1,\ldots,z_k}$ of $T_{z_1,\ldots,z_k}$ is defined by
\begin{equation}\label{CharacterT}
\chi_{z_1,\ldots,z_k}(x)=\Phi_{z_1,\ldots,z_k}(x,e)=\left\langle T_{z_1,\ldots,z_k}(x,e)\zeta_0,\zeta_0\right\rangle_{L^2\left(\mathfrak{S}_{G},P_{t_1,\ldots,t_k}^{\Ewens}\right)},\;\; x\in G_{\infty},
\end{equation}
in accordance with the general formula (\ref{MagenDavid}).
As it is explained in Section \ref{SectionPHAGENERAL} the function
$$
\chi_{z_1,\ldots,z_k}:\; G_{\infty}\longrightarrow\C
$$
defined by equation (\ref{CharacterT}) is a character of $G_{\infty}$
(in the sense of Definition \ref{characters7}). The problem of harmonic analysis on the Gelfand pair $\left(G_{\infty}\times G_{\infty},\diag\left(G_{\infty}\right)\right)$
considered below is to represent $\chi_{z_1,\ldots,z_k}$ in terms of the extreme characters of
$G_{\infty}$.
\section{A formula for $\chi_{z_1,\ldots,z_k}$}
In this Section we derive  a formula for $\chi_{z_1,\ldots,z_k}$. For this purpose we introduce the following notation.
Let $\lambda$ be a Young diagram, and assume that the box $\Box$ of $\lambda$ is situated on the intersection of the $i$th row and the $j$th column of $\lambda$. Then $c(\Box)=j-i$ is the content of the box $\Box$, and
$h(\Box)=\lambda_i-i+\lambda_j'-j+1$ is the hook-length of $\Box$ in $\lambda$.

Assume that $z\in\C\setminus\{0\}$. Set
\begin{equation}\label{StandardZMeasure}
M_z^{(n)}\left(\lambda\right)=\frac{n!}{\left(z\bar{z}\right)_n}\prod\limits_{\Box\in\lambda}
\frac{(z+c(\Box))(\bar{z}+c(\Box))}{h^2(\Box)},\;\; |\lambda|=n,
\end{equation}
where $|\lambda|$ denotes the number of boxes in $\lambda$.
This object, $M_z^{(n)}$, is called the $z$-measure with the parameter $z$, and it is a probability measure on the set $\Y_n$ of all Young diagrams with $n$ boxes, see Kerov, Olshanski, and Vershik \cite{KerovOlshanskiVershik}, Theorem 4.1.1.

Recall that the irreducible representations of $G_n=G\sim S(n)$
are parameterized by multiple partitions $\Lambda_n^{(k)}$,
$\Lambda_n^{(k)}\in\Y_n^{(k)}$, where
$\Y_n^{(k)}$ is the set of multiple partitions of $n$ into $k$ components introduced in Section \ref{SubSectionMPS}.
Denote by $\chi^{\Lambda_n^{(k)}}$ the character of the irreducible representation parameterized by $\Lambda_n^{(k)}$.
Also, $\DIM\left(\Lambda_n^{(k)}\right)$ denotes the dimension of the irreducible representation parameterized by $\Lambda_n^{(k)}$.
\begin{thm}\label{TheoremMultpleZmeasures} For $n=1,2,\ldots$
\begin{equation}\label{SPHERICALFUNCTIONZMEASURES}
\chi_{z_1,\ldots,z_k}\biggl|_{G\sim S(n)}
=\underset{\Lambda_n^{(k)}\in\Y_n^{(k)}}{\sum}
M_{z_1,\ldots,z_k}^{(n)}\left(\Lambda_n^{(k)}\right)
\frac{\chi^{\Lambda_n^{(k)}}}{\DIM\left(\Lambda_n^{(k)}\right)},
\end{equation}
Here $M_{z_1,\ldots,z_k}^{(n)}$ are probability measures on the set
$\Y_n^{(k)}$.
These measures can be expressed in terms of the $z$-measures $M_{a_1}^{(|\lambda^{(1)}|)}$, $\ldots$,
$M_{a_k}^{(|\lambda^{(k)}|)}$ defined by equation (\ref{StandardZMeasure}) as
\begin{equation}\label{MultipleZmeasures}
\begin{split}
M_{z_1,\ldots,z_k}^{(n)}\left(\lambda^{(1)},\ldots,\lambda^{(k)}\right)
&=
\frac{n!}{\left(a_1\bar{a}_1+\ldots+a_k\bar{a}_k\right)_n}
\frac{(a_1\bar{a}_1)_{|\lambda^{(1)}|}}{|\lambda^{(1)}|!}
\ldots
\frac{(a_k\bar{a}_k)_{|\lambda^{(k)}|}}{|\lambda^{(k)}|!}\\
&\times M_{a_1}^{(|\lambda^{(1)}|)}\left(\lambda^{(1)}\right)\ldots
M_{a_k}^{(|\lambda^{(k)}|)}\left(\lambda^{(k)}\right).
\end{split}
\end{equation}
The parameters $a_1$, $\ldots$, $a_k$ can be written in terms of $z_1$, $\ldots$, $z_k$ as follows. Let $G_{\ast}=\left\{c_1,\ldots,c_k\right\}$ be the set of conjugacy classes in $G$,
and $G^{\ast}=\left\{\gamma^1,\ldots,\gamma^k\right\}$ be the set of the irreducible characters of $G$.
 Then
\begin{equation}\label{Parametersa}
a_l=\sum\limits_{i=1}^k\frac{z_i}{\zeta_{c_i}}\;\overline{\gamma^l(c_i)},\;\; l=1,\ldots,k,
\end{equation}
where $\zeta_{c_i}=\frac{|G|}{|c_i|}$.
\end{thm}
The measures $M^{(n)}_{z_1,\ldots,z_k}$ defined by equation (\ref{MultipleZmeasures}) are called the \textit{multiple $z$-measures}.
The fact that $M^{(n)}_{z_1,\ldots,z_k}$ is a probability measure on $\Y_n^{(k)}$ follows immediately from equation (\ref{SPHERICALFUNCTIONZMEASURES}). Alternatively, this fact can be checked directly using formula (\ref{MultipleZmeasures}).

In order to prove Theorem \ref{TheoremMultpleZmeasures} we need different facts from the theory of symmetric functions related to
representation theory of the finite wreath products $G_n=G\sim S(n)$. We collect these
results in the next Section.
\subsection{Symmetric functions and  characters of $G\sim S(n)$}
The basic reference for this Section is Macdonald \cite{Macdonald}, Appendix B.
\subsubsection{The algebra $\Sym\left(G\right)$ of  symmetric functions.}\label{SectionAlgebraSymmetricFunctions}
For each conjugacy class $c_1$, $\ldots$, $c_k$ of $G$ we assign a sequence of variables, namely we assign
the sequence $\left(x_{ic_1}\right)_{i\geq 1}$ for the conjugacy class $c_1$, $\ldots$, the sequence $\left(x_{ic_k}\right)_{i\geq 1}$ for the conjugace class $c_k$.
Denote by $p_r(c_1)$, $\ldots$, $p_r(c_k)$ the $r$th power symmetric functions in variables
$\left(x_{ic_1}\right)_{i\geq 1}$, $\ldots$, $\left(x_{ic_k}\right)_{i\geq 1}$, respectively. The algebra $\Sym\left(G\right)$ is defined as that generated by
$\left(p_r(c_1)\right)_{r\geq 1}$, $\ldots$, $\left(p_r(c_k)\right)_{r\geq 1}$, i.e.
$$
\Sym\left(G\right)=\C\left[p_{r_1}(c_1),\ldots,p_{r_k}(c_k):\;\; r_1\geq 1,\ldots, r_k\geq 1\right].
$$
For each family $\Lambda=\left(\lambda^{(1)},\ldots,\lambda^{(k)}\right)$ of $k$ Young diagrams
$\lambda^{(1)},\ldots,\lambda^{(k)}$ we define
$$
p_{\Lambda}=p_{\lambda^{(1)}}\left(c_1\right)\ldots p_{\lambda^{(k)}}\left(c_k\right),
$$
where $p_{\lambda^{(j)}}\left(c_j\right)$ denotes the power symmetric function in variables
$\left(x_{ic_j}\right)_{i\geq 1}$ parameterized by the Young diagram $\lambda^{(j)}$.
It is known that the $p_{\Lambda}$ form a $\C$-basis of $\Sym(G)$.

In addition,  for each $r\geq 1$ define
\begin{equation}\label{1.1}
p_r\left(\gamma^l\right)=\sum\limits_{i=1}^k\zeta^{-1}_{c_i}\gamma^l\left(c_i\right)p_r\left(c_i\right),\;\; l\in\left\{1,\ldots,k\right\}.
\end{equation}
where  $\gamma^l\left(c_i\right)$ denotes the value of the irreducible character $\gamma^l$ on the conjugacy class $c_i$. The functions $p_r\left(\gamma^l\right)$ are algebraically independent and generate $\Sym\left(G\right)$ as $\C$-algebra
$$
\Sym\left(G\right)=\C\left[p_{r_1}(\gamma^1),\ldots,p_{r_k}(\gamma^k):\;\; r_1\geq 1,\ldots, r_k\geq 1\right].
$$
The orthogonality of the irreducible characters $\gamma^l$,  and equation (\ref{1.1}) imply
\begin{equation}\label{1.3}
p_r\left(c_i\right)=\sum\limits_{l=1}^k\overline{\gamma^l\left(c_i\right)}p_r\left(\gamma^l\right),\;\; r\geq 1.
\end{equation}
Equation (\ref{1.3}) is called \textit{the change of variables formula}.
The Schur functions $s_{\mu}\left(\gamma^l\right)$ can be introduced by the formula
\begin{equation}
s_{\mu}\left(\gamma^l\right)=\sum\limits_{\varrho\vdash n}z_{\varrho}^{-1}\chi_{\varrho}^{\mu}p_{\varrho}\left(\gamma^l\right),
\end{equation}
where $\chi_{\varrho}^{\mu}$ is the value of the irreducible character
of the symmetric group $S(n)$ parameterized by the Young diagram $\mu$ at elements in the conjugacy class $C_{\varrho}$ of $S(n)$, $z_{\varrho}=\frac{n!}{\left|C_{\varrho}\right|}$, and $p_{\varrho}\left(\gamma^l\right)=p_{\varrho_1}\left(\gamma^l\right)p_{\varrho_2}\left(\gamma^l\right)\ldots$.

\subsubsection{The characteristic map}\label{SectionAlgebraSymmetricFunctions1}
Denote by $R\left(G\sim S(n)\right)$ the complex vector space spanned by $G^{\ast}$. In this space introduce a hermitian scalar product
$$
\left<u,v\right>_{R\left(G\sim S(n)\right)}=\frac{1}{n!}
\sum\limits_{x\in G\sim S(n)}u(x)\overline{v(x)}.
$$
If $f\in R\left(G\sim S(n)\right)$, then $\ch(f)$ is defined by
\begin{equation}
\ch(f)=\left\langle f,\psi\right\rangle_{G\sim S(n)}=\frac{1}{n!|G|^n}\sum\limits_{x\in G\sim S(n)}f(x)\overline{\psi(x)},
\end{equation}
where
$$
\psi(x)=p_{\Lambda_n^{(k)}}\left(c_1,\ldots,c_k\right)=p_{\lambda^{(1)}}(c_1)\ldots p_{\lambda^{(k)}}(c_k),
$$
provided that $x\in G\sim S(n)$ belongs to the conjugacy class parameterized by $\Lambda_n^{(k)}=\left(\lambda^{(1)},\ldots,\lambda^{(k)}\right)$. If $f_{\left(\lambda^{(1)},\ldots,\lambda^{(k)}\right)}$
is the value of $f$ at elements of the conjugacy class parameterized by $\Lambda_n^{(k)}=\left(\lambda^{(1)},\ldots,\lambda^{(k)}\right)$, then
\begin{equation}
\begin{split}
&\ch(f)=\sum\limits_{\Lambda_n^{(k)}\in\Y_n^{(k)}}\;\;\frac{1}{\prod\limits_{l=1}^k\left(\prod\limits_{j=1}^n
j^{r_j(\lambda^{(l)})}r_j(\lambda^{(l)})!\right)}
\frac{1}{\prod\limits_{l=1}^k\zeta_{c_l}^{r_1(\lambda^{(l)})+\ldots+r_n(\lambda^{(l)})}}\\
&\times f_{\left(\lambda^{(1)},\ldots,\lambda^{(k)}\right)}p_{\lambda^{(1)}}(c_1)\ldots p_{\lambda^{(k)}}(c_k),
\end{split}
\end{equation}
where $\zeta_{c_l}$ and $r_j\left(\lambda^{(l)}\right)$ are defined in Section
\ref{SectionFiniteWreathProduct}.

The map $\ch:\; R\left(G\sim S(n)\right)\longrightarrow \Sym\left(G\right)$ is called the characteristic map.
Set
$$
R\left(G\right)=\bigoplus\limits_{n\geq 0}R\left(G\sim S(n)\right),
$$
and define in $R\left(G\right)$ a scalar product by
$$
\left<f,g\right>_{R\left(G\right)}=\sum\limits_{n\geq 0}\left<f_n,g_n\right>_{R\left(G\sim S(n)\right)},
$$
where $f=\sum\limits_nf_n$, $g=\sum\limits_n g_n$ with
$f_n\in R\left(G\sim S(n)\right)$ and $g\in R\left(G\sim S(n)\right)$.
A multiplication in $R\left(G\sim S(n)\right)$ can be introduced. With this multiplication, $R\left(G\right)$ turns into
a graded algebra.
The map $\ch$ gives rise to an isometric isomorphism of the graded algebras
$R\left(G\right)$ and $\Sym\left(G\right)$. In particular, we have
\begin{equation}
\ch\left(\chi^{\Lambda_n^{(k)}}\right)=s_{\lambda^{(1)}}\left(\gamma^{1}\right)
\ldots s_{\lambda^{(k)}}\left(\gamma^{k}\right),
\end{equation}
see Macdonald \cite{Macdonald}, Appendix B, equation (9.4).
\subsubsection{The Frobenius-type character formula for $G\sim S(n)$}
Denote by $\chi^{\Lambda_n^{(k)}}_{\tLambda_n^{(k)}}$ the irreducible character of $G\sim S(n)$ parameterized by the multiple partition
$\Lambda_n^{(k)}=\left(\lambda^{(1)},\ldots,\lambda^{(k)}\right)$ of $n$, and evaluated at the conjugacy class of $G\sim S(n)$ parameterized by the multiple partition $\tLambda_n^{(k)}=\left(\mu^{(1)},\ldots,\mu^{(k)}\right)$ of $n$.
The Frobenius-type character formula for $\chi^{\Lambda_n^{(k)}}_{\tLambda_n^{(k)}}$
is
\begin{equation}\label{TheFrobeniusCharacterFormula}
p_{\mu^{(1)}}\left(c_1\right)\ldots p_{\mu^{(k)}}\left(c_k\right)
=\sum\limits_{\Lambda_n^{(k)}\in\Y_n^{(k)}}
\overline{\chi^{\Lambda_n^{(k)}}_{\tLambda_n^{(k)}}}\;
s_{\lambda^{(1)}}\left(\gamma^1\right)
\ldots s_{\lambda^{(k)}}\left(\gamma^k\right),
\end{equation}
see Macdonald \cite{Macdonald}, Appendix B, $\S$ 9.
Here $p_{\mu^{(l)}}\left(c_l\right)$ are the power symmetric functions and $s_{\lambda^{(l)}}\left(\gamma^l\right)$ are the Schur functions introduced in Section \ref{SectionAlgebraSymmetricFunctions}. The Frobenius-type character
formula (equation (\ref{TheFrobeniusCharacterFormula})) enables to derive an explicit expression for the dimensions of irreducible representations of $G\sim S(n)$.
\begin{prop}Denote by $\DIM\left(\Lambda_n^{(k)}\right)$ the dimension of the irreducible representation of $G\sim S(n)$ parameterized by the multiple partition
$\Lambda_n^{(k)}=\left(\lambda^{(1)},\ldots,\lambda^{(k)}\right)$ of $n$. We have
\begin{equation}\label{2.1.2.1}
\DIM\left(\Lambda_n^{(k)}\right)=\frac{n!}{|\lambda^{(1)}|!\ldots|\lambda^{(k)}|!}\prod\limits_{l=1}^k\left(d_l\right)^{|\lambda^{(l)}|}\dim\lambda^{(l)},
\end{equation}
where $d_l$ is the value of the irreducible character $\gamma^l$ of $G$ at the unit element, and $\dim\lambda^{(l)}$ is the number of standard Young diagrams of shape $\lambda^{(l)}$.
\end{prop}
\begin{proof} See Macdonald \cite{Macdonald}, Appendix B, $\S$ 9.
\end{proof}
\subsection{Proof of Theorem \ref{TheoremMultpleZmeasures}}
\begin{lem}\label{LEMMAZCYCLES}Given  $z_1\in\C\setminus\{0\}$, $\ldots$, $z_k\in\C\setminus\{0\}$, the expansion of the central function
$$
x\longrightarrow z_1^{[x]_{c_1}}\ldots z_k^{[x]_{c_k}}
$$
defined on the group $G\sim S(n)$ in terms of the irreducible characters $\chi^{\Lambda_n^{(k)}}$, $\Lambda_n^{(k)}=\left(\lambda^{(1)},\ldots,\lambda^{(k)}\right)\in\Y_n^{(k)}$,
of $G\sim S(n)$
can be written as
\begin{equation}\label{EqMainLemmZ}
z_1^{[x]_{c_1}}\ldots z_k^{[x]_{c_k}}=
\sum\limits_{\Lambda_n^{(k)}\in\Y_n^{(k)}}
\left(\prod\limits_{\Box\in\lambda^{(1)}}\frac{a_1+c(\Box)}{h(\Box)}\right)
\ldots\left(\prod\limits_{\Box\in\lambda^{(k)}}\frac{a_k+c(\Box)}{h(\Box)}\right)
\chi^{\Lambda_n^{(k)}}(x),
\end{equation}
where the parameters $a_1$, $\ldots$, $a_k$ are defined by equation (\ref{Parametersa}).
\end{lem}
\begin{proof}
Let $\ch$ be the characteristic map introduced in Section \ref{SectionAlgebraSymmetricFunctions1}.
Then it is not hard  to check that the relation
\begin{equation}\label{LemExpans}
\begin{split}
&1+\sum\limits_{r\geq 1}\ch\left[z_1^{[x]_{c_1}}\ldots z_k^{[x]_{c_k}}\right]u^r\\
&=\exp\left[\left(\frac{z_1}{\zeta_{c_1}}\right)\sum\limits_{r=1}^{\infty}\frac{u^r}{r!}p_r(c_1)\right]\ldots
\exp\left[\left(\frac{z_k}{\zeta_{c_k}}\right)\sum\limits_{r=1}^{\infty}\frac{u^r}{r!}p_r(c_k)\right]
\end{split}
\end{equation}
between the formal power series is satisfied.  The right-hand side of the equation  above can be rewritten as
\begin{equation}\label{EXINLIM}
\exp\left[\sum\limits_{r=1}^{\infty}\left(\sum\limits_{i=1}^k\frac{z_i}{\zeta_{c_i}}p_r(c_i)\right)
\frac{u^r}{r!}\right].
\end{equation}
Taking into account the change of variables formula (equation (\ref{1.3})) we see that (\ref{EXINLIM}) can be represented in the form
\begin{equation}\label{EXINLIM1}
\exp\left[\sum\limits_{r=1}^{\infty}\left(\sum\limits_{j=1}^ka_jp_r(\gamma^j)\right)
\frac{u^r}{r!}\right],
\end{equation}
where  the parameters $a_1$, $\ldots$, $a_k$ are defined by  equation (\ref{Parametersa}).
It is known that
\begin{equation}
\exp\left[a_l\sum\limits_{r=1}^{\infty}p_r(\gamma^{l})\frac{u^r}{r!}\right]
=\sum\limits_{\lambda^{(l)}\in\Y}\left(\prod\limits_{\Box\in\lambda^{(l)}}\frac{a_l+c(\Box)}{h(\Box)}\right)s_{\lambda^{(l)}}(u\gamma^{l}),
\end{equation}
where $l=1,\ldots, k$, see, for example, Borodin and Olshanski
\cite{BorodinOlshanskiBook}, Section 11. This can be used together with  expansion (\ref{LemExpans}) to conclude that
\begin{equation}
\begin{split}
&\ch\left[z_1^{[x]_{c_1}}\ldots z_k^{[x]_{c_k}}\right]\\
&=
\underset{|\lambda^{(1)}|+\ldots+|\lambda^{(k)}|=n}{\sum\limits_{\lambda^{(1)},\ldots,\lambda^{(k)}}}
\left(\prod\limits_{\Box\in\lambda^{(1)}}\frac{a_1+c(\Box)}{h(\Box)}\right)
\ldots\left(\prod\limits_{\Box\in\lambda^{(k)}}\frac{a_k+c(\Box)}{h(\Box)}\right)
s_{\lambda^{(1)}}\left(\gamma^{1}\right)\ldots s_{\lambda^{(k)}}\left(\gamma^{k}\right).
\end{split}
\end{equation}
The last equation implies (\ref{EqMainLemmZ}).
\end{proof}
Now we are ready to complete the proof of
of Theorem \ref{TheoremMultpleZmeasures}. Recall that $\chi_{z_1,\ldots,z_k}$ is defined by equation (\ref{CharacterT}). The representation $T_{z_1,\ldots,z_k}$ acting in $L^2\left(\mathfrak{S}_G,P_{t_1,\ldots,t_k}^{\Ewens}\right)$ is the inductive limit of the representations
$$
T_{z_1,\ldots,z_k}\biggl|_{ L^2\left(G_1, P_{t_1,\ldots,t_k;1}^{\Ewens}\right)},\;\;T_{z_1,\ldots,z_k}\biggl|_{ L^2\left(G_2, P_{t_1,\ldots,t_k;2}^{\Ewens}\right)},\ldots ,
$$
and $\zeta_0$ can be regarded as an element of $L^2\left(G_n, P_{t_1,\ldots,t_k;n}^{\Ewens}\right)$ for each $n=1,2,\ldots$. Since
$P^{\Ewens}_{t_1,\ldots,t_k}$ is the projective limit measure,
$P^{\Ewens}_{t_1,\ldots,t_k}=\underset{\longleftarrow}{\lim}
P_{t_1,\ldots,t_k;n}^{\Ewens}$, we can write
\begin{equation}
\chi_{z_1,\ldots,z_k}\biggl|_{G\sim S(n)}(x)=\left\langle T_{z_1,\ldots,z_k}(x,e)\zeta_0,\zeta_0\right\rangle_{L^2\left(G_n,P_{t_1,\ldots,t_k;n}^{\Ewens}\right)},\;\; x\in G\sim S(n).
\end{equation}
Since $F^n_{z_1,\ldots,z_k}$ is an isometry, and it intertwines the representations
$T_{z_1,\ldots,z_k}\biggl|_{L^2\left(G_n, P_{t_1,\ldots,t_k;n}^{\Ewens}\right)}$
and $\Reg^n$ of $G_n\times G_n$ (see Proposition \ref{Proposition3.7.1}),
we obtain
\begin{equation}
\left\langle T_{z_1,\ldots,z_k}(x,e)\zeta_0,\zeta_0
\right\rangle_{L^2\left(G_n,P_{t_1,\ldots,t_k;n}^{\Ewens}\right)}
=\left\langle\Reg^n(x,e) F_{z_1,\ldots,z_k}^n\zeta_0, F_{z_1,\ldots,z_k}^{n}\zeta_0\right\rangle_{L^2\left(G_n,\mu_n^{\Uniform}\right)}.
\end{equation}
Recall that $\zeta_0$ represents the unit vector in $L^2\left(G_n,P_{t_1,\ldots,t_k;n}^{\Ewens}\right)$, so its image $F_{z_1,\ldots,z_k}^n\zeta_0$
in $L^2\left(G_n,\mu_n^{\Uniform}\right)$ is the function $F_{z_1,\ldots,z_k}^n(y)$ defined by equation (\ref{3.7.1}).  Thus we obtain
\begin{equation}
\chi_{z_1,\ldots,z_k}(x)=\frac{1}{|G|^nn!}\sum\limits_{y\in G\sim S(n)}F_{z_1,\ldots,z_k}(yx)
\overline{F_{z_1,\ldots,z_k}(y)},\;\; x\in G\sim S(n),
\end{equation}
or
\begin{equation}
\chi_{z_1,\ldots,z_k}(x)=
\frac{1}{\left|G\right|^n
\left(\frac{t_1}{\zeta_{c_1}}+\ldots+\frac{t_k}{\zeta_{c_k}}\right)_n}
\sum\limits_{y\in\; G\sim S(n)}z_1^{[xy]_{c_1}}z_2^{[xy]_{c_2}}\ldots z_k^{[xy]_{c_k}}
\overline{z_1^{[y]}z_2^{[y]}\ldots z_k^{[y]}}.
\end{equation}
Next  we use Lemma \ref{LEMMAZCYCLES} to rewrite this expression as
\begin{equation}
\begin{split}
&\frac{1}{\left|G\right|^n
\left(\frac{t_1}{\zeta_{c_1}}+\ldots+\frac{t_k}{\zeta_{c_k}}\right)_n}
\underset{|\lambda^{(1)}|+\ldots+|\lambda^{(k)}|=n}{\sum\limits_{\lambda^{(1)},\ldots,\lambda^{(k)}}}
\;
\underset{|\mu^{(1)}|+\ldots+|\mu^{(k)}|=n}{\sum\limits_{\mu^{(1)},\ldots,\mu^{(k)}}}
\biggl[\\
&\left(\prod\limits_{\Box\in\lambda^{(1)}}\frac{a_1+c(\Box)}{h(\Box)}\right)
\ldots\left(\prod\limits_{\Box\in\lambda^{(k)}}\frac{a_k+c(\Box)}{h(\Box)}\right)
\left(\prod\limits_{\Box\in\mu^{(1)}}\frac{\bar{a}_1+c(\Box)}{h(\Box)}\right)
\ldots\left(\prod\limits_{\Box\in\mu^{(k)}}\frac{\bar{a}_k+c(\Box)}{h(\Box)}\right)\\
&\times\sum\limits_{y\;\in G\sim S(n)}\chi^{\lambda^{(1)},\ldots,\lambda^{(k)}}(xy)
\chi^{\mu^{(1)},\ldots,\mu^{(k)}}(y^{-1})
\biggr].
\nonumber
\end{split}
\end{equation}
We know that for irreducible characters $\chi^{\pi_1}$ and $\chi^{\pi_2}$
of any finite group $\mathcal{G}$ the orthogonality condition
\begin{equation}
\frac{1}{\left|\mathcal{G}\right|}\sum\limits_{h\in\mathcal{G}}\chi^{\pi_1}(gh)\chi^{\pi_2}(h^{-1})=
\left\{
\begin{array}{ll}
\frac{\chi^{\pi_1}(g)}{\dim(\pi_1)}, & \pi_1\sim\pi_2,\\
0, & \mbox{otherwise}
\end{array}
\right.
\end{equation}
is satisfied. Therefore, the sum over $G\sim S(n)$ gives the normalized character parameterized by $\lambda^{(1)}$, $\ldots$, $\lambda^{(k)}$, and we obtain
\begin{equation}
\begin{split}
&\chi_{z_1,\ldots,z_k}(x)=
\frac{n!}{
\left(\frac{t_1}{\zeta_{c_1}}+\ldots+\frac{t_k}{\zeta_{c_k}}\right)_n}\\
&\times
\underset{|\lambda^{(1)}|+\ldots+|\lambda^{(k)}|=n}{\sum\limits_{\lambda^{(1)},\ldots,\lambda^{(k)}}}\biggl[
\left(\prod\limits_{\Box\in\lambda^{(1)}}\frac{\left(a_1+c(\Box)\right)\left(\bar{a}_1+c(\Box)\right)}{h^2(\Box)}\right)
\ldots\left(\prod\limits_{\Box\in\lambda^{(k)}}\frac{\left(a_k+c(\Box)\right)\left(\bar{a}_k+c(\Box)\right)}{h^2(\Box)}\right)\\
&\times\frac{\chi^{\lambda^{(1)},\ldots,\lambda^{(k)}}(x)}{\DIM(\lambda^{(1)},\ldots,\lambda^{(k)})}\biggr],
\nonumber
\end{split}
\end{equation}
where $x\in G\sim S(n)$.
The orthogonality of the irreducible characters $\gamma^{j}$  together with equation (\ref{Parametersa})
can be used to obtain the equality
$$
\frac{t_1}{\zeta_{c_1}}+\ldots+\frac{t_k}{\zeta_{c_k}}=a_1\bar{a}_1+\ldots+a_k\bar{a}_k.
$$
Taking into account formula (\ref{StandardZMeasure}) we get equation (\ref{SPHERICALFUNCTIONZMEASURES}) with
$M_{z_1,\ldots,z_k}^{(n)}\left(\lambda^{(1)},\ldots,\lambda^{(k)}\right)$ defined by equation
(\ref{MultipleZmeasures}).\qed\\
\section{Spectral decomposition of $\chi_{z_1,\ldots,z_k}$}
The function $\chi_{z_1,\ldots,z_k}:\; G_{\infty}=G\sim S(\infty)\longrightarrow\C$ is
a character of the group $G_{\infty}$ (in the sense of Definition
\ref{characters7}). As any character of $G_{\infty}$ it admits an integral representation called the \textit{spectral decomposition of a character}.
Such integral representation can be deduced as a consequence of the relation between
characters of $G_{\infty}$ and harmonic functions  on a certain branching graph $\Gamma\left(G\right)$. The graph $\Gamma\left(G\right)$ reflects the branching rules for the characters of irreducible representations of $G_n=G\sim S(n)$, and harmonic functions on $\Gamma\left(G\right)$ can be represented as  Poisson-like integrals over certain set $\Omega\left(G\right)$,  which is  a generalization of the Thoma set.
\subsection{The branching rule for the characters of the finite wreath products}
If $p<n$, then the canonical inclusion
$$
i_{n,p}: G\sim S(p)\longrightarrow G\sim S(n)
$$
is defined by
$$
i_{n,p}\left[\left(\left(g_1,\ldots,g_p\right),s\right)\right]
=\left(\left(g_1,\ldots,g_p,e_{G},\ldots,e_{G}\right),\;s(p+1)(p+2)\ldots n\right),
$$
where $e_{G}$ denotes the unit element of $G$. Under this inclusion $G\sim S(p)$ can be understood as a subgroup of $G\sim S(n)$. Let $\chi^{\Lambda_n^{(k)}}$ be the character of the irreducible representation of $G\sim S(n)$ parameterized by $\Lambda_n^{(k)}=\left(\lambda^{(1)},\ldots,\lambda^{(k)}\right)\in\Y_n^{(k)}$. In order to present a formula for the restriction $\chi^{\Lambda_n^{(k)}}\biggl|_{G\sim S(n-1)}$ of $\chi^{\Lambda_n^{(k)}}$
to the subgroup $G\sim S(n-1)$ of $G\sim S(n)$ we need the following notation. Let $\Lambda_n^{(k)}\in\Y_n^{(k)}$ and $\tLambda_{n-1}^{(k)}\in\Y_{n-1}^{(k)}$.
Thus $\Lambda_n^{(k)}=\left(\lambda^{(1)},\ldots,\lambda^{(k)}\right)$ and
$\tLambda_{n-1}^{(k)}=\left(\mu^{(1)},\ldots,\mu^{(k)}\right)$ are such that
$|\lambda^{(1)}|+\ldots+|\lambda^{(k)}|=n$ and $|\mu^{(1)}|+\ldots+|\mu^{(k)}|=n-1$.
Assume that there exist $l\in\{1,\ldots,k\}$ such that $\mu^{(l)}\nearrow\lambda^{(l)}$
(i.e. $\lambda^{(l)}$ is obtained from $\mu^{(l)}$ by adding one box), and such that $\mu^{(i)}=\lambda^{(i)}$ for each $i$, $i\in\left\{1,\ldots,k\right\}$, $i\neq l$.
Then we write $\tLambda_{n-1}^{(k)}\nearrow\Lambda_n^{(k)}$. With this notation the branching rule for the characters of the finite wreath product $G\sim S(n)$ can be written as
\begin{equation}\label{2.3.branching}
\chi^{\Lambda_n^{(k)}}\biggl|_{G\sim S(n-1)}
=\sum\limits_{\tLambda_{n-1}^{(k)}:\;\tLambda_{n-1}^{(k)}\nearrow\Lambda_n^{(k)}}\Upsilon\left(\tLambda_{n-1}^{(k)},\Lambda_n^{(k)}\right)\chi^{\tLambda_{n-1}^{(k)}},
\end{equation}
where
\begin{equation}\label{2.3.multiplicity}
\Upsilon\left(\tLambda_{n-1}^{(k)},\Lambda_n^{(k)}\right)=\left\{
\begin{array}{ll}
d_1, & \mu^{(1)}\nearrow\lambda^{(1)},\\
d_2, & \mu^{(2)}\nearrow\lambda^{(2)}, \\
\vdots, & \\
d_k, & \mu^{(k)}\nearrow\lambda^{(k)},
\end{array}
\right.
\end{equation}
and $d_l=\gamma^l\left(e_{G}\right)$, $l=1,\ldots,k$.
The branching rule for the characters (equation (\ref{2.3.branching})) gives rise to the recurrence
relation for the dimensions of irreducible representations,
\begin{equation}
\DIM\left(\Lambda_n^{(k)}\right)
=\sum\limits_{\tLambda_{n-1}^{(k)}:\;\tLambda_{n-1}^{(k)}\nearrow\Lambda_n^{(k)}}\Upsilon\left(\tLambda_{n-1}^{(k)},\Lambda_n^{(k)}\right)\DIM\left(\tLambda_{n-1}^{(k)}\right).
\end{equation}
\subsection{The branching graph $\Gamma(G)$. Representation of harmonic functions on $\Gamma\left(G\right)$}
\subsubsection{The branching graph $\Gamma(G)$}
Let $\Y^{(k)}$ denote the union of the sets $\Y_n^{(k)}$
(with the understanding that $\Y_0^{(k)}$ contains the element
$\Lambda_0^{(k)}=\left(\emptyset,\ldots,\emptyset\right)$ only).
We define a branching graph $\Gamma\left(G\right)$ with the vertex set $\Y^{(k)}$  by declaring that a pair of vertices $\left(\Lambda_{n-1}^{(k)},\tLambda_n^{(k)}\right)$ is connected by an edge of multiplicity $\Upsilon\left(\Lambda_{n-1}^{(k)},\tLambda_n^{(k)}\right)$ if and only if $\Lambda_{n-1}^{(k)}\nearrow\tLambda_n^{(k)}$.
In other words, the graph $\Gamma\left(G\right)$ is the branching graph which reflects the branching rule for the characters of irreducible representations of $G\sim S(n)$ (see equation (\ref{2.3.branching})). In particular, $\DIM\left(\Lambda_n^{(k)}\right)$ can be understood as the number of oriented paths from $\Lambda_0^{(k)}=\left(\emptyset,\ldots,\emptyset\right)$ to $\Lambda_n^{(k)}$ on the branching graph $\Gamma(G)$.
\subsubsection{Harmonic functions and coherent systems of measures}
\begin{defn}\label{DEFINITIONHARMONICFUNCTIONG} A function  $\varphi:\; \Gamma\left(G\right)\longrightarrow\R_{\geq 0}$ is called harmonic if for each $n=1,2,\ldots$
\begin{equation}\label{3.2.1.ha}
\varphi\left(\Lambda_{n-1}^{(k)}\right)
=\sum\limits_{\tLambda_n^{(k)}:\;\tLambda_n^{(k)}\searrow\Lambda_{n-1}^{(k)}}
\Upsilon\left(\Lambda_{n-1}^{(k)},\tLambda_n^{(k)}\right)\varphi\left(\tLambda_n^{(k)}\right),
\end{equation}
where $\Upsilon$ is defined by (\ref{2.3.multiplicity}), and $\varphi\left(\Lambda_0^{(k)}\right)=1$.
\end{defn}
Let $\varphi $ be a harmonic function on $\Gamma\left(G\right)$. Set
\begin{equation}\label{3.2.2.h}
\mathcal{M}_n^{(k)}\left(\Lambda_n^{(k)}\right)=\DIM\left(\Lambda_n^{(k)}\right)
\varphi\left(\Lambda_n^{(k)}\right),\;\; n=0,1,2,\ldots .
\end{equation}
Then $\mathcal{M}_n^{(k)}$ is a probability measure on $\Y_n^{(k)}$. The sequence
$\left(\mathcal{M}_n^{(k)}\right)_{n=0}^{\infty}$
(where each element is defined by equation (\ref{3.2.2.h})) is called a coherent system of probability measures associated with the harmonic function $\varphi$ on $\Gamma\left(G\right)$.

If $\Lambda_{n-1}^{(k)}=\left(\mu^{(1)},\ldots,\mu^{(k)}\right)$, and
$\tLambda_n^{(k)}=\left(\lambda^{(1)},\ldots,\lambda^{(k)}\right)$, then it follows from equations (\ref{3.2.1.ha}), (\ref{3.2.2.h}) that
\begin{equation}\label{3.3.harmonic}
\frac{\mathcal{M}_{n-1}^{(k)}\left(\Lambda_{n-1}^{(k)}\right)}{\DIM\left(\Lambda_{n-1}^{(k)}\right)}
=\sum\limits_{j=1}^{k}d_j\underset{\lambda^{(i)}=\mu^{(i)},\; i\neq j}{\sum\limits_{\lambda^{(j)}\searrow\;\mu^{(j)}}}
\frac{\mathcal{M}_{n}^{(k)}\left(\tLambda_{n}^{(k)}\right)}{\DIM\left(\tLambda_{n}^{(k)}\right)},
\end{equation}
for each coherent system $\left(\mathcal{M}_n^{(k)}\right)_{n=0}^{\infty}$ of probability measures associated with a harmonic function on $\Gamma\left(G\right)$.
\subsubsection{Representation of harmonic functions}
Hora and Hirai \cite{HoraHirai} proved that there is
one-to-one correspondence between harmonic functions on $\Gamma\left(G\right)$ and probability measures on the set $\Omega\left(G\right)$ defined by
\begin{equation}\label{3.3.1.Thoma}
\begin{split}
&\Omega(G)=\biggl\{(\alpha,\beta,\delta)\biggr|
\alpha=\left(\alpha^{(1)},\ldots,\alpha^{(k)}\right),\beta=\left(\beta^{(1)},\ldots,\beta^{(k)}\right),
\delta=\left(\delta^{(1)},\ldots,\delta^{(k)}\right);\\
&\alpha^{(l)}=\left(\alpha_1^{(l)}\geq\alpha_2^{(l)}\geq\ldots\geq 0\right), \beta^{(l)}=\left(\beta_1^{(l)}\geq\beta_2^{(l)}\geq\ldots\geq 0\right),\\
&\delta^{(1)}\geq 0,\ldots,\delta^{(k)}\geq 0,\\
&\mbox{where}\;\;
\sum\limits_{i=1}^{\infty}\alpha_i^{(l)}+\beta_i^{(l)}\leq\delta^{(l)},\; 1\leq l\leq k,\;\mbox{and}\;\sum\limits_{l=1}^k\delta^{(l)}=1\biggl\}.
\end{split}
\end{equation}
In order to present this result we use the extended power symmetric functions $P_{r,l}^{o}\left(.\right)$ on $\Omega\left(G\right)$. These functions are obtained by specializing the power sums $p_1$, $p_2$, $\ldots$ to the following expressions
\begin{equation}\label{3.3.2.PowerSums}
\begin{split}
&p_1\longrightarrow P^{o}_{1,l}\left(\alpha,\beta,\delta\right)=\delta^{(l)},\\
&p_r\longrightarrow P^{o}_{r,l}\left(\alpha,\beta,\delta\right)
=\sum\limits_{i=1}^{\infty}\left(\lambda_i^{(l)}\right)^r+(-1)^{r-1}
\sum\limits_{i=1}^{\infty}\left(\beta_i^{(l)}\right)^r,\; r=2,3,\ldots .
\end{split}
\end{equation}
Here $l\in\left\{1,\ldots,k\right\}$.
Given the extended power symmetric functions $P_{r,l}^{o}$ we introduce
the extended Schur functions $S_{\lambda^{(l)}}^{o}(.)$ on $\Omega\left(G\right)$.  Namely, to obtain $S_{\lambda^{(l)}}^{o}(.)$  we express the Schur function
$s_{\lambda^{(l)}}$ as a polynomial in variables $p_1$, $p_2$, $\ldots$,
and replace each $p_r$ by the extended power symmetric functions $P_{r,l}^{o}(.)$ defined by equation (\ref{3.3.2.PowerSums}).
With this notation we are ready to state the representation theorem for harmonic functions on $\Gamma\left(G\right)$.
\begin{thm}\label{THEOREMANALOGTHOMA}
There is a bijective correspondence between the set of  harmonic functions $\varphi$ on $\Gamma(G)$, and the set of  probability measures $P$ on the generalized Thoma set $\Omega(G)$. This correspondence is determined by
\begin{equation}\label{HarmonicFunctionProbabilityMeasure}
\varphi\left(\Lambda_n^{(k)}\right)=\int\limits_{\Omega(G)}\mathbb{K}\left(\Lambda_n^{(k)},\omega\right)P(d\omega),\;\;\forall \Lambda_n^{(k)}=\left(\lambda^{(1)},\ldots,\lambda^{(k)}\right)\in\Y_n^{(k)},
\end{equation}
where $n=1,2,...$, and $k$ is the number of conjugacy classes in $G$.
 The  kernel $\mathbb{K}\left(\Lambda_n^{(k)},\omega\right)$ is given by
\begin{equation}\label{LimitingMartinKernel}
\mathbb{K}\left(\Lambda_n^{(k)},\omega\right)=\prod\limits_{l=1}^k\frac{1}{d_l^{|\lambda^{(l)}|}}
S_{\lambda^{(l)}}^{o}\left(\alpha,\beta,\delta\right).
\end{equation}
Here $S_{\lambda^{(l)}}^{o}\left(\alpha,\beta,\delta\right)$ denotes the extended Schur function parameterized by $\lambda^{(l)}$, and $d_l$ are the dimensions of irreducible representations of $G$.
\end{thm}
\begin{proof}See Hora and Hirai \cite{HoraHirai}, Theorem 2.5 and Theorem 3.1.
\end{proof}
\begin{rem} Quite different proof of Theorem \ref{THEOREMANALOGTHOMA} can be found in Strahov \cite{StrahovMPS}, where harmonic functions on the Jack deformation $\Gamma_{\theta}\left(G\right)$ of $\Gamma\left(G\right)$ are considered.
\end{rem}
\subsection{Coherent systems of probability measures on $\Gamma\left(G\right)$ and characters}
The following theorem relates a character of $G\sim S(\infty)$ with a coherent system of probability measures on the branching graph $\Gamma \left(G\right)$.
\begin{thm}\label{THEOREM4.1.2.CharCSystems}
Denote by $\chi$ a character of $G\sim S(\infty)$  (in the sense of Definition
\ref{characters7}), and by $\chi_n$ its restriction to $G\sim S(n)$.
Then
\begin{equation}\label{4.1.2.R.CH.S}
\chi_n=\sum\limits_{\Lambda_n^{(k)}\in\Y_n^{(k)}}
\mathcal{M}_n\left(\Lambda_n^{(k)}\right)
\frac{\chi^{\Lambda_n^{(k)}}}{\DIM\Lambda_n^{(k)}},
\end{equation}
where $\chi^{\Lambda_n^{(k)}}$ is the character of the irreducible representation of $G\sim S(n)$ parameterized by $\Lambda_n^{(k)}$, and $\mathcal{M}_n$ is a probability measure on $\Y_n^{(k)}$. Moreover, equation (\ref{4.1.2.R.CH.S}) determines a bijective correspondence
$$
\chi\longleftrightarrow\left(\mathcal{M}_n\right)_{n=0}^{\infty}
$$
between the characters of $G\sim S(\infty)$ and the coherent systems of probability measures on $\Gamma\left(G\right)$.
\end{thm}
\begin{proof}Let $\chi$ be a character of $G\sim S(\infty)$. Then its restriction $\chi_n$ to $G\sim S(n)$ is a normalized central function on $G\sim S(n)$. Equation (\ref{4.1.2.R.CH.S}) is a
representation of $\chi_n$ as a linear combination of normalized irreducible characters of $G\sim S(n)$. Since $\chi_n$ is normalized and central,
the coefficients $\mathcal{M}_n\left(\Lambda_n^{(k)}\right)$ are such that
$$
\mathcal{M}_n\left(\Lambda_n^{(k)}\right)\geq 0,\;\;\sum\limits_{\Lambda_n^{(k)}\in\Y_n^{(k)}}\mathcal{M}_n\left(\Lambda_n^{(k)}\right)=1,
$$
as it follows from specialization of a general theorem on characters (Borodin and Olshanski \cite{BorodinOlshanskiBook}, Proposition 1.6) to the case of $G\sim S(n)$. The fact that $\left(\mathcal{M}_n\right)_{n=0}^{\infty}$ is a coherent system  follows from the branching rule for the irreducible characters of $G\sim S(n)$, see equation (\ref{2.3.branching}).

In the opposite direction, assume that $\left(\mathcal{M}_n\right)_{n=0}^{\infty}$ is a coherent system of probability measures on $\Gamma\left(G\right)$, and define a function
$$
\chi:\;\; G\sim S(\infty)\rightarrow\C
$$
by its restriction  $\chi_n$ to $G\sim S(n)$,
where each $\chi_n$ is defined by the right-hand side of equation (\ref{4.1.2.R.CH.S}).
Again, Proposition 1.6 in Borodin and Olshanski \cite{BorodinOlshanskiBook} can be applied to conclude that each $\chi_n$ is positive definite, central, and normalized. The fact that $\left(\mathcal{M}_n\right)_{n=0}^{\infty}$ is a coherent system of probability measures implies equation (\ref{3.3.harmonic}). Using this equation we check that the consistency condition
$\chi_n |_{G\sim S(n-1)}=\chi_{n-1}$ holds true.
\end{proof}
\subsection{Integral representation of $\chi_{z_1,\ldots,z_k}$}\label{SectionIntRepChizzz}
Let $\chi$ be any character of $G\sim S(\infty)$. The following Proposition provides an integral representation for $\chi$.
\begin{prop}\label{Proposition7.1.1}(a) For any character $\chi$ of $G\sim S(\infty)$ there exists a probability measure $P$ on the generalized Thoma set $\Omega(G)$ such that
\begin{equation}\label{7.1.1.1}
\chi(x)=\int\limits_{\Omega(G)}f_{\omega}(x)P(d\omega),\;\; x\in G\sim S(\infty).
\end{equation}
The function
$
f_{\omega}:\;G\sim S(\infty)\longrightarrow\C
$
is parameterized by a point $\omega=\left(\alpha,\beta,\delta\right)$ of the generalized Thoma set $\Omega(G)$, and it is
defined by the formula
\begin{equation}\label{7.1.1.2}
f_{\omega}(x)=\sum\limits_{\Lambda_n^{(k)}\in\Y_n^{(k)}}
\left[\prod\limits_{l=1}^k\frac{1}{d_l^{|\lambda^{(l)}|}}
S_{\lambda^{(l)}}^{o}\left(\alpha,\beta,\delta\right)\right]
\chi^{\Lambda_n^{(k)}}(x),
\end{equation}
where $x\in G\sim S(n)$, ,
$S_{\lambda^{(l)}}^{o}\left(\alpha,\beta,\delta\right)$ denotes the extended Schur function parameterized by $\lambda^{(l)}$, and
$\chi^{\Lambda_n^{(k)}}$ is the character of the irreducible representation of $G\sim S(n)$ parameterized by $\Lambda_n^{(k)}=\left(\lambda^{(1)},\ldots,\lambda^{(k)}\right)\in\Y_n^{(k)}$.
\\
(b) Given $\chi\in\mathcal{X}\left(G\sim S(\infty)\right)$ the measure $P$ is the unique probability measure on $\Omega\left(G\right)$ for which (\ref{7.1.1.1}) is satisfied.\\
(c) Conversely, each probability measure $P$ on the generalized Thoma set $\Omega(G)$ gives rise to a character $\chi$ of $G\sim S(\infty)$, this character is given by equations (\ref{7.1.1.1}) and (\ref{7.1.1.2}).
\end{prop}
\begin{rem}
(a) Representation (\ref{7.1.1.1}) implies that the function $f_{\omega}$ is the extreme character of $G\sim S(\infty)$, and each extreme character of this
group can be represented as in equation (\ref{7.1.1.2}). Indeed, the set of all probability measures on $\Omega\left(G\right)$ is a convex set, and the extreme points of this set are the delta-measures. On the other hand, it follows from Theorem \ref{THEOREMANALOGTHOMA} and Theorem \ref{THEOREM4.1.2.CharCSystems} that there exists a one-to-one correspondence between the extreme points of $\mathcal{X}\left(G\sim S(\infty)\right)$, and the extreme points of the set of all probability measures on $\Omega\left(G\right)$.\\
(b) An explicit formula for $\chi^{\Lambda_n^{(k)}}$ was derived in Hirai, Hirai, Hora \cite{HiraiHiraiHoraI}. This formula can be used to get a different expression for $f_{\omega}$, see Hora and Hirai \cite{HoraHirai}, Theorem 3.4.
\end{rem}

\begin{proof}
Proposition \ref{Proposition7.1.1}
follows as a Corollary of Theorem \ref{THEOREMANALOGTHOMA}, and of Theorem
\ref{THEOREM4.1.2.CharCSystems}. Indeed, the ratio
$$
\frac{\mathcal{M}_n\left(\Lambda_n^{(k)}\right)}{\DIM\left(\Lambda_n^{(k)}\right)}
$$
in the right-hand side of equation (\ref{4.1.2.R.CH.S}) can be understood as the value of a harmonic function at $\Lambda_n^{(k)}$. Therefore, this ratio can be represented as the integral in the right-hand side of equation (\ref{HarmonicFunctionProbabilityMeasure}). This gives formulae (\ref{7.1.1.1}) and (\ref{7.1.1.2}).
\end{proof}
From Proposition \ref{Proposition7.1.1} we
conclude that there exists an integral representation for $\chi_{z_1,\ldots,z_k}$.
Namely, there exists a unique probability measure $P_{z_1,\ldots,z_k}$ on the generalized Thoma set $\Omega\left(G\right)$ such that
\begin{equation}\label{ksizzzmeasure}
\chi_{z_1,\ldots,z_k}(x)=\int\limits_{\Omega(G)}f_{\omega}(x)P_{z_1,\ldots,z_k}(d\omega),\;\; x\in G\sim S(\infty),
\end{equation}
where $f_{\omega}(x)$ is defined by equation (\ref{7.1.1.2}). \textit{The problem of harmonic analysis on $G\sim S(\infty)$ is to describe $P_{z_1,\ldots,z_k}$ explicitly}. In what follows we refer to $M_{z_1,\ldots,z_k}^{(n)}$ as to the multiple $z$-measures, and to $P_{z_1,\ldots,z_k}$ as to the multiple spectral $z$-measures.

The function $\varphi_{z_1,\ldots,z_k}$ on $\Gamma\left(G\right)$ defined by
$$
\varphi_{z_1,\ldots,z_k}\left(\Lambda_n^{(k)}\right)=
\frac{M_{z_1,\ldots,z_k}^{(n)}\left(\Lambda_n^{(k)}\right)}{\DIM\left(\Lambda_n^{(k)}\right)},
\;\;\Lambda_n^{(k)}=\left(\lambda^{(1)},\ldots,\lambda^{(k)}\right)\in\Y_n^{(k)},
$$
is harmonic. Theorem \ref{THEOREMANALOGTHOMA} can be applied to $\varphi_{z_1,\ldots,z_k}$, and we can write
\begin{equation}\label{ZMEASURESPECTRALREPRESENTATION}
\frac{M_{z_1,\ldots,z_k}^{(n)}\left(\Lambda_n^{(k)}\right)}{\DIM\left(\Lambda_n^{(k)}\right)}=\int\limits_{\Omega(G)}\mathbb{K}\left(\Lambda_n^{(k)},\omega\right)P_{z_1,\ldots,z_k}(d\omega),
\end{equation}
where $\mathbb{K}\left(\Lambda_n^{(k)},\omega\right)$ is defined by equation (\ref{LimitingMartinKernel}), and $P_{z_1,\ldots,z_k}$
is the same probability measure on the generalized Thoma set $\Omega(G)$
as  that in equation (\ref{ksizzzmeasure}). This is due the fact  that the formulae (\ref{ksizzzmeasure}) and
(\ref{ZMEASURESPECTRALREPRESENTATION}) are related via the correspondence between $\chi_{z_1,\ldots,z_k}$ and
$\varphi_{z_1,\ldots,z_k}$, see Theorem \ref{TheoremMultpleZmeasures}.
In what follows it is important that $P_{z_1,\ldots,z_k}$ is a unique
probability measure satisfying (\ref{ZMEASURESPECTRALREPRESENTATION}).
\section{Description of the spectral measures}
It is known (see, for example, Borodin and Olshanski \cite{BorodinOlshanskiMarkov})
that the $z$-measures $M_z^{(n)}$ defined by equation (\ref{StandardZMeasure}) form
a coherent system of probability measures on the Young graph $\Y$, and it can be represented as
\begin{equation}\label{ZmeasuresSpectralRepresentation}
\frac{M_z^{(n)}(\lambda)}{\dim\lambda}=\int\limits_{\Omega}s_{\lambda}^{o}(\omega)P_z(d\omega),\;\;\forall\lambda\in\Y_n,
\end{equation}
where $\Omega$ is the Thoma set defined by equation (\ref{ThomaSet}).
The functions  $s_{\lambda}^{o}(w)=s_{\lambda}^{o}(\alpha,\beta)$  in the right-hand side of
equation (\ref{ZmeasuresSpectralRepresentation}) are the Schur symmetric functions expressed as polynomials in variables
$\left\{p_r^{o}(\alpha,\beta):\;\;r\geq 1\right\}$ defined by
\begin{equation}\label{proJack1}
p_r^{o}(\alpha,\beta)=\left\{
  \begin{array}{ll}
    1, & r=1,\\
    \sum\limits_{i=1}^{\infty}\alpha_i^r+(-1)^{r-1}\sum\limits_{i=1}^{\infty}\beta_i^r, &  r\geq 2.
  \end{array}
\right.
\end{equation}
The measure $P_z$ is a unique probability measure on $\Omega$ corresponding to $M_z^{(n)}$. The measures $P_z$ are called \textit{the spectral $z$-measures},
see Borodin and Olshanski \cite{BorodinOlshanskiLetters}, Section 2.
A detail description of $P_{z}$ is available in the literature,
see Refs. \cite{Borodin1,Borodin2,BorodinOlshanskiLetters,OlshanskiNato}.
\begin{prop}All measures $P_{z}$, $z\in\C$ are supported by a subset $\Omega_0$ of the Thoma set $\Omega$.
The subset $\Omega_0$ is defined by
\begin{equation}\label{Omegao}
\Omega_0=\left\{\omega: \omega=(\alpha,\beta),\;
\begin{array}{c}
\alpha=\left(\alpha_1\geq\alpha_2\geq\ldots\geq 0\right)\\
\beta=\left(\beta_1\geq\beta_2\geq\ldots\geq 0\right)                     \end{array},\;\sum\limits_{i=1}^{\infty}(\alpha_i+\beta_i)=1\right\}.
\end{equation}
\end{prop}
\begin{proof} See Borodin and Olshanski \cite{BorodinOlshanskiLetters}, Section 5, Theorem I.
\end{proof}
Consider the probability space $\left(\Omega\left(G\right),
P_{z_1,\ldots,z_k}\right)$. The coordinates $\left(\widetilde{\alpha},\widetilde{\beta},\widetilde{\delta}\right)$ of $\Omega\left(G\right)$ are functions in $\omega\in\Omega\left(G\right)$, hence we may view them
as random variables. Theorem \ref{THEOREMDESCRIPTIONOFPz1z2zk}
provides an information on distribution of $\left(\widetilde{\alpha},\widetilde{\beta},\widetilde{\delta}\right)$, and thus gives a description of the multiple spectral $z$-measures $P_{z_1,\ldots,z_k}$ in terms of the spectral $z$-measures
$P_{z_1}$, $\ldots$, $P_{z_k}$.

Recall that the Dirichlet distribution $D\left(\tau_1,\ldots,\tau_k\right)$ with parameters $\tau_1$, $\ldots$, $\tau_k$ is defined by equation (\ref{PoissonDirichletDensity}).
\begin{thm}\label{THEOREMDESCRIPTIONOFPz1z2zk} Assume that for each $l=1,\ldots,k$
$$
\alpha^{(l)}=\left(\alpha_1^{(l)}\geq\alpha_2^{(l)}\geq\ldots\geq 0\right),\;\;
\beta^{(l)}=\left(\beta_1^{(l)}\geq\beta_2^{(l)}\geq\ldots\geq 0\right)
$$
are random variables whose joint distribution is determined by the spectral measure  $P_{z_l}$, and that these collections of random varaiables are pairwise independent for different $l$.
Let
$\delta^{(1)}$, $\ldots$, $\delta^{(k)}$ be random variables independent on $\alpha^{(l)}$ and $\beta^{(l)}$ whose joint distribution is the Dirichlet distribution $D\left(\tau_1,\ldots,\tau_k\right)$ with parameters
$\tau_1=a_1\bar{a}_1$, $\ldots$, $\tau_k=a_k\bar{a}_k$, where $a_l$, $l=1,\ldots, k$, are given by equation (\ref{Parametersa}).
In addition, let $\left(\widetilde{\alpha},\widetilde{\beta},\widetilde{\delta}\right)$
be the random coordinates of $\Omega\left(G\right)$ whose joint distribution is determined by the multiple spectral $z$-measure $P_{z_1,\ldots,z_k}$. Then
\begin{equation}
\widetilde{\alpha}^{(l)}=\delta^{(l)}\alpha^{(l)},\; \widetilde{\beta}^{(l)}=
\delta^{(l)}\beta^{(l)},\;\widetilde{\delta}^{(l)}=\delta^{(l)}
\end{equation}
in distribution, for each $l=1,\ldots,k$.
\end{thm}
\begin{proof}
We use the spectral representations for the $z$-measures $M_{z_l}^{(n)}$  (equation (\ref{ZmeasuresSpectralRepresentation})), the spectral representation for the multiple $z$-measures (equation (\ref{ZMEASURESPECTRALREPRESENTATION})), formula (\ref{MultipleZmeasures}), and formula (\ref{2.1.2.1})
for $\DIM\left(\Lambda_n^{(k)}\right)$ to obtain the relation
\begin{equation}
\begin{split}
&\int\limits_{\Omega(G)}\prod\limits_{j=1}^k
S_{\lambda^{(j)}}^{o}\left(\tilde{\alpha},\tilde{\beta},\tilde{\delta}\right)P_{z_1,\ldots,z_k}\left(d\tilde{\omega}\right)=\frac{\left(\tau_1\right)_{|\lambda^{(1)}|}\ldots \left(\tau_k\right)_{|\lambda^{(k)}|}}{\left(\tau_1+\ldots+\tau_k\right)_n}\\
&\times
\int\limits_{\Omega_0^{(1)}}\ldots\int\limits_{\Omega_0^{(k)}}
s_{\lambda^{(1)}}^{o}\left(\alpha^{(1)},\beta^{(1)}\right)\ldots s_{\lambda^{(k)}}^{o}
\left(\alpha^{(k)},\beta^{(k)}\right)P_{z_1}\left(d\omega^{(1)}\right)\ldots
P_{z_k}\left(d\omega^{(k)}\right),
\end{split}
\end{equation}
where $\Omega^{(1)}_0$, $\ldots$, $\Omega^{(k)}_0$ are the $k$ copies of the set $\Omega_0$ defined by equation
(\ref{Omegao}).
We rewrite the prefactor in the right-hand side of the equation above as
$$
\frac{\Gamma\left(\tau_1+\ldots+\tau_k\right)}{\Gamma\left(\tau_1\right)\ldots\Gamma\left(\tau_k\right)}
\underset{\delta^{(1)}+\ldots+\delta^{(k)}=1}{\underset{\delta^{(1)}\geq 0,\ldots, \delta^{(k)}\geq 0}{\int\ldots\int}}
\left(\delta^{(1)}\right)^{\tau_1+|\lambda^{(1)}|-1}\ldots
\left(\delta^{(k)}\right)^{\tau_k+|\lambda^{(k)}|-1}d\delta^{(1)}\ldots d\delta^{(k)}.
$$
The Schur functions are homogeneous symmetric functions, so we can write
$$
\left(\delta^{(l)}\right)^{|\lambda^{(l)}|}s_{\lambda^{(l)}}^{o}\left(\alpha^{(l)},\beta^{(l)}\right)
=s_{\lambda^{(l)}}^{o}\left(\delta^{(l)}\alpha^{(l)},\delta^{(l)}\beta^{(l)}\right),\; 1\leq l\leq k,
$$
and  obtain the relation
\begin{equation}\label{RelationSpectralZ}
\begin{split}
&\int\limits_{\Omega(G)}\prod\limits_{j=1}^kS_{\lambda^{(j)}}^{o}\left(\tilde{\alpha},\tilde{\beta},\tilde{\delta}\right)P_{z_1,\ldots,z_k}\left(d\tilde{w}\right)\\
&=
\int\limits_{\Omega_0^{(1)}}\ldots\int\limits_{\Omega_0^{(k)}}
\underset{\delta^{(1)}+\ldots+\delta^{(k)}=1}{\underset{\delta^{(1)}\geq 0,\ldots, \delta^{(k)}\geq 0}{\int\ldots\int}}
\prod\limits_{l=1}^ks_{\lambda^{(l)}}^{o}\left(\delta^{(l)}\alpha^{(l)},\delta^{(l)}\beta^{(l)}\right)
P_{z_1}\left(dw^{(1)}\right)\ldots
P_{z_k}\left(dw^{(k)}\right)\\
&\times D\left(\tau_1,\ldots,\tau_k\right)\left(d\delta^{(1)},\ldots,d\delta^{(k)}\right).
\end{split}
\end{equation}
Assume that $P_{z_1,\ldots,z_k}$ is such that
$$
\tilde{\alpha}^{(l)}=\delta^{(l)}\alpha^{(l)},\;\; l=1,\ldots, k
$$
in distribution, and
$$
\tilde{\delta}^{(l)}=\delta^{(l)}, \;\; l=1,\ldots, k
$$
in distribution where $\alpha^{(1)}$, $\ldots$, $\alpha^{(k)}$ are independent with distributions
$P_{z_1}$, $\ldots$, $P_{z_k}$ respectively, the joint distribution of
$\alpha^{(1)}$, $\ldots$, $\alpha^{(k)}$ is $D(\tau_1,\ldots,\tau_k)$, and each $\delta^{(l)}$,
$l=1,\ldots, k$ is independent of $\alpha^{(1)}$, $\ldots$
$\alpha^{(k)}$.
Then
$$
S_{\lambda^{(l)}}^{o}\left(\tilde{\alpha},\tilde{\beta},\tilde{\delta}\right)
=s_{\lambda^{(l)}}^{o}\left(\delta^{(l)}\alpha^{(l)},\delta^{(l)}\beta^{(l)}\right)
$$
in distribution, $P_{z_1,\ldots,z_k}$ is concentrated on $\Omega_0(G)$ defined by
\begin{equation}\label{3.3.1.Thoma0}
\begin{split}
&\Omega_0(G)=\biggl\{(\alpha,\beta,\delta)\biggr|
\alpha=\left(\alpha^{(1)},\ldots,\alpha^{(k)}\right),\beta=\left(\beta^{(1)},\ldots,\beta^{(k)}\right),
\delta=\left(\delta^{(1)},\ldots,\delta^{(k)}\right);\\
&\alpha^{(l)}=\left(\alpha_1^{(l)}\geq\alpha_2^{(l)}\geq\ldots\geq 0\right), \beta^{(l)}=\left(\beta_1^{(l)}\geq\beta_2^{(l)}\geq\ldots\geq 0\right),\\
&\delta^{(1)}\geq 0,\ldots,\delta^{(k)}\geq 0,\\
&\mbox{where}\;\;
\sum\limits_{i=1}^{\infty}\alpha_i^{(l)}+\beta_i^{(l)}=\delta^{(l)},\; 1\leq l\leq k,\;\mbox{and}\;\sum\limits_{l=1}^k\delta^{(l)}=1\biggl\},
\end{split}
\end{equation}
and equation (\ref{RelationSpectralZ}) is satisfied.
Since $P_{z_1,\ldots,z_k}$ is a unique probability measure for which equation
(\ref{RelationSpectralZ}) holds true, the statement of Theorem \ref{THEOREMDESCRIPTIONOFPz1z2zk} follows.
\end{proof}
\section{The point process $\mathcal{P}_{z_1,\ldots,z_k}$}\label{SectionPointProcess11}
\subsection{Definition of $\mathcal{P}_{z_1,\ldots,z_k}$}
Let $I=[-1,1]$, and $I^*=[-1,1]\setminus\{0\}$. Set
$$
\mathfrak{X}=\left\{1,\ldots,k\right\}\times I^{\ast}
$$
The set $\mathfrak{X}$ can be represented as
$$
\mathfrak{X}=\left\{\left(l,x^l\right):\;\;l\in \{1,\ldots,k\},\; x^l\in I^{\ast}\right\}.
$$
Denote by $\Conf\left(\mathfrak{X}\right)$ the collection of all finite and countably infinite subsets of $\mathfrak{X}$. Each $C\in\Conf\left(\mathfrak{X}\right)$ is called a point configuration, and
$\Conf\left(\mathfrak{X}\right)$ is called the space of point configurations.
Clearly, each $C\in\Conf\left(\mathfrak{X}\right)$ can be written as
$$
C=C_1\cup\ldots\cup C_k,
$$
where each $C_l$ is a subset of $I^{\ast}$.

Let $\Omega\left(G\right)$ be the generalized Thoma set defined by equation
(\ref{3.3.1.Thoma}). Define the map
$$
\phi:\; \Omega\left(G\right)\longrightarrow\Conf\left(\mathfrak{X}\right)
$$
by
$$
\left(\alpha^{(l)},\beta^{(l)},\delta^{(l)}\right)\rightarrow
C_l=\left\{\alpha_i^{(l)}\neq 0\right\}\cup\left\{-\beta_i^{(l)}\neq 0\right\},\;\; l=1,\ldots,k.
$$
We regard $\alpha_i^{(l)}$, $-\beta_i^{(l)}$ as coordinates of particles on the $l$th level. On each level we forget the ordering, remove the possible zero coordinates, and change the sign of the $\beta$-coordinates.

Recall that the probability measure $P_{z_1,\ldots,z_k}$ on $\Omega\left(G\right)$ was introduced in Section \ref{SectionIntRepChizzz} as the spectral measure of the characters $\chi_{z_1,\ldots,z_k}$, and that this measure was described in terms of the spectral $z$-measures $P_{z_1}$, $\ldots$, $P_{z_k}$ in Theorem \ref{THEOREMDESCRIPTIONOFPz1z2zk}. Denote by $\mathcal{P}_{z_1,\ldots,z_k}$ the pushforward of $P_{z_1,\ldots,z_k}$ under the map $\phi$. The measure $\mathcal{P}_{z_1,\ldots,z_k}$ is a probability measure on $\Conf\left(\mathfrak{X}\right)$, i.e. it is a point process on the space $\mathfrak{X}$.

A sequence $\varrho_1^{z_1,\ldots,z_k}$, $\varrho_2^{z_1,\ldots,z_k}$,
$\ldots$ of functions, where, for any $n$, $\varrho_n^{z_1,\ldots,z_k}$ is a symmetric function on $\mathfrak{X}^n$, can be assigned to the point process $\mathcal{P}_{z_1,\ldots,z_k}$. These functions are called the correlation functions of $\mathcal{P}_{z_1,\ldots,z_k}$, and they are defined by
\begin{equation}\label{correlation11.1.1}
\begin{split}
&\mathbb{E}\left(\prod\limits_{l=1}^k\prod\limits_{j}
\left(1+\varphi\left(\left(l,\alpha_j^{(l)}\right)\right)\right)
\left(1+\varphi\left(\left(l,-\beta_j^{(l)}\right)\right)\right)\right)\\
&=\sum\limits_{n=0}^{\infty}\frac{1}{n!}
\left(\sum\limits_{m_1,\ldots,m_n=1}^k
\int\limits_{I^n}
\prod\limits_{i=1}^n\varphi\left(m_i,x_i^{m_i}\right)
\varrho_n^{z_1,\ldots,z_k}\left[\left(m_1,x_1^{m_1}\right),\ldots,\left(m_n,x_n^{m_n}\right)\right]dx_1^{m_1}\ldots dx_n^{m_n}\right),
\end{split}
\end{equation}
where $\varphi$ is a compactly supported Borel function on $\mathfrak{X}$.
Note that equation (\ref{correlation11.1.1}) is equivalent to
\begin{equation}\label{correlation11.1.2}
\begin{split}
&\mathbb{E}\left(\prod\limits_{l=1}^k\prod\limits_{j}
\left(1+\varphi\left(\left(l,\alpha_j^{(l)}\right)\right)\right)
\left(1+\varphi\left(\left(l,-\beta_j^{(l)}\right)\right)\right)\right)\\
&=\sum\limits_{n=0}^{\infty}\frac{1}{n!}
\biggl(\sum\limits_{n_1,\ldots,n_k\geq 0}^k
\left(\begin{array}{c}
  n\\
  n_1,\ldots,n_k
\end{array}\right)
\\
&
\times\int\limits_{I^n}\varrho_n^{z_1,\ldots,z_k}\left[
\left(1,x_1^{1}\right),\ldots,\left(1,x_{n_1}^{1}\right),
\ldots,\left(k,x_1^k\right),\ldots, \left(k,x_{n_k}^k\right)\right]
\prod\limits_{l=1}^k\prod\limits_{i_l=1}^{n_l}\varphi\left(l,x_{i_l}^{l}\right)dx_{i_l}^{l}\biggr).
\end{split}
\end{equation}
Equation (\ref{correlation11.1.2}) is especially convenient for computations with
correlation functions.

\subsection{Lifting}\label{SectionLifting}
Set $t_1=|z_1|^2$, $\ldots$, $t_k=|z_k|^2$, and let $s_1$, $\ldots,s_k$ be independent gamma distributed random variables such that the distribution of $s_l$, $1\leq l\leq k$, has the form
$$
\frac{1}{\Gamma(t_l)}x^{t_l-1}e^{-x}dx.
$$
We assume that $s_1$, $\ldots,s_k$ are independent on  $\mathcal{P}_{z_1,\ldots,z_k}$. Given a configuration $C=\cup_{l=1}^kC_l$ we multiply the coordinates of all particles of $C_l$ by $s_l$,. The result is a point process on
$\widetilde{\mathfrak{X}}
=\{1,\ldots,k\}\times\R^{\ast},
$
where $\R^{\ast}=\R\setminus\{0\}$. We denote this process by
$\widetilde{\mathcal{P}}_{z_1,\ldots,z_k}$.

The correlation functions $\widetilde{\varrho}_n^{z_1,\ldots,z_k}$
of the lifted point process $\widetilde{\mathcal{P}}_{z_1,\ldots,z_k}$ are defined by formulas similar to (\ref{correlation11.1.1}) and (\ref{correlation11.1.2}). In particular, the integration over $I^{n}$ is replaced by integration over $\R^{n}$.
\begin{prop} The relation between the correlation functions $\widetilde{\varrho}_n^{z_1,\ldots,z_k}$ of the lifted point process
$\widetilde{\mathcal{P}}_{z_1,\ldots,z_k}$ and the correlation functions $\varrho_n^{z_1,\ldots,z_k}$ of the original point process $\mathcal{P}_{z_1,\ldots,z_k}$
 is
\begin{equation}
\begin{split}
&\widetilde{\varrho}_n^{z_1,\ldots,z_k}\left[\left(1,x_1^1\right),\ldots,\left(1,x_{n_1}^1\right),\ldots,\left(k,x_{1}^{k}\right),\ldots,\left(k,x^k_{n_k}\right)\right]\\
&=\int\limits_0^{\infty}\ldots\int\limits_{0}^{\infty}
\varrho_n^{z_1,\ldots,z_k}\left[\left(1,\frac{x_1^1}{s_1}\right),\ldots,\left(1,\frac{x_{n_1}^1}{s_1}\right),\ldots,\left(k,\frac{x_{1}^{k}}{s_k}\right),\ldots,\left(k,\frac{x^k_{n_k}}{s_k}\right)\right]
\frac{ds_1}{s_1^{n_1}}\ldots\frac{ds_k}{s_k^{n_k}}.
\nonumber
\end{split}
\end{equation}
\begin{proof} Application of formula (\ref{correlation11.1.2}), and of its analogue for the correlation function of the lifted process $\widetilde{\mathcal{P}}_{z_1,\ldots,z_k}$.
\end{proof}

\end{prop}

\subsection{The Whittaker point  process}\label{SectionWhitt}
In what follows (see Theorem \ref{MAINTHEOREMCORRELATIONS} below) we will
express the correlation functions $\widetilde{\varrho}_n^{z_1,\ldots,z_k}$ of the lifted point process $\widetilde{\mathcal{P}}^{z_1,\ldots,z_k}$ in terms of the known correlation functions $\varrho_{n,z_1}^{\Whittaker}$, $\ldots$,
$\varrho_{n,z_k}^{\Whittaker}$ of the Whittaker processes $\mathcal{P}_{z_1}^{\Whittaker}$, $\ldots$, $\mathcal{P}_{z_k}^{\Whittaker}$
with parameters
$z_1$,$\ldots$,$z_k$.  By definition, the Whittaker point process $\mathcal{P}_{z}^{\Whittaker}$  with a parameter $z\in\C\setminus\left\{0\right\}$ is a determinantal point process on $\R^{\ast}=\R\setminus\left\{0\right\}$ with a kernel expressed through the Whittaker function $W_{\kappa,\mu}(x)$  with parameters $\kappa,\mu\in\C$.
The function $W_{\kappa,\mu}(x)$
is a unique solution of
$$
W^{''}-\left(\frac{1}{4}-\frac{\kappa}{x}+\frac{\mu^2-\frac{1}{4}}{x^2}\right)W=0
$$
with the condition $W(x)\sim x^{\kappa}e^{-\frac{x}{2}}$ as $x\rightarrow+\infty$.

Assume that $z=a+ib\in\C\setminus\{0\}$, and set
$$
P_{\pm}(x)=\frac{|z|}{|\Gamma(1\pm z)|}W_{\pm a+\frac{1}{2}, ib}(x),\;\;
Q_{\pm}(x)=\frac{|z|^{3}x^{-\frac{1}{2}}}{|\Gamma(1\pm z)|}W_{\pm a-\frac{1}{2}, ib}(x).
$$
Define
\begin{equation}
K^z_{\Whittaker}(x,y)=
\left\{
\begin{array}{ll}
\frac{P_+(x)Q_+(y)-Q_+(x)P_+(y)}{x-y}, & x>0, y>0,\\
\frac{P_+(x)P_-(-y)+Q_+(x)Q_+(-y)}{x-y}, & x>0, y<0,\\
\frac{P_+(x)P_+(y)+Q_-(-x)Q_+(y)}{x-y}, & x<0, y>0,\\
-\frac{P_-(-x)Q_-(-y)-Q_-(-x)P_-(-y)}{x-y}, & x<0, y<0.\\
\end{array}
\right.
\end{equation}
The correlation functions of $\mathcal{P}_{z}^{\Whittaker}$ can be written as
\begin{equation}
\varrho_{n,z}^{\Whittaker}\left(x_1,\ldots,x_n\right)=
\det\left[K^z_{\Whittaker}(x_i,x_j)\right]_{i,j=1}^n,
\end{equation}
where $n=1,2$, $\ldots$; $x_1$, $\ldots$, $x_n\in\R^{\ast}$.

It is known that the spectral $z$-measure $P_{z}$ defined by equation (\ref{ZmeasuresSpectralRepresentation}) can be described by
the Whittaker point process $\mathcal{P}_{z}^{\Whittaker}$, see Borodin and Olshanski \cite{BorodinOlshanskiLetters}, and references therein. Namely, the measure $P_z$ is a probability measure on the Thoma set $\Omega$. Introduce
a map
$$
\Omega\longrightarrow\Conf\left(I^{\ast}\right),\;
\omega=(\alpha,\beta)\longrightarrow C=\left\{\left(\alpha_i\neq 0\right)\cup\left(-\beta_j\neq 0\right)\right\},
$$
where $\Conf\left(I^{\ast}\right)$ is the collection of all finite and countably infinite subsets of $I^{\ast}$. The measure $\mathcal{P}_{z}$ is the pushforward of $P_z$ under this map, and it can be understood as a point process on $I^{\ast}$.
The lifting $\widetilde{\mathcal{P}}_z$ of $\mathcal{P}_z$ constructed with the gamma-distributed (with the parameter $t=|z|^2$) random variable is the Whittaker point process $\mathcal{P}_z^{\Whittaker}$.
\subsection{A formula for the correlation functions of $\widetilde{P}_{z_1,\ldots,z_k}$}
In this Section we express the correlation functions $\widetilde{\varrho}_n^{z_1,\ldots,z_k}$ of the lifted point process $\widetilde{P}_{z_1,\ldots,z_k}$ (introduced in Section \ref{SectionLifting}) in terms of the correlation functions
$\varrho_{n_1,z_1}^{\Whittaker}$, $\ldots$, $\varrho_{n_k,z_k}^{\Whittaker}$
of the Whittaker point processes $\mathcal{P}_{z_1}^{\Whittaker}$,
$\ldots$, $\mathcal{P}_{z_k}^{\Whittaker}$  described in Section \ref{SectionWhitt}.

\begin{thm}\label{MAINTHEOREMCORRELATIONS} The correlation functions $\widetilde{\varrho}_n^{z_1,\ldots,z_k}$ of $\widetilde{\mathcal{P}}_{z_1,\ldots,z_k}$
can be written as
\begin{equation}\label{MainFormulaCorF}
\begin{split}
&\widetilde{\varrho}_n^{z_1,\ldots,z_k}\left(x_1^{(1)},\ldots,x^{(1)}_{n_1};\ldots;x_1^{(k)},\ldots,x_{n_k}^{(k)}\right)
=\frac{\Gamma\left(a_1\bar{a}_1+\ldots+a_k\bar{a}_k\right)}{\Gamma\left(a_1\bar{a}_1\right)\ldots\Gamma\left(a_k\bar{a}_k\right)}\\
&\times\underset{\delta_1+\ldots+\delta_k=1}{\underset{\delta_1\geq 0,\ldots, \delta_k\geq 0}{\int\ldots\int}}
\varrho_{n_1,z_1}^{\Whittaker}\left(\frac{x_1^{(1)}}{\delta_1},\ldots,\frac{x_{n_1}^{(1)}}{\delta_1}\right)
\ldots\varrho_{n_k,z_k}^{\Whittaker}\left(\frac{x_1^{(k)}}{\delta_k},\ldots,\frac{x_{n_k}^{(k)}}{\delta_k}\right)\\
&\times\delta_1^{a_1\bar{a}_1-n_1-1}\ldots\delta_k^{a_k\bar{a}_k-n_k-1}d\delta_1\ldots
d\delta_{k}.
\end{split}
\end{equation}
Here $n=1,2,\ldots$; $n_1+\ldots+n_k=n$; $x_1^{(1)}$, $\ldots$, $x_{n_1}^{(1)}$; $\ldots$;
$x_1^{(k)}$, $\ldots$, $x_n^{(k)}\in\R^{\ast}$, and
$\varrho_{n_1,z_1}^{\Whittaker}$, $\ldots$, $\varrho_{n_k,z_k}^{\Whittaker}$ are the correlation functions of the Whittaker
determinantal processes with parameters $z_1$, $\ldots$, $z_k$, respectively. The parameters $a_1$, $\ldots$, $a_k$ are defined by equation (\ref{Parametersa}).
\end{thm}
\begin{proof}
Recall that $\widetilde{\mathcal{P}}_{z_1,\ldots,z_k}$ lives on point configurations
$C$ which can be represented as
$$
C=C_1\cup\ldots\cup C_k,\;\; C_l=\left\{\tilde{\alpha}_i^{(l)}\neq 0\right\}\cup\left\{-\tilde{\beta}_i^{(l)}\neq 0\right\}.
$$
The correlation functions $\widetilde{\varrho}_n^{z_1,\ldots,z_k}$ can be defined by equation (\ref{correlation11.1.2}) as soon as $\alpha_j^{(l)}$ is replaced by
$\tilde{\alpha}_j^{(l)}$, $\beta_j^{(l)}$ is replaced by $\tilde{\beta}_j^{(l)}$, and $I^{n}$ is replaced by $\R^n$.  Namely, we have
\begin{equation}\label{correlation11.1.2.new}
\begin{split}
&\mathbb{E}\left(\prod\limits_{l=1}^k\prod\limits_{j}
\left(1+\varphi\left(\left(l,\tilde{\alpha}_j^{(l)}\right)\right)\right)
\left(1+\varphi\left(\left(l,-\tilde{\beta}_j^{(l)}\right)\right)\right)\right)\\
&=\sum\limits_{n=0}^{\infty}\frac{1}{n!}
\biggl(\sum\limits_{n_1,\ldots,n_k\geq 0}^k
\left(\begin{array}{c}
  n\\
  n_1,\ldots,n_k
\end{array}\right)
\\
&
\times\int\limits_{R^n}\widetilde{\varrho}_n^{z_1,\ldots,z_k}\left[
\left(1,x_1^{1}\right),\ldots,\left(1,x_{n_1}^{1}\right),
\ldots,\left(k,x_1^k\right),\ldots, \left(k,x_{n_k}^k\right)\right]
\prod\limits_{l=1}^k\prod\limits_{i_l=1}^{n_l}\varphi\left(l,x_{i_l}^{l}\right)dx_{i_l}^{l}\biggr).
\end{split}
\end{equation}

We use Theorem \ref{THEOREMDESCRIPTIONOFPz1z2zk}, and the definition of lifting in Section \ref{SectionLifting} to conclude that
\begin{equation}\label{11r}
\widetilde{\alpha}_j^{(l)}=s_l\delta^{(l)}\alpha_j^{(l)},\;\;
\widetilde{\beta}_j^{(l)}=s_l\delta^{(l)}\beta_j^{(l)};\;\; l=1,\ldots,k
\end{equation}
in distribution, where $\alpha^{(l)}=\left(\alpha_1^{(l)}\geq\alpha_2^{(l)}\geq\ldots\geq 0\right)$, $\beta^{(l)}=\left(\beta_1^{(l)}\geq\beta_2^{(l)}\geq\ldots\geq 0\right)$,  $\delta^{(l)}$ are random variables whose distribution is described in  the statement of Theorem \ref{THEOREMDESCRIPTIONOFPz1z2zk}, and $s_l$ is the gamma distributed (with the parameter $t_l=|z_l|^2$) random variable. Taking relation (\ref{11r}) into account, and using independence  of random variables we can write
\begin{equation}\label{11.7.Ex}
\begin{split}
&\mathbb{E}\left[\prod\limits_{l=1}^k\prod\limits_{j}\left(1+\varphi\left(\left(l,\widetilde{\alpha}_j^{(l)}\right)\right)\right)\left(1+\varphi\left(\left(l,-\widetilde{\beta}_j^{(l)}
\right)\right)\right)\right]\\
&=\int\limits_0^{\infty}\ldots\int\limits_0^{\infty}
\frac{s_1^{t_1-1}\ldots s_k^{t_k-1}e^{-s_1-\ldots-s_k}}{\Gamma(t_1)\ldots\Gamma(t_k)}
ds_1\ldots ds_k\\
&\times\underset{\delta^{(1)}+\ldots+\delta^{(k)}=1}{\underset{\delta^{(1)}\geq 0,\ldots, \delta^{(k)}\geq 0}{\int\ldots\int}}
\frac{\Gamma\left(a_1\bar{a}_1+\ldots+a_k\bar{a}_k\right)}{\Gamma\left(a_1\bar{a}_1\right)\ldots\Gamma\left(a_k\bar{a}_k\right)}
\left(\delta^{(1)}\right)^{a_1\bar{a}_1-1}\ldots\left(\delta^{(k)}\right)^{a_k\bar{a}_k-1}
d\delta^{(1)}\ldots d\delta^{(k)}\\
&\times\prod\limits_{l=1}^k\mathbb{E}
\left[\prod\limits_{j}
\left(1+\varphi\left(\left(l,s_l\delta^{(l)}\alpha_j^{(l)}\right)\right)\right)
\left(1+\varphi\left(\left(l,-s_l\delta^{(l)}\beta_j^{(l)}
\right)\right)\right)\right],
\end{split}
\end{equation}
where the expectation in the right-hand side is with respect to $\mathcal{P}_{z_l}$.
Each such expectation can be represented as
\begin{equation}\label{11.8.Ex}
\begin{split}
&\mathbb{E}
\left[\prod\limits_{j}
\left(1+\varphi\left(\left(l,s_l\delta^{(l)}\alpha_j^{(l)}\right)\right)\right)
\left(1+\varphi\left(\left(l,-s_l\delta^{(l)}\beta_j^{(l)}
\right)\right)\right)\right]\\
&=\sum\limits_{n_l=0}^{\infty}\frac{1}{n_l!}
\int\limits_{I^{n_l}}
\varphi\left(\left(l,s_l\delta^{(l)}x_1^{(l)}\right)\right)
\ldots\varphi\left(\left(l,s_l\delta^{(l)}x_{n_l}^{(l)}\right)\right)
\varrho_{n_l}^{z_l}\left(x_1^{(l)},\ldots,x_{n_l}^{(l)}\right)dx_1^{(l)}\ldots dx_{n_l}^{(l)}.
\end{split}
\end{equation}
We insert (\ref{11.8.Ex}) into (\ref{11.7.Ex}), change the variables, and use the relation between the correlation functions $\varrho_{n_1}^{z_l}$ of $\mathcal{P}_{z_l}$, and the correlation functions of the corresponding lifted process $\widetilde{\mathcal{P}}_{z_l}$.
We compare the result of these manipulations with equation (\ref{correlation11.1.2.new}).
Taking into account that the lifted process $\widetilde{\mathcal{P}}_{z_l}$ is the Whittaker point process $\mathcal{P}_{z_l}^{\Whittaker}$, we get formula (\ref{MainFormulaCorF}).
\end{proof}

\end{document}